\documentclass{amsart}
\usepackage{amssymb}
\usepackage{enumerate}
\usepackage{graphicx}

\def\cl#1{{#1}^{\rm cl}}
\def\AO{{\mathcal  A}}
\def\sact{*}
\def\np{\bigskip\noindent}
\def\nl{\smallskip\noindent}
\def\disregard#1{}

\def\op{{\rm op}}
\def\GL{{\rm GL}}
\def\Supp{{\rm Supp}}
\def\M{{\rm M}}
\def\A{{\rm A}}
\def\det{{\rm det}}
\def\C{{\rm {\bf C}}}
\def\H{{\rm {\bf H}}}
\def\TL{{\rm {\bf TL}}}
\def\TB{{\rm {\bf TB}}}
\def\KT{{\rm {\bf KT}}}
\def\BMW{{\rm {\bf B}}}
\def\Br{{\rm {\bf Br}}}
\def\BrD{{\rm{\bf BrD}}}
\def\BrM{{\rm {\bf BrM}}}
\def\BrT{{\rm {\bf BrT}}}
\def\Ar{{\rm {\bf A}}}
\def\ADE{{\rm ADE}}
\def\D{{\rm D}}
\def\E{{\rm E}}
\def\B{{\rm B}}
\def\W{{\rm {\bf W}}}
\def\Sym{{\rm \Sigma}}
\def\hgt{{\rm ht}}
\def\np{\medskip}
\def\nl{\smallskip\noindent}
\def\isom{\cong}
\def\End{{\rm End}}
\def\Geod{{\rm Geod}}
\def\Proj{{\rm Proj}}
\def\CT{{\rm {\bf CT}}}
\def\spe{extremal}
\def\Z{{\mathbb Z}}
\def\CCC{{\mathbb C}}
\def\Q{{\mathbb Q}}
\def\R{{\mathbb R}}
\def\Co{{\mathcal C}}
\def\Do{{\mathcal D}}
\def\U{{\mathcal U}}
\def\B{{\mathcal B}}
\def\Mon{{\mathcal Mon}}

\newtheorem{Thm}{\bf{Theorem}}[section]
\newtheorem{Def}[Thm]{\bf{Definition}}

\newtheorem{Lm}[Thm]{\bf{Lemma}}

\newtheorem{Prop}[Thm]{\bf{Proposition}}

\newtheorem{Remark}[Thm]{\bf{Remark}}
\newtheorem{Remarks}[Thm]{\bf{Remarks}}

\renewcommand{\phi}{\varphi}

\DeclareMathOperator{\Aut}{Aut}
\DeclareMathOperator{\tr}{tr}
\DeclareMathOperator{\het}{ht}

\newcommand{\erz} [1] {\mbox{$\langle #1 \rangle$}}
\newcommand{\ov}{\overline}
\newcommand{\sli}{{\mathfrak{sl}}}
\newcommand{\F}{{\mathbb{F}}}
\newcommand{\spa}{\operatorname{span}}
\newcommand{\K}{{\mathbb{K}}}
\newcommand{\N}{{\mathbb{N}}}
\newcommand{\Rad}{\operatorname{Rad}}
\newcommand{\NilRad}{\operatorname{NilRad}}
\newcommand{\SanRad}{\operatorname{SanRad}}
\newcommand{\ch}{\operatorname{char}}
\newcommand{\eps}{\varepsilon}
\newcommand{\im}{{\rm Im}\,}
\newcommand{\adj}{\sim}
\newcommand{\wh}{\widehat}

\newcommand{\g}{{\mathfrak{g}}}
\newcommand{\nk}{{\mathfrak{n}}}
\newcommand{\RR}{{\mathbb{R}}}
\newcommand{\NN}{{\mathbb{N}}}
\newcommand{\ZZ}{{\mathbb{Z}}}
\newcommand{\Ec}{{\mathcal{E}}}
\newcommand{\Dc}{{\mathcal{D}}}
\newcommand{\Lc}{{\mathcal{L}}}
\newcommand{\Rc}{{\mathcal{R}}}

\def\a{\alpha}
\def\b{\beta}
\def\c{\gamma}
\def\e{\epsilon}
\def\s{\sigma}
\def\t{\tau}
\def\gam{\gamma}
\def\alp{\alpha}
\def\la{\langle}
\def\ra{\rangle}
\newcommand{\ident}{1}

\author{Arjeh M. Cohen
\& Di\'{e} A.H. Gijsbers
\& David B. Wales}
\address{Arjeh M. Cohen\\
Department of Mathematics and Computer Science\\
Eindhoven University of Technology\\
POBox 513\\
5600 MB Eindhoven\\
The Netherlands}
\email{A.M.Cohen@tue.nl}
\address{Di\'{e} A.H. Gijsbers\\
}
\email{dahgijsbers@gmail.com}
\address{David B. Wales\\
Mathematics Department\\
Sloan Lab\\
Caltech\\
Pasadena, CA 91125\\
USA}
\email{dbw@its.caltech.edu}

\title{Tangle and Brauer Diagram Algebras of Type $\D_n$}

\date{\today}

\begin{document}

\begin{abstract}
A generalization of the Kauffman tangle algebra is given for Coxeter type
$\D_n$.  The tangles involve a pole of order $2$.  The algebra is shown to be
isomorphic to the Birman-Murakami-Wenzl algebra of the same type.  This result
extends the isomorphism between the two algebras in the classical case, which,
in our set-up, occurs when the Coxeter type is $\A_{n-1}$.  The proof involves
a diagrammatic version of the Brauer algebra of type $\D_n$ of which the
generalized Temperley-Lieb algebra of type $\D_n$ is a subalgebra.

\bigskip\noindent
{\sc keywords:}
associative algebra, BMW algebra, Brauer algebra, Temperley-Lieb algebra,
tangle, Brauer diagram, Coxeter groups

\end{abstract}

\maketitle

\section{Introduction}
\label{sec:intro}

In \cite{MorWas}, Morton and Wasserman described an explicit isomorphism
between the Birman-Murakami-Wenzl (BMW) algebra $\BMW(\A_{n-1})$ of type
$\A_{n-1}$, which is given by means of a presentation by generators and
relations, and the Kauffman tangle algebra $\KT(\A_{n-1})$ connected to braids
on $n$ strands. In this paper we introduce a tangle algebra $\KT(\D_{n}) $
with a pole of order two and show that it is isomorphic to the BMW algebra
$\BMW(\D_{n})$ of type $\D_n$.  This construction extends the one of
\cite{MorWas} by a pole of order two. In \cite{All}, Allcock had a similar
pole involving $\D_n$.  In \cite{Har}, R.~H\"aring-Oldenburg deals with the
case of type ${\rm B}_n$ but uses further relations.

We also construct a Brauer diagram algebra $\BrD(\D_{n})$ of type $\D_n$.
This algebra is constructed with a basis of diagrams, much like the diagrams
Brauer used in \cite{Bra}.  The algebra extends Green's \cite{Gre}
diagrammatic description of the Temperley-Lieb algebra of type $\D_n$.

The tangle algebra $\KT(\D_n)$ is introduced in Definition \ref{df:KTDn}
below.  The BMW algebras $\BMW(M)$ were defined for arbitrary graphs $M$ in
\cite{CGW}.  Here, in Definition \ref{BMW-def} below, we introduce an integral
version of these over the domain $R = \Z[\delta^{\pm1},l^{\pm1},m]/((1-\delta
)m -l + l^{-1})$, where $l,m,\delta$ are indeterminates.  The Brauer diagram
algebra $\BrD(\D_n)$ will be defined over the quotient ring $\ov R =
R/(l-1,m)\isom \Z[\delta^{\pm1}]$, see Definition \ref{df:BrD}. As in
the BMW algebra case, for each graph $M$, a Brauer algebra $\Br(M)$ over $\ov
R$ has been defined by generators and relations, see \cite{CFW}.  As described
in \cite{CGWBMW}, modding out $l-1$ and $m$ gives a surjective $R$-equivariant
homomorphism $\mu:\BMW(M)\to \Br(M)$. Our main results can be summarized
as follows, where $n!! = 1 \cdot 3 \cdot 5 \cdots (2n-3)(2n-1)$ and
$d(n) = (2^{n}+1)n!!-(2^{n-1}+1)n!$.

\begin{Thm}\label{th:main}
The algebras mentioned have the following properties for $n\ge2$.
\begin{enumerate}[(i)]
\item % i 
\label{phisurji}
There is a surjective $R$-equivariant homomorphism
$\psi : \KT(\D_n)\to \BrD(\D_n)$.
\item % ii 
\label{rankKTDn}
There is a $\Z[\delta^{\pm1}]$-algebra isomorphism
$\nu: \Br(\D_n)\to \BrD(\D_n)$. Both algebras are free of dimension
$d(n)$.
\item %iii 
\label{BMW2KTiso}
There is an $R$-algebra isomorphism $\phi: \BMW(\D_n) \to \KT(\D_n)$.  Both
algebras are free of dimension $d(n)$.
\item The diagram below of $R$-equivariant homomorphisms
is commutative.
\end{enumerate}
\begin{center}
\begin{tabular}{ccccccc}
$\BMW(\D_n)$&$ \overset{\displaystyle \mu}{\longrightarrow}$&$\Br(\D_n)$\\
$\phi\ \downarrow$&&$\downarrow\ \nu$\\
$\KT(\D_n)$&$\overset{\displaystyle \psi}{\longrightarrow}$&$\BrD(\D_n)$
\end{tabular}
\end{center}
\end{Thm}

Here, the Coxeter type $\D_n$ is understood to
be $\A_1\A_1$ if $n=2$ and $\A_3$ if $n=3$. 

In \cite{Goo}, Goodman and Hauschild gave a similar construction for
affine BMW algebras and affine Kauffman tangle algebras, with a pole appearing
in the tangles.  Here we also define an $(n,n)$ tangle algebra, denoted
$\KT(\D_n)^{(1)}$, with respect to a pole of order two, cf.\ Definition
\ref{ch5main}.  Our tangle algebra differs from the algebras introduced in
\cite{Goo} in that some of our tangles with twists around the pole cannot be
simplified whereas the tangles of \cite{Goo} can.

The work for $\D_n$ encompasses $\A_{n-1}$.  The precise definition of the
classical tangle algebra $\KT(\A_{n-1})$ can be obtained from
Definition~\ref{ch5main} below after removing the pole and ignoring all
relations connected to it, and so $\KT(\A_{n-1})$ is a subalgebra of
$\KT(\D_n)$. By restriction of $\phi$ we find an isomorphism between the
classical BMW algebra, $\BMW(\A_{n-1})$, and $\KT(\A_{n-1})$.  This gives an
alternative proof to the one by Morton and Wasserman in \cite{MorWas},
cf.~Remark \ref{rmks:AD}(ii) below.

This paper is organized as follows.  In Section \ref{sec:tangleAlgebras}, we
introduce the tangle algebras $\KT(\D_n)^{(1)}$ and $\KT(\D_n)$, recall the
BMW algebra $\BMW(\D_n)$, and exhibit the homomorphism $\phi : \BMW(\D_n)\to
\KT(\D_n)$.  The presentation by generators and relations of $\BMW(\D_n)$
gives rise to a natural homomorphism $\phi: \BMW(\D_n)\to \KT(\D_n)$ of
$R$-algebras.  In Section \ref{sec:totallyDescTangles}, we introduce totally
descending tangles for which Reidemeister moves can be made.  We also find a
standard expression for tangles in terms of closed strands and twists around
the pole.  This will enable us to prove that $\phi$ is surjective, see
Theorem~\ref{th:surjhomo}. By \cite[Theorem 1.1]{CGWBMW} the dimension of
$\KT(\D_n)$ is at most $d(n)$.  In Section \ref{sec:BrauerTypeD} we deal with
the Brauer diagram algebra of type $\D_n$. Knowledge from \cite{CGW} helps us
to identify $\BrD(\D_n)$ with the Brauer algebra $\Br(\D_n)$, see Proposition
\ref{prop:AfterBMW}. This takes care of Theorem \ref{th:main}(ii).  There is a
surjective homomorphism $\psi$ of rings from the tangle algebra $\KT(\D_{n}) $
onto the Brauer diagram algebra $\BrD(\D_{n})$, see Proposition
\ref{homomKTDn2CDn}; this establishes Theorem \ref{th:main}(i) and helps us
find a lower bound for the dimension of $\KT(\D_n)$.  These facts are used in
the isomorphism proof of $\BMW(\D_n)$ and $\KT(\D_n)$ in Theorem
\ref{th:MainIso}, which settles Theorem \ref{th:main} (iii) and (iv).

We finish by discussing a slightly larger tangle algebra, $\KT(\D_{n})^{(2)}$,
for which we also provide a presentation by means of generators and relations.

The work reported here grew out of the Ph.~D.~thesis of one of us,
\cite{DAHG}. The other two authors wish to acknowledge Caltech and Technische
Universiteit Eindhoven for enabling mutual visits.

\section{Tangle algebras of type $\A_{n-1}$ and $\D_n$}
\label{sec:tangleAlgebras}
Let $M$ be a Coxeter diagram of rank $n$ without multiple bonds. We define
the BMW algebra by means of $2n$ generators and eleven kinds of relations.
For each node $i$ of the diagram $M$ we define two generators $g_i$ and $e_i$
with $i = 1,\ldots,n$. If two nodes are connected in the diagram we write $i
\sim j$, with $i,j$ the indices of the two nodes, and if they are not
connected we write $i \not \sim j$.  In this paper we will only be needing
$M$ of type $\A_{n-1}$ and $\D_n$.

\begin{Def}\label{BMW-def}
\rm
Let $M$ be a Coxeter diagram of rank $n$ without multiple bonds. The {\rm BMW} algebra of
type $M$ is the algebra, denoted by $\BMW(M)$, with unit element, over $R$,
whose presentation is given on generators $g_i$ and
$e_i$ ($i=1,\ldots,n$) by the following defining relations.

\begin{center}
\begin{tabular}{lcllr}
(B1)&\qquad& $g_ig_j=g_jg_i$& when &$i \not \sim j$,\\
(B2)&\qquad& $g_ig_jg_i=g_jg_ig_j$& when &$i\sim j$,\\
(D1)&\qquad& $me_i=l(g_i^2+mg_i-1)$& for all &$i$,\\
(R1)&\qquad& $g_ie_i=l^{-1}e_i$& for all &$i$,\\
(R2)&\qquad& $e_ig_je_i=le_i$& when &$i \sim j$,\\
(RSer)&\qquad& $e_ig_i=l^{-1}e_i$& for all &$i$,\\
(HSee)\label{e-sQ-eQ}&\quad& $e_i^2=\delta e_i$& for all &$i$,\\
(HCer)&\quad& $e_ig_j= g_je_i$& when &$i \not \sim j$,\\
(HCee)&\quad& $e_ie_j= e_je_i$& when &$i \not \sim j$,\\
(RNrre)&\quad& $g_jg_ie_j=e_ie_j$& when &$i\sim j$,\\
(RNerr)&\quad& $e_ig_jg_i=e_ie_j$& when &$i\sim j$.
\end{tabular}
\end{center}
\end{Def}

\np The first two relations are the braid relations commonly
associated with the Coxeter diagram $M$.
Just as for
Artin and Coxeter groups, if $M$ is the disjoint union of two diagrams $M_1$
and $M_2$, then $\BMW$ is the direct sum of the two BMW algebras $\BMW(M_1)$ and
$\BMW(M_2)$. For the solution of many problems concerning $\BMW$, this gives an easy reduction to
the case of connected diagrams $M$.

If $S$ is a ring containing $R$ in which $m$ is invertible, only the first
five relations are needed as defining relations for
$\BMW(M)\otimes_R S$; this is shown in \cite{CGW}.
It also follows from arguments of \cite{CGW} that the $g_i$ are invertible
elements in $\BMW(M)$, so that there is a group homomorphism from the Artin
group $A$ of type $M$ to the group $\BMW(M)^\times$ of invertible elements
of $\BMW(M)$ sending the $i$-th generator $s_i$ of $A$ to $g_i$.  The fact
that the BMW algebras of type $\A_{n-1}$ coincide with those defined by
Birman and Wenzl in \cite{BirWen} or by Murakami in \cite{Mur} is given in
\cite[Theorem 2.7]{CGW}.

The other kind of algebras to be introduced are tangle algebras over $R$.
We first recall from
\cite{MorWas} the definition of a tangle as a piece of a link diagram in the
plane.  
A $(k,n)$-tangle is a piece of a knot diagram in $ \R^2 \times [0,1]$,
consisting of piece-wise linear curves, called {\em strands}, such that every
strand intersects the boundary of $\R^2 \times [0,1]$ transversally in either
none or two of the points from $K = \{(1,0,1),\ldots,(k,0,1)\}\cup
\{1,0,0),\ldots,(n,0,0)\}$ and such that $K$ is the set of endpoints of
strands.  The elements of $K$ are called the {\em endpoints} of the tangle.  A
crossing of two strands is called {\em positive} if the strand moving from top
right to bottom left crosses over the other strand; the opposite crossing will
be called {\em negative}.

\begin{figure}[htbp]
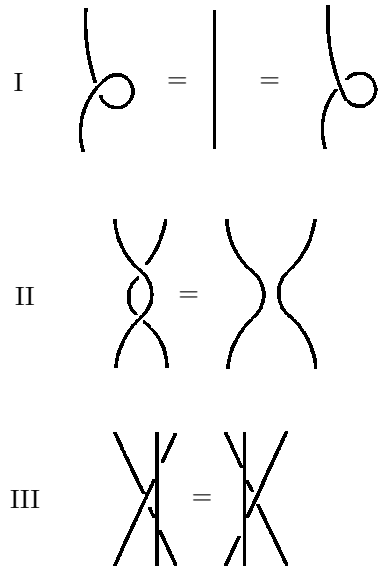


\caption{Reidemeister moves I, II, and III} \label{pic:reidem}
\end{figure}

Two tangles are ambient isotopic if they are related by a sequence of
Reidemeister moves I, II, and III (see Figure~\ref{pic:reidem}) together with
isotopies of $\R^2\times [0,1]$ fixing the boundary.  It is well known that
the closures of two tangles represent the same knot up to isotopy if and only
if they are ambient isotopic. Here we will restrict attention to regular
isotopy, cf. \cite{MorWas}.

\begin{Def}
\rm
Two tangles
are said to be {\em regularly isotopic} if they are related by a
sequence of only Reidemeister moves II and III together with
isotopies of $\R^2\times [0,1]$ fixing the boundary.
\end{Def}

A $(k,n)$-tangle and an $(n,s)$-tangle can be composed by placing the first
tangle on top of the second and connecting the endpoints at the bottom of the
first tangle to the endpoints at the top of the second and using an
isotopy to move the set of endpoints to their standard positions.

In this paper a new set of tangles is introduced which we will call {\em
tangles of type $\D$}. These tangles have an additional strand, called a {\em
pole}, with properties different from the other, ordinary strands in the
tangle. We start with a general definition of a tangle with a pole.  The pole
will be a vertical axis, which is to the left of $K$ through $(0,0,0)$.

\begin{figure}[htbp]
\unitlength 0.75mm
\begin{picture}(20.07,49.53)(0,0)
\linethickness{1mm}
\put(10,17){\line(0,1){23}}
\linethickness{0.3mm}
\multiput(6.41,29.74)(0.15,0.07){3}{\line(1,0){0.15}}
\multiput(5.95,29.51)(0.15,0.08){3}{\line(1,0){0.15}}
\multiput(5.52,29.25)(0.15,0.09){3}{\line(1,0){0.15}}
\multiput(5.1,28.96)(0.1,0.07){4}{\line(1,0){0.1}}
\multiput(4.7,28.65)(0.1,0.08){4}{\line(1,0){0.1}}
\multiput(4.32,28.31)(0.09,0.08){4}{\line(1,0){0.09}}
\multiput(3.97,27.96)(0.09,0.09){4}{\line(0,1){0.09}}
\multiput(3.63,27.57)(0.08,0.1){4}{\line(0,1){0.1}}
\multiput(3.33,27.17)(0.08,0.1){4}{\line(0,1){0.1}}
\multiput(3.04,26.75)(0.07,0.11){4}{\line(0,1){0.11}}
\multiput(2.79,26.32)(0.09,0.15){3}{\line(0,1){0.15}}
\multiput(2.56,25.86)(0.08,0.15){3}{\line(0,1){0.15}}
\multiput(2.36,25.4)(0.1,0.23){2}{\line(0,1){0.23}}
\multiput(2.19,24.92)(0.08,0.24){2}{\line(0,1){0.24}}
\multiput(2.06,24.43)(0.07,0.24){2}{\line(0,1){0.24}}
\multiput(1.95,23.94)(0.11,0.49){1}{\line(0,1){0.49}}
\multiput(1.87,23.44)(0.08,0.5){1}{\line(0,1){0.5}}
\multiput(1.83,22.93)(0.04,0.5){1}{\line(0,1){0.5}}
\multiput(1.82,22.43)(0.01,0.51){1}{\line(0,1){0.51}}
\multiput(1.82,22.43)(0.02,-0.51){1}{\line(0,-1){0.51}}
\multiput(1.84,21.92)(0.05,-0.5){1}{\line(0,-1){0.5}}
\multiput(1.89,21.42)(0.08,-0.5){1}{\line(0,-1){0.5}}
\multiput(1.98,20.92)(0.12,-0.49){1}{\line(0,-1){0.49}}
\multiput(2.09,20.42)(0.07,-0.24){2}{\line(0,-1){0.24}}
\multiput(2.24,19.94)(0.09,-0.24){2}{\line(0,-1){0.24}}
\multiput(2.41,19.46)(0.07,-0.15){3}{\line(0,-1){0.15}}
\multiput(2.62,19)(0.08,-0.15){3}{\line(0,-1){0.15}}
\multiput(2.86,18.55)(0.09,-0.14){3}{\line(0,-1){0.14}}
\multiput(3.12,18.12)(0.07,-0.1){4}{\line(0,-1){0.1}}
\multiput(3.41,17.71)(0.08,-0.1){4}{\line(0,-1){0.1}}
\multiput(3.72,17.31)(0.08,-0.09){4}{\line(0,-1){0.09}}
\multiput(4.06,16.93)(0.09,-0.09){4}{\line(1,0){0.09}}
\multiput(4.43,16.58)(0.1,-0.08){4}{\line(1,0){0.1}}
\multiput(4.81,16.25)(0.1,-0.08){4}{\line(1,0){0.1}}
\multiput(5.21,15.95)(0.14,-0.09){3}{\line(1,0){0.14}}
\multiput(5.64,15.67)(0.15,-0.08){3}{\line(1,0){0.15}}
\multiput(6.08,15.42)(0.15,-0.07){3}{\line(1,0){0.15}}
\multiput(6.53,15.19)(0.23,-0.1){2}{\line(1,0){0.23}}

\linethickness{0.3mm}
\multiput(13,32)(0.48,0.12){1}{\line(1,0){0.48}}
\multiput(13.48,32.12)(0.24,0.07){2}{\line(1,0){0.24}}
\multiput(13.96,32.26)(0.23,0.08){2}{\line(1,0){0.23}}
\multiput(14.42,32.43)(0.23,0.1){2}{\line(1,0){0.23}}
\multiput(14.88,32.62)(0.15,0.07){3}{\line(1,0){0.15}}
\multiput(15.32,32.84)(0.14,0.08){3}{\line(1,0){0.14}}
\multiput(15.75,33.09)(0.14,0.09){3}{\line(1,0){0.14}}
\multiput(16.17,33.36)(0.1,0.07){4}{\line(1,0){0.1}}
\multiput(16.57,33.65)(0.1,0.08){4}{\line(1,0){0.1}}
\multiput(16.95,33.96)(0.09,0.08){4}{\line(1,0){0.09}}
\multiput(17.32,34.29)(0.09,0.09){4}{\line(0,1){0.09}}
\multiput(17.67,34.65)(0.08,0.09){4}{\line(0,1){0.09}}
\multiput(17.99,35.02)(0.08,0.1){4}{\line(0,1){0.1}}
\multiput(18.3,35.41)(0.07,0.1){4}{\line(0,1){0.1}}
\multiput(18.58,35.82)(0.09,0.14){3}{\line(0,1){0.14}}
\multiput(18.84,36.24)(0.08,0.15){3}{\line(0,1){0.15}}
\multiput(19.08,36.68)(0.07,0.15){3}{\line(0,1){0.15}}
\multiput(19.29,37.13)(0.09,0.23){2}{\line(0,1){0.23}}
\multiput(19.47,37.59)(0.08,0.23){2}{\line(0,1){0.23}}
\multiput(19.63,38.06)(0.07,0.24){2}{\line(0,1){0.24}}
\multiput(19.76,38.53)(0.11,0.48){1}{\line(0,1){0.48}}
\multiput(19.87,39.02)(0.08,0.49){1}{\line(0,1){0.49}}
\multiput(19.95,39.51)(0.05,0.49){1}{\line(0,1){0.49}}

\linethickness{0.3mm}
\multiput(20,0)(0.04,0.5){1}{\line(0,1){0.5}}
\multiput(20.04,0.5)(0.02,0.5){1}{\line(0,1){0.5}}
\multiput(20.07,1.01)(0,0.51){1}{\line(0,1){0.51}}
\multiput(20.06,2.02)(0.01,-0.51){1}{\line(0,-1){0.51}}
\multiput(20.02,2.52)(0.03,-0.5){1}{\line(0,-1){0.5}}
\multiput(19.97,3.03)(0.05,-0.5){1}{\line(0,-1){0.5}}
\multiput(19.9,3.53)(0.07,-0.5){1}{\line(0,-1){0.5}}
\multiput(19.81,4.02)(0.09,-0.5){1}{\line(0,-1){0.5}}
\multiput(19.7,4.52)(0.11,-0.49){1}{\line(0,-1){0.49}}
\multiput(19.58,5.01)(0.06,-0.24){2}{\line(0,-1){0.24}}
\multiput(19.44,5.49)(0.07,-0.24){2}{\line(0,-1){0.24}}
\multiput(19.27,5.97)(0.08,-0.24){2}{\line(0,-1){0.24}}
\multiput(19.1,6.44)(0.09,-0.24){2}{\line(0,-1){0.24}}
\multiput(18.9,6.91)(0.1,-0.23){2}{\line(0,-1){0.23}}
\multiput(18.69,7.37)(0.07,-0.15){3}{\line(0,-1){0.15}}
\multiput(18.46,7.82)(0.08,-0.15){3}{\line(0,-1){0.15}}
\multiput(18.21,8.26)(0.08,-0.15){3}{\line(0,-1){0.15}}
\multiput(17.95,8.69)(0.09,-0.14){3}{\line(0,-1){0.14}}
\multiput(17.67,9.11)(0.09,-0.14){3}{\line(0,-1){0.14}}
\multiput(17.37,9.52)(0.07,-0.1){4}{\line(0,-1){0.1}}
\multiput(17.06,9.92)(0.08,-0.1){4}{\line(0,-1){0.1}}
\multiput(16.74,10.31)(0.08,-0.1){4}{\line(0,-1){0.1}}
\multiput(16.4,10.68)(0.08,-0.09){4}{\line(0,-1){0.09}}
\multiput(16.05,11.05)(0.09,-0.09){4}{\line(0,-1){0.09}}
\multiput(15.69,11.4)(0.09,-0.09){4}{\line(1,0){0.09}}
\multiput(15.31,11.73)(0.09,-0.08){4}{\line(1,0){0.09}}
\multiput(14.92,12.05)(0.1,-0.08){4}{\line(1,0){0.1}}
\multiput(14.52,12.36)(0.1,-0.08){4}{\line(1,0){0.1}}
\multiput(14.1,12.65)(0.1,-0.07){4}{\line(1,0){0.1}}
\multiput(13.68,12.93)(0.14,-0.09){3}{\line(1,0){0.14}}
\multiput(13.25,13.19)(0.14,-0.09){3}{\line(1,0){0.14}}
\multiput(12.81,13.43)(0.15,-0.08){3}{\line(1,0){0.15}}
\multiput(12.36,13.66)(0.15,-0.08){3}{\line(1,0){0.15}}
\multiput(11.9,13.87)(0.15,-0.07){3}{\line(1,0){0.15}}
\multiput(11.43,14.07)(0.23,-0.1){2}{\line(1,0){0.23}}
\multiput(10.96,14.24)(0.24,-0.09){2}{\line(1,0){0.24}}
\multiput(10.48,14.4)(0.24,-0.08){2}{\line(1,0){0.24}}
\multiput(9.99,14.54)(0.24,-0.07){2}{\line(1,0){0.24}}
\multiput(9.5,14.66)(0.25,-0.06){2}{\line(1,0){0.25}}
\multiput(9.01,14.77)(0.49,-0.1){1}{\line(1,0){0.49}}
\multiput(8.51,14.85)(0.5,-0.09){1}{\line(1,0){0.5}}
\multiput(8.01,14.92)(0.5,-0.07){1}{\line(1,0){0.5}}
\multiput(7.5,14.97)(0.5,-0.05){1}{\line(1,0){0.5}}
\multiput(7,15)(0.5,-0.03){1}{\line(1,0){0.5}}

\linethickness{1mm}
\put(10,0){\line(0,1){12}}
\put(15,26){\makebox(0,0)[cl]{}}

\end{picture}
$\qquad$ $\qquad$ $\qquad$ $\qquad$
\unitlength 0.4mm
\begin{picture}(50.62,82)(0,0)
\put(28.75,39.38){\makebox(0,0)[cc]{$=$}}

\linethickness{1mm}
\put(38.75,0.62){\line(0,1){78.75}}
\put(43.88,46){\makebox(0,0)[cl]{}}

\linethickness{0.3mm}
\put(50.62,0.62){\line(0,1){79.38}}
\linethickness{1mm}
\put(10,39.5){\line(0,1){23}}
\linethickness{1mm}
\put(10,67.5){\line(0,1){12}}
\linethickness{0.3mm}
\multiput(8,64.5)(0.49,0.08){1}{\line(1,0){0.49}}
\multiput(8.49,64.58)(0.49,0.1){1}{\line(1,0){0.49}}
\multiput(8.97,64.68)(0.48,0.11){1}{\line(1,0){0.48}}
\multiput(9.46,64.79)(0.48,0.13){1}{\line(1,0){0.48}}
\multiput(9.94,64.93)(0.47,0.15){1}{\line(1,0){0.47}}
\multiput(10.41,65.07)(0.47,0.16){1}{\line(1,0){0.47}}
\multiput(10.88,65.24)(0.46,0.18){1}{\line(1,0){0.46}}
\multiput(11.34,65.42)(0.23,0.1){2}{\line(1,0){0.23}}
\multiput(11.79,65.61)(0.22,0.11){2}{\line(1,0){0.22}}
\multiput(12.24,65.82)(0.22,0.11){2}{\line(1,0){0.22}}
\multiput(12.68,66.05)(0.22,0.12){2}{\line(1,0){0.22}}
\multiput(13.12,66.29)(0.21,0.13){2}{\line(1,0){0.21}}
\multiput(13.54,66.55)(0.21,0.14){2}{\line(1,0){0.21}}
\multiput(13.95,66.82)(0.2,0.14){2}{\line(1,0){0.2}}
\multiput(14.36,67.11)(0.2,0.15){2}{\line(1,0){0.2}}
\multiput(14.76,67.4)(0.13,0.1){3}{\line(1,0){0.13}}
\multiput(15.14,67.72)(0.12,0.11){3}{\line(1,0){0.12}}
\multiput(15.51,68.04)(0.12,0.11){3}{\line(1,0){0.12}}
\multiput(15.88,68.38)(0.12,0.12){3}{\line(1,0){0.12}}
\multiput(16.23,68.73)(0.11,0.12){3}{\line(0,1){0.12}}
\multiput(16.57,69.09)(0.11,0.12){3}{\line(0,1){0.12}}
\multiput(16.89,69.47)(0.1,0.13){3}{\line(0,1){0.13}}
\multiput(17.2,69.85)(0.15,0.2){2}{\line(0,1){0.2}}
\multiput(17.5,70.25)(0.14,0.2){2}{\line(0,1){0.2}}
\multiput(17.79,70.65)(0.14,0.21){2}{\line(0,1){0.21}}
\multiput(18.06,71.07)(0.13,0.21){2}{\line(0,1){0.21}}
\multiput(18.32,71.49)(0.12,0.22){2}{\line(0,1){0.22}}
\multiput(18.56,71.92)(0.11,0.22){2}{\line(0,1){0.22}}
\multiput(18.79,72.36)(0.11,0.22){2}{\line(0,1){0.22}}
\multiput(19,72.81)(0.1,0.23){2}{\line(0,1){0.23}}
\multiput(19.19,73.27)(0.09,0.23){2}{\line(0,1){0.23}}
\multiput(19.38,73.73)(0.16,0.47){1}{\line(0,1){0.47}}
\multiput(19.54,74.2)(0.15,0.47){1}{\line(0,1){0.47}}
\multiput(19.69,74.67)(0.13,0.48){1}{\line(0,1){0.48}}
\multiput(19.82,75.15)(0.12,0.48){1}{\line(0,1){0.48}}
\multiput(19.94,75.63)(0.1,0.49){1}{\line(0,1){0.49}}
\multiput(20.03,76.12)(0.08,0.49){1}{\line(0,1){0.49}}
\multiput(20.12,76.6)(0.06,0.49){1}{\line(0,1){0.49}}
\multiput(20.18,77.1)(0.05,0.49){1}{\line(0,1){0.49}}
\multiput(20.23,77.59)(0.03,0.49){1}{\line(0,1){0.49}}
\multiput(20.26,78.08)(0.01,0.5){1}{\line(0,1){0.5}}
\multiput(20.27,79.07)(0,-0.5){1}{\line(0,-1){0.5}}
\multiput(20.25,79.57)(0.02,-0.5){1}{\line(0,-1){0.5}}

\linethickness{0.3mm}
\multiput(7.54,64.32)(0.23,0.09){2}{\line(1,0){0.23}}
\multiput(7.09,64.11)(0.23,0.1){2}{\line(1,0){0.23}}
\multiput(6.65,63.88)(0.22,0.12){2}{\line(1,0){0.22}}
\multiput(6.23,63.62)(0.21,0.13){2}{\line(1,0){0.21}}
\multiput(5.82,63.33)(0.2,0.14){2}{\line(1,0){0.2}}
\multiput(5.43,63.03)(0.13,0.1){3}{\line(1,0){0.13}}
\multiput(5.06,62.7)(0.12,0.11){3}{\line(1,0){0.12}}
\multiput(4.71,62.35)(0.12,0.12){3}{\line(1,0){0.12}}
\multiput(4.38,61.98)(0.11,0.12){3}{\line(0,1){0.12}}
\multiput(4.07,61.59)(0.1,0.13){3}{\line(0,1){0.13}}
\multiput(3.78,61.19)(0.14,0.2){2}{\line(0,1){0.2}}
\multiput(3.52,60.76)(0.13,0.21){2}{\line(0,1){0.21}}
\multiput(3.28,60.33)(0.12,0.22){2}{\line(0,1){0.22}}
\multiput(3.07,59.88)(0.11,0.22){2}{\line(0,1){0.22}}
\multiput(2.89,59.42)(0.09,0.23){2}{\line(0,1){0.23}}
\multiput(2.73,58.95)(0.16,0.47){1}{\line(0,1){0.47}}
\multiput(2.6,58.47)(0.13,0.48){1}{\line(0,1){0.48}}
\multiput(2.5,57.98)(0.1,0.49){1}{\line(0,1){0.49}}
\multiput(2.43,57.49)(0.07,0.49){1}{\line(0,1){0.49}}
\multiput(2.39,57)(0.04,0.49){1}{\line(0,1){0.49}}
\multiput(2.37,56.5)(0.01,0.5){1}{\line(0,1){0.5}}
\multiput(2.37,56.5)(0.01,-0.5){1}{\line(0,-1){0.5}}
\multiput(2.39,56)(0.04,-0.49){1}{\line(0,-1){0.49}}
\multiput(2.43,55.51)(0.07,-0.49){1}{\line(0,-1){0.49}}
\multiput(2.5,55.02)(0.1,-0.49){1}{\line(0,-1){0.49}}
\multiput(2.6,54.53)(0.13,-0.48){1}{\line(0,-1){0.48}}
\multiput(2.73,54.05)(0.16,-0.47){1}{\line(0,-1){0.47}}
\multiput(2.89,53.58)(0.09,-0.23){2}{\line(0,-1){0.23}}
\multiput(3.07,53.12)(0.11,-0.22){2}{\line(0,-1){0.22}}
\multiput(3.28,52.67)(0.12,-0.22){2}{\line(0,-1){0.22}}
\multiput(3.52,52.24)(0.13,-0.21){2}{\line(0,-1){0.21}}
\multiput(3.78,51.81)(0.14,-0.2){2}{\line(0,-1){0.2}}
\multiput(4.07,51.41)(0.1,-0.13){3}{\line(0,-1){0.13}}
\multiput(4.38,51.02)(0.11,-0.12){3}{\line(0,-1){0.12}}
\multiput(4.71,50.65)(0.12,-0.12){3}{\line(1,0){0.12}}
\multiput(5.06,50.3)(0.12,-0.11){3}{\line(1,0){0.12}}
\multiput(5.43,49.97)(0.13,-0.1){3}{\line(1,0){0.13}}
\multiput(5.82,49.67)(0.2,-0.14){2}{\line(1,0){0.2}}
\multiput(6.23,49.38)(0.21,-0.13){2}{\line(1,0){0.21}}
\multiput(6.65,49.12)(0.22,-0.12){2}{\line(1,0){0.22}}
\multiput(7.09,48.89)(0.23,-0.1){2}{\line(1,0){0.23}}
\multiput(7.54,48.68)(0.23,-0.09){2}{\line(1,0){0.23}}

\linethickness{1mm}
\put(10,1){\line(0,1){23}}
\linethickness{1mm}
\put(10,29){\line(0,1){12}}
\linethickness{0.3mm}
\multiput(8,26)(0.5,0.02){1}{\line(1,0){0.5}}
\multiput(8.5,26.02)(0.5,0.04){1}{\line(1,0){0.5}}
\multiput(9,26.06)(0.5,0.06){1}{\line(1,0){0.5}}
\multiput(9.49,26.12)(0.49,0.08){1}{\line(1,0){0.49}}
\multiput(9.98,26.2)(0.49,0.1){1}{\line(1,0){0.49}}
\multiput(10.47,26.3)(0.48,0.13){1}{\line(1,0){0.48}}
\multiput(10.95,26.43)(0.48,0.15){1}{\line(1,0){0.48}}
\multiput(11.43,26.58)(0.47,0.17){1}{\line(1,0){0.47}}
\multiput(11.9,26.74)(0.23,0.09){2}{\line(1,0){0.23}}
\multiput(12.36,26.93)(0.23,0.1){2}{\line(1,0){0.23}}
\multiput(12.82,27.14)(0.22,0.11){2}{\line(1,0){0.22}}
\multiput(13.26,27.37)(0.22,0.12){2}{\line(1,0){0.22}}
\multiput(13.69,27.62)(0.21,0.13){2}{\line(1,0){0.21}}
\multiput(14.11,27.89)(0.2,0.14){2}{\line(1,0){0.2}}
\multiput(14.52,28.17)(0.13,0.1){3}{\line(1,0){0.13}}
\multiput(14.92,28.48)(0.13,0.11){3}{\line(1,0){0.13}}
\multiput(15.3,28.8)(0.12,0.11){3}{\line(1,0){0.12}}
\multiput(15.67,29.14)(0.12,0.12){3}{\line(0,1){0.12}}
\multiput(16.02,29.49)(0.11,0.12){3}{\line(0,1){0.12}}
\multiput(16.36,29.86)(0.11,0.13){3}{\line(0,1){0.13}}
\multiput(16.68,30.24)(0.1,0.13){3}{\line(0,1){0.13}}
\multiput(16.98,30.64)(0.14,0.21){2}{\line(0,1){0.21}}
\multiput(17.26,31.05)(0.13,0.21){2}{\line(0,1){0.21}}
\multiput(17.53,31.47)(0.12,0.22){2}{\line(0,1){0.22}}
\multiput(17.77,31.91)(0.11,0.22){2}{\line(0,1){0.22}}
\multiput(18,32.35)(0.1,0.23){2}{\line(0,1){0.23}}
\multiput(18.21,32.8)(0.09,0.23){2}{\line(0,1){0.23}}
\multiput(18.4,33.27)(0.17,0.47){1}{\line(0,1){0.47}}
\multiput(18.56,33.74)(0.15,0.48){1}{\line(0,1){0.48}}
\multiput(18.71,34.22)(0.12,0.48){1}{\line(0,1){0.48}}
\multiput(18.83,34.7)(0.1,0.49){1}{\line(0,1){0.49}}
\multiput(18.93,35.19)(0.08,0.49){1}{\line(0,1){0.49}}
\multiput(19.01,35.68)(0.06,0.5){1}{\line(0,1){0.5}}
\multiput(19.07,36.18)(0.04,0.5){1}{\line(0,1){0.5}}
\multiput(19.11,36.67)(0.01,0.5){1}{\line(0,1){0.5}}
\multiput(19.11,37.67)(0.01,-0.5){1}{\line(0,-1){0.5}}
\multiput(19.08,38.17)(0.03,-0.5){1}{\line(0,-1){0.5}}
\multiput(19.03,38.67)(0.05,-0.5){1}{\line(0,-1){0.5}}
\multiput(18.96,39.16)(0.07,-0.49){1}{\line(0,-1){0.49}}
\multiput(18.86,39.65)(0.1,-0.49){1}{\line(0,-1){0.49}}
\multiput(18.75,40.13)(0.12,-0.49){1}{\line(0,-1){0.49}}
\multiput(18.61,40.61)(0.14,-0.48){1}{\line(0,-1){0.48}}
\multiput(18.45,41.09)(0.16,-0.47){1}{\line(0,-1){0.47}}
\multiput(18.27,41.55)(0.09,-0.23){2}{\line(0,-1){0.23}}
\multiput(18.07,42.01)(0.1,-0.23){2}{\line(0,-1){0.23}}
\multiput(17.84,42.46)(0.11,-0.22){2}{\line(0,-1){0.22}}
\multiput(17.6,42.89)(0.12,-0.22){2}{\line(0,-1){0.22}}
\multiput(17.34,43.32)(0.13,-0.21){2}{\line(0,-1){0.21}}
\multiput(17.06,43.73)(0.14,-0.21){2}{\line(0,-1){0.21}}
\multiput(16.77,44.13)(0.15,-0.2){2}{\line(0,-1){0.2}}
\multiput(16.45,44.52)(0.1,-0.13){3}{\line(0,-1){0.13}}
\multiput(16.12,44.89)(0.11,-0.12){3}{\line(0,-1){0.12}}
\multiput(15.78,45.25)(0.12,-0.12){3}{\line(0,-1){0.12}}
\multiput(15.41,45.6)(0.12,-0.11){3}{\line(1,0){0.12}}
\multiput(15.03,45.92)(0.13,-0.11){3}{\line(1,0){0.13}}
\multiput(14.64,46.23)(0.13,-0.1){3}{\line(1,0){0.13}}
\multiput(14.24,46.52)(0.2,-0.15){2}{\line(1,0){0.2}}
\multiput(13.82,46.8)(0.21,-0.14){2}{\line(1,0){0.21}}
\multiput(13.39,47.05)(0.21,-0.13){2}{\line(1,0){0.21}}
\multiput(12.95,47.28)(0.22,-0.12){2}{\line(1,0){0.22}}
\multiput(12.5,47.5)(0.23,-0.11){2}{\line(1,0){0.23}}

\linethickness{0.3mm}
\multiput(7.54,25.82)(0.23,0.09){2}{\line(1,0){0.23}}
\multiput(7.09,25.61)(0.23,0.1){2}{\line(1,0){0.23}}
\multiput(6.65,25.38)(0.22,0.12){2}{\line(1,0){0.22}}
\multiput(6.23,25.12)(0.21,0.13){2}{\line(1,0){0.21}}
\multiput(5.82,24.83)(0.2,0.14){2}{\line(1,0){0.2}}
\multiput(5.43,24.53)(0.13,0.1){3}{\line(1,0){0.13}}
\multiput(5.06,24.2)(0.12,0.11){3}{\line(1,0){0.12}}
\multiput(4.71,23.85)(0.12,0.12){3}{\line(1,0){0.12}}
\multiput(4.38,23.48)(0.11,0.12){3}{\line(0,1){0.12}}
\multiput(4.07,23.09)(0.1,0.13){3}{\line(0,1){0.13}}
\multiput(3.78,22.69)(0.14,0.2){2}{\line(0,1){0.2}}
\multiput(3.52,22.26)(0.13,0.21){2}{\line(0,1){0.21}}
\multiput(3.28,21.83)(0.12,0.22){2}{\line(0,1){0.22}}
\multiput(3.07,21.38)(0.11,0.22){2}{\line(0,1){0.22}}
\multiput(2.89,20.92)(0.09,0.23){2}{\line(0,1){0.23}}
\multiput(2.73,20.45)(0.16,0.47){1}{\line(0,1){0.47}}
\multiput(2.6,19.97)(0.13,0.48){1}{\line(0,1){0.48}}
\multiput(2.5,19.48)(0.1,0.49){1}{\line(0,1){0.49}}
\multiput(2.43,18.99)(0.07,0.49){1}{\line(0,1){0.49}}
\multiput(2.39,18.5)(0.04,0.49){1}{\line(0,1){0.49}}
\multiput(2.37,18)(0.01,0.5){1}{\line(0,1){0.5}}
\multiput(2.37,18)(0.01,-0.5){1}{\line(0,-1){0.5}}
\multiput(2.39,17.5)(0.04,-0.49){1}{\line(0,-1){0.49}}
\multiput(2.43,17.01)(0.07,-0.49){1}{\line(0,-1){0.49}}
\multiput(2.5,16.52)(0.1,-0.49){1}{\line(0,-1){0.49}}
\multiput(2.6,16.03)(0.13,-0.48){1}{\line(0,-1){0.48}}
\multiput(2.73,15.55)(0.16,-0.47){1}{\line(0,-1){0.47}}
\multiput(2.89,15.08)(0.09,-0.23){2}{\line(0,-1){0.23}}
\multiput(3.07,14.62)(0.11,-0.22){2}{\line(0,-1){0.22}}
\multiput(3.28,14.17)(0.12,-0.22){2}{\line(0,-1){0.22}}
\multiput(3.52,13.74)(0.13,-0.21){2}{\line(0,-1){0.21}}
\multiput(3.78,13.31)(0.14,-0.2){2}{\line(0,-1){0.2}}
\multiput(4.07,12.91)(0.1,-0.13){3}{\line(0,-1){0.13}}
\multiput(4.38,12.52)(0.11,-0.12){3}{\line(0,-1){0.12}}
\multiput(4.71,12.15)(0.12,-0.12){3}{\line(1,0){0.12}}
\multiput(5.06,11.8)(0.12,-0.11){3}{\line(1,0){0.12}}
\multiput(5.43,11.47)(0.13,-0.1){3}{\line(1,0){0.13}}
\multiput(5.82,11.17)(0.2,-0.14){2}{\line(1,0){0.2}}
\multiput(6.23,10.88)(0.21,-0.13){2}{\line(1,0){0.21}}
\multiput(6.65,10.62)(0.22,-0.12){2}{\line(1,0){0.22}}
\multiput(7.09,10.39)(0.23,-0.1){2}{\line(1,0){0.23}}
\multiput(7.54,10.18)(0.23,-0.09){2}{\line(1,0){0.23}}

\linethickness{0.3mm}
\multiput(19.91,1.67)(0.05,-0.49){1}{\line(0,-1){0.49}}
\multiput(19.82,2.16)(0.08,-0.49){1}{\line(0,-1){0.49}}
\multiput(19.71,2.64)(0.11,-0.48){1}{\line(0,-1){0.48}}
\multiput(19.58,3.11)(0.14,-0.47){1}{\line(0,-1){0.47}}
\multiput(19.41,3.58)(0.16,-0.47){1}{\line(0,-1){0.47}}
\multiput(19.22,4.04)(0.09,-0.23){2}{\line(0,-1){0.23}}
\multiput(19.01,4.48)(0.11,-0.22){2}{\line(0,-1){0.22}}
\multiput(18.77,4.91)(0.12,-0.22){2}{\line(0,-1){0.22}}
\multiput(18.5,5.33)(0.13,-0.21){2}{\line(0,-1){0.21}}
\multiput(18.22,5.73)(0.14,-0.2){2}{\line(0,-1){0.2}}
\multiput(17.9,6.11)(0.1,-0.13){3}{\line(0,-1){0.13}}
\multiput(17.57,6.48)(0.11,-0.12){3}{\line(0,-1){0.12}}
\multiput(17.22,6.83)(0.12,-0.12){3}{\line(1,0){0.12}}
\multiput(16.85,7.15)(0.12,-0.11){3}{\line(1,0){0.12}}
\multiput(16.46,7.45)(0.13,-0.1){3}{\line(1,0){0.13}}
\multiput(16.05,7.74)(0.2,-0.14){2}{\line(1,0){0.2}}
\multiput(15.63,7.99)(0.21,-0.13){2}{\line(1,0){0.21}}
\multiput(15.2,8.23)(0.22,-0.12){2}{\line(1,0){0.22}}
\multiput(14.75,8.43)(0.22,-0.1){2}{\line(1,0){0.22}}
\multiput(14.29,8.62)(0.23,-0.09){2}{\line(1,0){0.23}}
\multiput(13.82,8.77)(0.47,-0.16){1}{\line(1,0){0.47}}
\multiput(13.34,8.9)(0.48,-0.13){1}{\line(1,0){0.48}}
\multiput(12.86,9)(0.48,-0.1){1}{\line(1,0){0.48}}
\end{picture}
\caption{A pole twist and the relation of a pole of order two} \label{ch5twist} \label{ch5doubletwist2}
\end{figure}
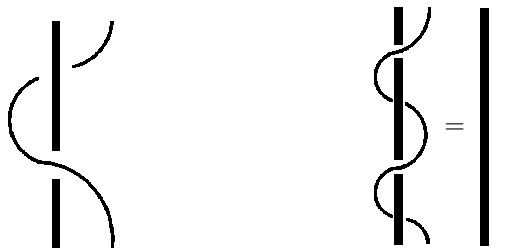

\begin{Def}
\rm
An {\em $(n,n)$-tangle with a pole} is an $(n,n)$-tangle which includes the
distinguished straight line segment in $\R^2 \times [0,1]$
connecting $(0,0,1)$ with $(0,0,0)$, called the {\em pole}.
\end{Def}

While we regard normal strands of a diagram as pieces of rope or rubber bands,
we can treat the pole as an iron pipe or bar. It is a fixed vertical strand
which cannot be deformed (or bent).  Because of this, Reidemeister I will never
occur for the pole.  Furthermore, we do not allow Reidemeister III where the
pole is one of the three strands.  Only Reidemeister II is allowed.  Here two
consecutive under or over crossings of one strand with the pole can be removed
leaving the pole intact.  It is here that we differ from \cite{Goo} where
Reidemeister III is allowed.

\np The (one time) encircling of the pole by a strand of the tangle is called
a {\em twist around the pole} or simply {\em pole twist}. See
Figure~\ref{ch5twist}, where the pole is depicted as a bold vertical strand.
For our purposes in $\D_n$ we will use a pole of order two:

\begin{Def}
\rm The pole is said to have order two if two consecutive twists around the
pole can be removed in that the resulting strand starts and finishes in the
same place but no longer goes around the pole.  Here consecutive means the
second twist follows immediately after the first twist with no other strands
between them.  See Figure~\ref{ch5doubletwist2}.
\end{Def}

\np We now define tangles of type $\A$, $\D^{(1)}$, $\D^{(2)}$, and $\D$.

\begin{Def}\label{df:tangletype}
\rm An $(n,n)$-tangle of type $\A$ is an $(n,n)$-tangle with no strands going
around the pole.  An $(n,n)$-tangle of type $\D^{(1)}$ is an $(n,n)$-tangle
with a pole of order two.  If there are an even number of pole twists, it is
called of type $\D^{(2)}$.  A tangle of type $\D^{(2)}$ is said to be of type
$\D$ if it has a horizontal strand whenever it has a closed strand twisting
around the pole.
\end{Def}

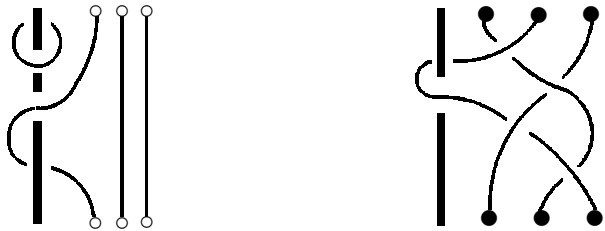
\begin{figure}[hbtp]
\unitlength 0.75mm
\begin{picture}(26.36,45.28)(0,0)
\linethickness{0.3mm}
\multiput(3.08,36.05)(0.09,0.08){4}{\line(1,0){0.09}}
\multiput(2.75,35.69)(0.08,0.09){4}{\line(0,1){0.09}}
\multiput(2.46,35.29)(0.07,0.1){4}{\line(0,1){0.1}}
\multiput(2.22,34.86)(0.08,0.14){3}{\line(0,1){0.14}}
\multiput(2.03,34.4)(0.09,0.23){2}{\line(0,1){0.23}}
\multiput(1.89,33.92)(0.07,0.24){2}{\line(0,1){0.24}}
\multiput(1.82,33.43)(0.08,0.49){1}{\line(0,1){0.49}}
\multiput(1.8,32.94)(0.02,0.49){1}{\line(0,1){0.49}}
\multiput(1.8,32.94)(0.04,-0.49){1}{\line(0,-1){0.49}}
\multiput(1.84,32.45)(0.1,-0.48){1}{\line(0,-1){0.48}}
\multiput(1.94,31.96)(0.08,-0.23){2}{\line(0,-1){0.23}}
\multiput(2.1,31.5)(0.07,-0.15){3}{\line(0,-1){0.15}}
\multiput(2.31,31.05)(0.09,-0.14){3}{\line(0,-1){0.14}}
\multiput(2.57,30.63)(0.08,-0.1){4}{\line(0,-1){0.1}}
\multiput(2.88,30.25)(0.09,-0.09){4}{\line(1,0){0.09}}
\multiput(3.24,29.9)(0.1,-0.08){4}{\line(1,0){0.1}}
\multiput(3.63,29.6)(0.14,-0.08){3}{\line(1,0){0.14}}
\multiput(4.06,29.35)(0.23,-0.1){2}{\line(1,0){0.23}}
\multiput(4.51,29.15)(0.24,-0.07){2}{\line(1,0){0.24}}
\multiput(4.98,29)(0.49,-0.09){1}{\line(1,0){0.49}}
\multiput(5.47,28.92)(0.49,-0.03){1}{\line(1,0){0.49}}
\multiput(5.96,28.89)(0.49,0.03){1}{\line(1,0){0.49}}
\multiput(6.45,28.92)(0.49,0.09){1}{\line(1,0){0.49}}
\multiput(6.94,29)(0.24,0.07){2}{\line(1,0){0.24}}
\multiput(7.41,29.15)(0.23,0.1){2}{\line(1,0){0.23}}
\multiput(7.86,29.35)(0.14,0.08){3}{\line(1,0){0.14}}
\multiput(8.29,29.6)(0.1,0.08){4}{\line(1,0){0.1}}
\multiput(8.68,29.9)(0.09,0.09){4}{\line(1,0){0.09}}
\multiput(9.04,30.25)(0.08,0.1){4}{\line(0,1){0.1}}
\multiput(9.35,30.63)(0.09,0.14){3}{\line(0,1){0.14}}
\multiput(9.61,31.05)(0.07,0.15){3}{\line(0,1){0.15}}
\multiput(9.82,31.5)(0.08,0.23){2}{\line(0,1){0.23}}
\multiput(9.98,31.96)(0.1,0.48){1}{\line(0,1){0.48}}
\multiput(10.08,32.45)(0.04,0.49){1}{\line(0,1){0.49}}
\multiput(10.1,33.43)(0.02,-0.49){1}{\line(0,-1){0.49}}
\multiput(10.03,33.92)(0.08,-0.49){1}{\line(0,-1){0.49}}
\multiput(9.89,34.4)(0.07,-0.24){2}{\line(0,-1){0.24}}
\multiput(9.7,34.86)(0.09,-0.23){2}{\line(0,-1){0.23}}
\multiput(9.46,35.29)(0.08,-0.14){3}{\line(0,-1){0.14}}
\multiput(9.17,35.69)(0.07,-0.1){4}{\line(0,-1){0.1}}
\multiput(8.84,36.05)(0.08,-0.09){4}{\line(0,-1){0.09}}
\multiput(8.46,36.37)(0.09,-0.08){4}{\line(1,0){0.09}}

\linethickness{1mm}
\put(6,1){\line(0,1){18}}
\linethickness{0.3mm}
\multiput(5.78,21.4)(0.5,-0.01){1}{\line(1,0){0.5}}
\multiput(6.28,21.39)(0.5,0.02){1}{\line(1,0){0.5}}
\multiput(6.77,21.4)(0.49,0.05){1}{\line(1,0){0.49}}
\multiput(7.27,21.46)(0.49,0.09){1}{\line(1,0){0.49}}
\multiput(7.76,21.55)(0.24,0.06){2}{\line(1,0){0.24}}
\multiput(8.24,21.67)(0.24,0.08){2}{\line(1,0){0.24}}
\multiput(8.71,21.82)(0.23,0.09){2}{\line(1,0){0.23}}
\multiput(9.17,22.01)(0.15,0.07){3}{\line(1,0){0.15}}
\multiput(9.62,22.22)(0.14,0.08){3}{\line(1,0){0.14}}
\multiput(10.05,22.47)(0.14,0.09){3}{\line(1,0){0.14}}
\multiput(10.46,22.75)(0.1,0.08){4}{\line(1,0){0.1}}
\multiput(10.86,23.05)(0.09,0.08){4}{\line(1,0){0.09}}
\multiput(11.23,23.39)(0.09,0.09){4}{\line(0,1){0.09}}
\multiput(11.57,23.74)(0.08,0.09){4}{\line(0,1){0.09}}
\multiput(11.9,24.12)(0.07,0.1){4}{\line(0,1){0.1}}
\multiput(12.19,24.52)(0.09,0.14){3}{\line(0,1){0.14}}
\multiput(12.46,24.94)(0.08,0.15){3}{\line(0,1){0.15}}
\multiput(12.69,25.38)(0.07,0.15){3}{\line(0,1){0.15}}

\linethickness{0.3mm}
\multiput(15.95,2.61)(0.05,-0.49){1}{\line(0,-1){0.49}}
\multiput(15.87,3.09)(0.08,-0.48){1}{\line(0,-1){0.48}}
\multiput(15.76,3.57)(0.1,-0.48){1}{\line(0,-1){0.48}}
\multiput(15.64,4.04)(0.06,-0.24){2}{\line(0,-1){0.24}}
\multiput(15.48,4.51)(0.08,-0.23){2}{\line(0,-1){0.23}}
\multiput(15.31,4.97)(0.09,-0.23){2}{\line(0,-1){0.23}}
\multiput(15.11,5.42)(0.1,-0.22){2}{\line(0,-1){0.22}}
\multiput(14.89,5.86)(0.07,-0.15){3}{\line(0,-1){0.15}}
\multiput(14.65,6.28)(0.08,-0.14){3}{\line(0,-1){0.14}}
\multiput(14.39,6.7)(0.09,-0.14){3}{\line(0,-1){0.14}}
\multiput(14.1,7.1)(0.07,-0.1){4}{\line(0,-1){0.1}}
\multiput(13.8,7.48)(0.08,-0.1){4}{\line(0,-1){0.1}}
\multiput(13.48,7.85)(0.08,-0.09){4}{\line(0,-1){0.09}}
\multiput(13.14,8.2)(0.09,-0.09){4}{\line(0,-1){0.09}}
\multiput(12.78,8.54)(0.09,-0.08){4}{\line(1,0){0.09}}
\multiput(12.4,8.85)(0.09,-0.08){4}{\line(1,0){0.09}}
\multiput(12.01,9.15)(0.1,-0.07){4}{\line(1,0){0.1}}
\multiput(11.61,9.43)(0.13,-0.09){3}{\line(1,0){0.13}}
\multiput(11.19,9.68)(0.14,-0.09){3}{\line(1,0){0.14}}
\multiput(10.76,9.92)(0.14,-0.08){3}{\line(1,0){0.14}}
\multiput(10.32,10.13)(0.15,-0.07){3}{\line(1,0){0.15}}
\multiput(9.87,10.32)(0.23,-0.09){2}{\line(1,0){0.23}}
\multiput(9.4,10.49)(0.23,-0.08){2}{\line(1,0){0.23}}
\multiput(8.94,10.63)(0.23,-0.07){2}{\line(1,0){0.23}}
\multiput(8.46,10.75)(0.48,-0.12){1}{\line(1,0){0.48}}

\linethickness{0.3mm}
\multiput(4.97,21.3)(0.5,0.07){1}{\line(1,0){0.5}}
\multiput(4.47,21.17)(0.25,0.06){2}{\line(1,0){0.25}}
\multiput(3.99,20.99)(0.24,0.09){2}{\line(1,0){0.24}}
\multiput(3.53,20.77)(0.15,0.07){3}{\line(1,0){0.15}}
\multiput(3.1,20.5)(0.14,0.09){3}{\line(1,0){0.14}}
\multiput(2.7,20.19)(0.1,0.08){4}{\line(1,0){0.1}}
\multiput(2.33,19.84)(0.09,0.09){4}{\line(1,0){0.09}}
\multiput(1.99,19.46)(0.08,0.1){4}{\line(0,1){0.1}}
\multiput(1.7,19.04)(0.07,0.1){4}{\line(0,1){0.1}}
\multiput(1.45,18.6)(0.08,0.15){3}{\line(0,1){0.15}}
\multiput(1.25,18.13)(0.07,0.16){3}{\line(0,1){0.16}}
\multiput(1.1,17.64)(0.08,0.24){2}{\line(0,1){0.24}}
\multiput(0.99,17.14)(0.1,0.5){1}{\line(0,1){0.5}}
\multiput(0.94,16.64)(0.05,0.51){1}{\line(0,1){0.51}}
\multiput(0.94,16.13)(0,0.51){1}{\line(0,1){0.51}}
\multiput(0.94,16.13)(0.05,-0.51){1}{\line(0,-1){0.51}}
\multiput(0.99,15.62)(0.1,-0.5){1}{\line(0,-1){0.5}}

\linethickness{0.3mm}
\multiput(0.96,15.26)(0,0.48){1}{\line(0,1){0.48}}
\multiput(0.96,15.26)(0.05,-0.48){1}{\line(0,-1){0.48}}
\multiput(1.01,14.78)(0.11,-0.47){1}{\line(0,-1){0.47}}
\multiput(1.12,14.3)(0.08,-0.23){2}{\line(0,-1){0.23}}
\multiput(1.28,13.85)(0.07,-0.15){3}{\line(0,-1){0.15}}
\multiput(1.49,13.41)(0.09,-0.14){3}{\line(0,-1){0.14}}
\multiput(1.74,13)(0.08,-0.09){4}{\line(0,-1){0.09}}
\multiput(2.05,12.62)(0.09,-0.09){4}{\line(1,0){0.09}}
\multiput(2.39,12.28)(0.1,-0.07){4}{\line(1,0){0.1}}
\multiput(2.77,11.98)(0.14,-0.08){3}{\line(1,0){0.14}}
\multiput(3.18,11.73)(0.15,-0.07){3}{\line(1,0){0.15}}
\multiput(3.62,11.52)(0.23,-0.08){2}{\line(1,0){0.23}}

\linethickness{1mm}
\put(6,24.37){\line(0,1){3.13}}
\linethickness{0.3mm}
\multiput(12.84,25.74)(0.09,0.14){3}{\line(0,1){0.14}}
\multiput(13.12,26.16)(0.09,0.14){3}{\line(0,1){0.14}}
\multiput(13.38,26.59)(0.09,0.15){3}{\line(0,1){0.15}}
\multiput(13.64,27.03)(0.08,0.15){3}{\line(0,1){0.15}}
\multiput(13.88,27.47)(0.08,0.15){3}{\line(0,1){0.15}}
\multiput(14.12,27.91)(0.08,0.15){3}{\line(0,1){0.15}}
\multiput(14.35,28.37)(0.07,0.15){3}{\line(0,1){0.15}}
\multiput(14.56,28.82)(0.07,0.15){3}{\line(0,1){0.15}}
\multiput(14.77,29.28)(0.1,0.23){2}{\line(0,1){0.23}}
\multiput(14.96,29.75)(0.09,0.24){2}{\line(0,1){0.24}}
\multiput(15.15,30.22)(0.09,0.24){2}{\line(0,1){0.24}}
\multiput(15.32,30.69)(0.08,0.24){2}{\line(0,1){0.24}}
\multiput(15.48,31.17)(0.08,0.24){2}{\line(0,1){0.24}}
\multiput(15.63,31.65)(0.07,0.24){2}{\line(0,1){0.24}}
\multiput(15.78,32.14)(0.06,0.24){2}{\line(0,1){0.24}}
\multiput(15.9,32.63)(0.12,0.49){1}{\line(0,1){0.49}}
\multiput(16.02,33.12)(0.11,0.49){1}{\line(0,1){0.49}}
\multiput(16.13,33.61)(0.1,0.5){1}{\line(0,1){0.5}}
\multiput(16.23,34.11)(0.08,0.5){1}{\line(0,1){0.5}}
\multiput(16.31,34.6)(0.07,0.5){1}{\line(0,1){0.5}}
\multiput(16.38,35.1)(0.06,0.5){1}{\line(0,1){0.5}}
\multiput(16.45,35.61)(0.05,0.5){1}{\line(0,1){0.5}}
\multiput(16.5,36.11)(0.04,0.5){1}{\line(0,1){0.5}}
\multiput(16.54,36.61)(0.03,0.5){1}{\line(0,1){0.5}}
\multiput(16.56,37.12)(0.02,0.5){1}{\line(0,1){0.5}}

\linethickness{0.3mm}
\multiput(20.96,37.62)(0.04,-35.5){1}{\line(0,-1){35.5}}
\linethickness{0.3mm}
\multiput(25.34,37.62)(0.04,-35.5){1}{\line(0,-1){35.5}}
\linethickness{0.3mm}
\put(16.26,1.14){\circle{1.98}}

\linethickness{0.3mm}
\put(21,1.14){\circle{1.98}}

\linethickness{0.3mm}
\put(25.38,1.24){\circle{1.97}}

\linethickness{0.3mm}
\put(25.37,38.64){\circle{1.97}}

\linethickness{0.3mm}
\put(21,38.64){\circle{1.98}}

\linethickness{0.3mm}
\put(16.36,38.74){\circle{1.97}}

\linethickness{1mm}
\put(6,31.88){\line(0,1){7.13}}
\end{picture}
$\qquad$ $\qquad$ $\qquad$ $\qquad$
\unitlength 0.7mm
\begin{picture}(40,45)(0,0)
%amc1
\linethickness{0.3mm}
\multiput(8.78,31.89)(0.5,-0.01){1}{\line(1,0){0.5}}
\multiput(9.28,31.88)(0.5,0){1}{\line(1,0){0.5}}
\multiput(9.78,31.88)(0.5,0.02){1}{\line(1,0){0.5}}
\multiput(10.28,31.89)(0.5,0.03){1}{\line(1,0){0.5}}
\multiput(10.78,31.92)(0.5,0.04){1}{\line(1,0){0.5}}
\multiput(11.28,31.97)(0.5,0.06){1}{\line(1,0){0.5}}
\multiput(11.78,32.02)(0.5,0.07){1}{\line(1,0){0.5}}
\multiput(12.27,32.09)(0.49,0.08){1}{\line(1,0){0.49}}
\multiput(12.77,32.18)(0.49,0.1){1}{\line(1,0){0.49}}
\multiput(13.26,32.27)(0.49,0.11){1}{\line(1,0){0.49}}
\multiput(13.75,32.38)(0.49,0.12){1}{\line(1,0){0.49}}
\multiput(14.23,32.5)(0.48,0.14){1}{\line(1,0){0.48}}
\multiput(14.71,32.64)(0.48,0.15){1}{\line(1,0){0.48}}
\multiput(15.19,32.79)(0.47,0.16){1}{\line(1,0){0.47}}
\multiput(15.67,32.95)(0.47,0.17){1}{\line(1,0){0.47}}
\multiput(16.13,33.12)(0.23,0.09){2}{\line(1,0){0.23}}
\multiput(16.6,33.31)(0.23,0.1){2}{\line(1,0){0.23}}
\multiput(17.06,33.51)(0.23,0.11){2}{\line(1,0){0.23}}
\multiput(17.51,33.72)(0.22,0.11){2}{\line(1,0){0.22}}
\multiput(17.96,33.94)(0.22,0.12){2}{\line(1,0){0.22}}
\multiput(18.4,34.17)(0.22,0.12){2}{\line(1,0){0.22}}
\multiput(18.84,34.42)(0.21,0.13){2}{\line(1,0){0.21}}
\multiput(19.27,34.68)(0.21,0.13){2}{\line(1,0){0.21}}
\multiput(19.69,34.95)(0.21,0.14){2}{\line(1,0){0.21}}
\multiput(20.11,35.23)(0.2,0.15){2}{\line(1,0){0.2}}
\multiput(20.51,35.52)(0.13,0.1){3}{\line(1,0){0.13}}
\multiput(20.91,35.82)(0.13,0.1){3}{\line(1,0){0.13}}
\multiput(21.3,36.13)(0.13,0.11){3}{\line(1,0){0.13}}
\multiput(21.68,36.46)(0.12,0.11){3}{\line(1,0){0.12}}
\multiput(22.06,36.79)(0.12,0.11){3}{\line(1,0){0.12}}
\multiput(22.42,37.13)(0.12,0.12){3}{\line(1,0){0.12}}
\multiput(22.78,37.49)(0.12,0.12){3}{\line(0,1){0.12}}
\multiput(23.12,37.85)(0.11,0.12){3}{\line(0,1){0.12}}
\multiput(23.46,38.22)(0.11,0.13){3}{\line(0,1){0.13}}
\multiput(23.78,38.6)(0.11,0.13){3}{\line(0,1){0.13}}
\multiput(24.1,38.99)(0.1,0.13){3}{\line(0,1){0.13}}

%amc2
\linethickness{0.3mm}
\multiput(3.92,31.74)(0.48,0.15){1}{\line(1,0){0.48}}
\multiput(3.46,31.52)(0.23,0.11){2}{\line(1,0){0.23}}
\multiput(3.04,31.24)(0.21,0.14){2}{\line(1,0){0.21}}
\multiput(2.67,30.9)(0.13,0.11){3}{\line(1,0){0.13}}
\multiput(2.34,30.5)(0.11,0.13){3}{\line(0,1){0.13}}
\multiput(2.08,30.07)(0.13,0.22){2}{\line(0,1){0.22}}
\multiput(1.88,29.6)(0.1,0.23){2}{\line(0,1){0.23}}
\multiput(1.75,29.11)(0.13,0.49){1}{\line(0,1){0.49}}
\multiput(1.7,28.61)(0.06,0.5){1}{\line(0,1){0.5}}
\multiput(1.7,28.61)(0.02,-0.51){1}{\line(0,-1){0.51}}
\multiput(1.72,28.1)(0.09,-0.5){1}{\line(0,-1){0.5}}
\multiput(1.81,27.6)(0.16,-0.48){1}{\line(0,-1){0.48}}
\multiput(1.97,27.12)(0.11,-0.23){2}{\line(0,-1){0.23}}
\multiput(2.2,26.67)(0.15,-0.21){2}{\line(0,-1){0.21}}
\multiput(2.49,26.26)(0.12,-0.12){3}{\line(0,-1){0.12}}
\multiput(2.84,25.89)(0.13,-0.1){3}{\line(1,0){0.13}}
\multiput(3.24,25.58)(0.22,-0.13){2}{\line(1,0){0.22}}
\multiput(3.68,25.33)(0.24,-0.09){2}{\line(1,0){0.24}}
\multiput(4.15,25.14)(0.49,-0.12){1}{\line(1,0){0.49}}
\multiput(4.64,25.02)(0.51,-0.04){1}{\line(1,0){0.51}}
\multiput(5.15,24.98)(0.51,0.03){1}{\line(1,0){0.51}}

%amc3
\linethickness{1mm}
\put(6.28,0.64){\line(0,1){21.25}}
\linethickness{0.3mm}
\multiput(18.38,20.95)(0.13,-0.1){3}{\line(1,0){0.13}}
\multiput(17.98,21.25)(0.14,-0.1){3}{\line(1,0){0.14}}
\multiput(17.56,21.54)(0.21,-0.14){2}{\line(1,0){0.21}}
\multiput(17.14,21.82)(0.21,-0.14){2}{\line(1,0){0.21}}
\multiput(16.72,22.08)(0.21,-0.13){2}{\line(1,0){0.21}}
\multiput(16.28,22.34)(0.22,-0.13){2}{\line(1,0){0.22}}
\multiput(15.84,22.59)(0.22,-0.12){2}{\line(1,0){0.22}}
\multiput(15.4,22.82)(0.22,-0.12){2}{\line(1,0){0.22}}
\multiput(14.94,23.05)(0.23,-0.11){2}{\line(1,0){0.23}}
\multiput(14.49,23.26)(0.23,-0.11){2}{\line(1,0){0.23}}
\multiput(14.02,23.46)(0.23,-0.1){2}{\line(1,0){0.23}}
\multiput(13.56,23.65)(0.23,-0.09){2}{\line(1,0){0.23}}
\multiput(13.08,23.83)(0.47,-0.18){1}{\line(1,0){0.47}}
\multiput(12.61,23.99)(0.48,-0.17){1}{\line(1,0){0.48}}
\multiput(12.13,24.15)(0.48,-0.15){1}{\line(1,0){0.48}}
\multiput(11.64,24.29)(0.48,-0.14){1}{\line(1,0){0.48}}
\multiput(11.15,24.42)(0.49,-0.13){1}{\line(1,0){0.49}}
\multiput(10.66,24.53)(0.49,-0.12){1}{\line(1,0){0.49}}
\multiput(10.17,24.64)(0.49,-0.1){1}{\line(1,0){0.49}}
\multiput(9.67,24.73)(0.5,-0.09){1}{\line(1,0){0.5}}
\multiput(9.18,24.81)(0.5,-0.08){1}{\line(1,0){0.5}}
\multiput(8.68,24.88)(0.5,-0.07){1}{\line(1,0){0.5}}
\multiput(8.17,24.93)(0.5,-0.05){1}{\line(1,0){0.5}}
\multiput(7.67,24.97)(0.5,-0.04){1}{\line(1,0){0.5}}
\multiput(7.17,25)(0.5,-0.03){1}{\line(1,0){0.5}}
\multiput(6.66,25.02)(0.5,-0.02){1}{\line(1,0){0.5}}
\multiput(6.16,25.02)(0.5,-0){1}{\line(1,0){0.5}}
\multiput(5.65,25.01)(0.5,0.01){1}{\line(1,0){0.5}}

%amc4
\linethickness{1mm}
\put(6.28,29.03){\line(0,1){12.86}}
\linethickness{0.3mm}
\multiput(14.4,39.39)(0.09,-0.51){1}{\line(0,-1){0.51}}
\multiput(14.5,38.87)(0.14,-0.5){1}{\line(0,-1){0.5}}
\multiput(14.64,38.38)(0.1,-0.24){2}{\line(0,-1){0.24}}
\multiput(14.83,37.89)(0.12,-0.23){2}{\line(0,-1){0.23}}
\multiput(15.07,37.44)(0.14,-0.22){2}{\line(0,-1){0.22}}
\multiput(15.36,37)(0.11,-0.13){3}{\line(0,-1){0.13}}
\multiput(15.69,36.6)(0.12,-0.12){3}{\line(1,0){0.12}}
\multiput(16.06,36.24)(0.14,-0.11){3}{\line(1,0){0.14}}
\multiput(16.47,35.92)(0.22,-0.14){2}{\line(1,0){0.22}}

%amc5
\linethickness{0.3mm}
\multiput(20.03,32.51)(0.11,-0.12){3}{\line(0,-1){0.12}}
\multiput(20.37,32.15)(0.12,-0.12){3}{\line(0,-1){0.12}}
\multiput(20.72,31.8)(0.12,-0.12){3}{\line(1,0){0.12}}
\multiput(21.08,31.45)(0.12,-0.11){3}{\line(1,0){0.12}}
\multiput(21.44,31.11)(0.12,-0.11){3}{\line(1,0){0.12}}
\multiput(21.81,30.78)(0.13,-0.11){3}{\line(1,0){0.13}}
\multiput(22.19,30.46)(0.13,-0.1){3}{\line(1,0){0.13}}
\multiput(22.58,30.15)(0.13,-0.1){3}{\line(1,0){0.13}}
\multiput(22.97,29.84)(0.2,-0.15){2}{\line(1,0){0.2}}
\multiput(23.37,29.55)(0.2,-0.14){2}{\line(1,0){0.2}}
\multiput(23.78,29.26)(0.21,-0.14){2}{\line(1,0){0.21}}
\multiput(24.19,28.98)(0.21,-0.13){2}{\line(1,0){0.21}}
\multiput(24.61,28.72)(0.21,-0.13){2}{\line(1,0){0.21}}
\multiput(25.04,28.46)(0.22,-0.12){2}{\line(1,0){0.22}}
\multiput(25.47,28.21)(0.22,-0.12){2}{\line(1,0){0.22}}
\multiput(25.9,27.97)(0.22,-0.12){2}{\line(1,0){0.22}}
\multiput(26.34,27.74)(0.22,-0.11){2}{\line(1,0){0.22}}
\multiput(26.79,27.52)(0.23,-0.11){2}{\line(1,0){0.23}}
\multiput(27.24,27.31)(0.23,-0.1){2}{\line(1,0){0.23}}
\multiput(27.69,27.11)(0.23,-0.1){2}{\line(1,0){0.23}}
\multiput(28.15,26.92)(0.46,-0.18){1}{\line(1,0){0.46}}
\multiput(28.62,26.74)(0.47,-0.17){1}{\line(1,0){0.47}}
\multiput(29.08,26.57)(0.47,-0.16){1}{\line(1,0){0.47}}
\multiput(29.56,26.41)(0.47,-0.15){1}{\line(1,0){0.47}}

%amc6
\linethickness{0.3mm}
\multiput(32.53,11.89)(0.11,0.12){3}{\line(0,1){0.12}}
\multiput(32.87,12.26)(0.1,0.13){3}{\line(0,1){0.13}}
\multiput(33.18,12.64)(0.15,0.2){2}{\line(0,1){0.2}}
\multiput(33.47,13.05)(0.14,0.21){2}{\line(0,1){0.21}}
\multiput(33.74,13.47)(0.12,0.22){2}{\line(0,1){0.22}}
\multiput(33.99,13.9)(0.11,0.22){2}{\line(0,1){0.22}}
\multiput(34.21,14.35)(0.1,0.23){2}{\line(0,1){0.23}}
\multiput(34.41,14.8)(0.17,0.47){1}{\line(0,1){0.47}}
\multiput(34.58,15.27)(0.15,0.48){1}{\line(0,1){0.48}}
\multiput(34.73,15.75)(0.12,0.48){1}{\line(0,1){0.48}}
\multiput(34.85,16.24)(0.09,0.49){1}{\line(0,1){0.49}}
\multiput(34.94,16.73)(0.07,0.49){1}{\line(0,1){0.49}}
\multiput(35,17.22)(0.04,0.5){1}{\line(0,1){0.5}}
\multiput(35.04,17.72)(0.01,0.5){1}{\line(0,1){0.5}}
\multiput(35.04,18.72)(0.02,-0.5){1}{\line(0,-1){0.5}}
\multiput(34.99,19.21)(0.04,-0.5){1}{\line(0,-1){0.5}}
\multiput(34.92,19.71)(0.07,-0.49){1}{\line(0,-1){0.49}}
\multiput(34.82,20.2)(0.1,-0.49){1}{\line(0,-1){0.49}}
\multiput(34.69,20.68)(0.13,-0.48){1}{\line(0,-1){0.48}}
\multiput(34.54,21.15)(0.15,-0.48){1}{\line(0,-1){0.48}}
\multiput(34.36,21.62)(0.18,-0.47){1}{\line(0,-1){0.47}}
\multiput(34.16,22.08)(0.1,-0.23){2}{\line(0,-1){0.23}}
\multiput(33.93,22.52)(0.11,-0.22){2}{\line(0,-1){0.22}}
\multiput(33.68,22.95)(0.13,-0.22){2}{\line(0,-1){0.22}}
\multiput(33.4,23.37)(0.14,-0.21){2}{\line(0,-1){0.21}}
\multiput(33.1,23.77)(0.15,-0.2){2}{\line(0,-1){0.2}}
\multiput(32.78,24.15)(0.11,-0.13){3}{\line(0,-1){0.13}}
\multiput(32.44,24.51)(0.11,-0.12){3}{\line(0,-1){0.12}}
\multiput(32.08,24.86)(0.12,-0.12){3}{\line(1,0){0.12}}
\multiput(31.71,25.18)(0.13,-0.11){3}{\line(1,0){0.13}}
\multiput(31.31,25.49)(0.13,-0.1){3}{\line(1,0){0.13}}
\multiput(30.9,25.77)(0.21,-0.14){2}{\line(1,0){0.21}}
\multiput(30.47,26.03)(0.21,-0.13){2}{\line(1,0){0.21}}
\multiput(30.03,26.26)(0.22,-0.12){2}{\line(1,0){0.22}}

%amc7
\linethickness{0.3mm}
\multiput(25.88,25.34)(0.2,0.15){2}{\line(1,0){0.2}}
\multiput(25.49,25.03)(0.13,0.1){3}{\line(1,0){0.13}}
\multiput(25.1,24.72)(0.13,0.1){3}{\line(1,0){0.13}}
\multiput(24.72,24.4)(0.13,0.11){3}{\line(1,0){0.13}}
\multiput(24.35,24.08)(0.12,0.11){3}{\line(1,0){0.12}}
\multiput(23.98,23.74)(0.12,0.11){3}{\line(1,0){0.12}}
\multiput(23.62,23.4)(0.12,0.11){3}{\line(1,0){0.12}}
\multiput(23.27,23.06)(0.12,0.12){3}{\line(1,0){0.12}}
\multiput(22.92,22.7)(0.12,0.12){3}{\line(0,1){0.12}}
\multiput(22.58,22.34)(0.11,0.12){3}{\line(0,1){0.12}}
\multiput(22.24,21.98)(0.11,0.12){3}{\line(0,1){0.12}}
\multiput(21.92,21.6)(0.11,0.12){3}{\line(0,1){0.12}}
\multiput(21.6,21.22)(0.11,0.13){3}{\line(0,1){0.13}}
\multiput(21.29,20.84)(0.1,0.13){3}{\line(0,1){0.13}}
\multiput(20.98,20.45)(0.1,0.13){3}{\line(0,1){0.13}}
\multiput(20.68,20.05)(0.15,0.2){2}{\line(0,1){0.2}}
\multiput(20.39,19.65)(0.14,0.2){2}{\line(0,1){0.2}}
\multiput(20.11,19.24)(0.14,0.2){2}{\line(0,1){0.2}}
\multiput(19.84,18.82)(0.14,0.21){2}{\line(0,1){0.21}}
\multiput(19.57,18.41)(0.13,0.21){2}{\line(0,1){0.21}}
\multiput(19.31,17.98)(0.13,0.21){2}{\line(0,1){0.21}}
\multiput(19.06,17.55)(0.13,0.21){2}{\line(0,1){0.21}}
\multiput(18.82,17.12)(0.12,0.22){2}{\line(0,1){0.22}}
\multiput(18.59,16.68)(0.12,0.22){2}{\line(0,1){0.22}}
\multiput(18.36,16.24)(0.11,0.22){2}{\line(0,1){0.22}}
\multiput(18.14,15.79)(0.11,0.22){2}{\line(0,1){0.22}}
\multiput(17.94,15.34)(0.1,0.23){2}{\line(0,1){0.23}}
\multiput(17.74,14.89)(0.1,0.23){2}{\line(0,1){0.23}}
\multiput(17.54,14.43)(0.1,0.23){2}{\line(0,1){0.23}}
\multiput(17.36,13.97)(0.09,0.23){2}{\line(0,1){0.23}}
\multiput(17.19,13.5)(0.17,0.46){1}{\line(0,1){0.46}}
\multiput(17.02,13.04)(0.16,0.47){1}{\line(0,1){0.47}}
\multiput(16.87,12.57)(0.16,0.47){1}{\line(0,1){0.47}}
\multiput(16.72,12.09)(0.15,0.47){1}{\line(0,1){0.47}}
\multiput(16.59,11.61)(0.14,0.48){1}{\line(0,1){0.48}}
\multiput(16.46,11.14)(0.13,0.48){1}{\line(0,1){0.48}}
\multiput(16.34,10.65)(0.12,0.48){1}{\line(0,1){0.48}}
\multiput(16.23,10.17)(0.11,0.48){1}{\line(0,1){0.48}}
\multiput(16.13,9.68)(0.1,0.49){1}{\line(0,1){0.49}}
\multiput(16.04,9.2)(0.09,0.49){1}{\line(0,1){0.49}}
\multiput(15.96,8.71)(0.08,0.49){1}{\line(0,1){0.49}}
\multiput(15.88,8.22)(0.07,0.49){1}{\line(0,1){0.49}}
\multiput(15.82,7.72)(0.06,0.49){1}{\line(0,1){0.49}}
\multiput(15.77,7.23)(0.05,0.49){1}{\line(0,1){0.49}}
\multiput(15.72,6.74)(0.04,0.49){1}{\line(0,1){0.49}}
\multiput(15.69,6.24)(0.03,0.49){1}{\line(0,1){0.49}}
\multiput(15.66,5.75)(0.03,0.5){1}{\line(0,1){0.5}}
\multiput(15.65,5.25)(0.02,0.5){1}{\line(0,1){0.5}}
\multiput(15.64,4.75)(0.01,0.5){1}{\line(0,1){0.5}}
\multiput(15.64,4.75)(0,-0.5){1}{\line(0,-1){0.5}}
\multiput(15.64,4.26)(0.01,-0.5){1}{\line(0,-1){0.5}}

%amc8
\linethickness{0.3mm}
\multiput(29.4,28.76)(0.12,0.12){3}{\line(1,0){0.12}}
\multiput(29.77,29.11)(0.12,0.12){3}{\line(0,1){0.12}}
\multiput(30.13,29.48)(0.12,0.12){3}{\line(0,1){0.12}}
\multiput(30.48,29.85)(0.11,0.13){3}{\line(0,1){0.13}}
\multiput(30.82,30.23)(0.11,0.13){3}{\line(0,1){0.13}}
\multiput(31.14,30.63)(0.11,0.13){3}{\line(0,1){0.13}}
\multiput(31.46,31.03)(0.1,0.14){3}{\line(0,1){0.14}}
\multiput(31.76,31.44)(0.15,0.21){2}{\line(0,1){0.21}}
\multiput(32.06,31.85)(0.14,0.21){2}{\line(0,1){0.21}}
\multiput(32.34,32.28)(0.13,0.22){2}{\line(0,1){0.22}}
\multiput(32.61,32.71)(0.13,0.22){2}{\line(0,1){0.22}}
\multiput(32.86,33.15)(0.12,0.22){2}{\line(0,1){0.22}}
\multiput(33.11,33.6)(0.12,0.23){2}{\line(0,1){0.23}}
\multiput(33.34,34.06)(0.11,0.23){2}{\line(0,1){0.23}}
\multiput(33.56,34.52)(0.1,0.23){2}{\line(0,1){0.23}}
\multiput(33.77,34.98)(0.1,0.24){2}{\line(0,1){0.24}}
\multiput(33.96,35.45)(0.09,0.24){2}{\line(0,1){0.24}}
\multiput(34.14,35.93)(0.17,0.48){1}{\line(0,1){0.48}}
\multiput(34.31,36.41)(0.15,0.49){1}{\line(0,1){0.49}}
\multiput(34.46,36.9)(0.14,0.49){1}{\line(0,1){0.49}}
\multiput(34.6,37.39)(0.13,0.49){1}{\line(0,1){0.49}}
\multiput(34.73,37.88)(0.11,0.5){1}{\line(0,1){0.5}}
\multiput(34.84,38.38)(0.1,0.5){1}{\line(0,1){0.5}}
\multiput(34.94,38.88)(0.09,0.5){1}{\line(0,1){0.5}}

%amc9
\linethickness{0.3mm}
\multiput(35.44,4.2)(0.11,-0.22){2}{\line(0,-1){0.22}}
\multiput(35.21,4.65)(0.11,-0.22){2}{\line(0,-1){0.22}}
\multiput(34.98,5.08)(0.11,-0.22){2}{\line(0,-1){0.22}}
\multiput(34.75,5.52)(0.12,-0.22){2}{\line(0,-1){0.22}}
\multiput(34.51,5.95)(0.12,-0.22){2}{\line(0,-1){0.22}}
\multiput(34.26,6.38)(0.12,-0.21){2}{\line(0,-1){0.21}}
\multiput(34,6.8)(0.13,-0.21){2}{\line(0,-1){0.21}}
\multiput(33.75,7.22)(0.13,-0.21){2}{\line(0,-1){0.21}}
\multiput(33.48,7.64)(0.13,-0.21){2}{\line(0,-1){0.21}}
\multiput(33.21,8.05)(0.14,-0.21){2}{\line(0,-1){0.21}}
\multiput(32.93,8.46)(0.14,-0.2){2}{\line(0,-1){0.2}}
\multiput(32.65,8.87)(0.14,-0.2){2}{\line(0,-1){0.2}}
\multiput(32.37,9.27)(0.14,-0.2){2}{\line(0,-1){0.2}}
\multiput(32.07,9.67)(0.15,-0.2){2}{\line(0,-1){0.2}}
\multiput(31.77,10.06)(0.15,-0.2){2}{\line(0,-1){0.2}}
\multiput(31.47,10.45)(0.1,-0.13){3}{\line(0,-1){0.13}}
\multiput(31.16,10.84)(0.1,-0.13){3}{\line(0,-1){0.13}}
\multiput(30.85,11.22)(0.1,-0.13){3}{\line(0,-1){0.13}}
\multiput(30.53,11.6)(0.11,-0.13){3}{\line(0,-1){0.13}}
\multiput(30.2,11.97)(0.11,-0.12){3}{\line(0,-1){0.12}}
\multiput(29.88,12.34)(0.11,-0.12){3}{\line(0,-1){0.12}}
\multiput(29.54,12.7)(0.11,-0.12){3}{\line(0,-1){0.12}}
\multiput(29.2,13.06)(0.11,-0.12){3}{\line(0,-1){0.12}}
\multiput(28.86,13.42)(0.11,-0.12){3}{\line(0,-1){0.12}}
\multiput(28.51,13.77)(0.12,-0.12){3}{\line(0,-1){0.12}}
\multiput(28.15,14.11)(0.12,-0.12){3}{\line(1,0){0.12}}
\multiput(27.8,14.45)(0.12,-0.11){3}{\line(1,0){0.12}}
\multiput(27.43,14.79)(0.12,-0.11){3}{\line(1,0){0.12}}
\multiput(27.07,15.12)(0.12,-0.11){3}{\line(1,0){0.12}}
\multiput(26.69,15.44)(0.12,-0.11){3}{\line(1,0){0.12}}
\multiput(26.32,15.76)(0.13,-0.11){3}{\line(1,0){0.13}}
\multiput(25.94,16.08)(0.13,-0.11){3}{\line(1,0){0.13}}
\multiput(25.55,16.39)(0.13,-0.1){3}{\line(1,0){0.13}}
\multiput(25.16,16.69)(0.13,-0.1){3}{\line(1,0){0.13}}
\multiput(24.77,16.99)(0.2,-0.15){2}{\line(1,0){0.2}}
\multiput(24.37,17.29)(0.2,-0.15){2}{\line(1,0){0.2}}
\multiput(23.97,17.58)(0.2,-0.14){2}{\line(1,0){0.2}}
\multiput(23.56,17.86)(0.2,-0.14){2}{\line(1,0){0.2}}
\multiput(23.15,18.14)(0.2,-0.14){2}{\line(1,0){0.2}}

%amc10
\linethickness{0.3mm}
\multiput(29.02,9.05)(0.13,0.11){3}{\line(1,0){0.13}}
\multiput(28.64,8.7)(0.13,0.12){3}{\line(1,0){0.13}}
\multiput(28.27,8.34)(0.12,0.12){3}{\line(1,0){0.12}}
\multiput(27.92,7.96)(0.12,0.13){3}{\line(0,1){0.13}}
\multiput(27.58,7.57)(0.11,0.13){3}{\line(0,1){0.13}}
\multiput(27.26,7.17)(0.11,0.13){3}{\line(0,1){0.13}}
\multiput(26.95,6.76)(0.1,0.14){3}{\line(0,1){0.14}}
\multiput(26.65,6.34)(0.15,0.21){2}{\line(0,1){0.21}}
\multiput(26.37,5.91)(0.14,0.22){2}{\line(0,1){0.22}}
\multiput(26.11,5.47)(0.13,0.22){2}{\line(0,1){0.22}}
\multiput(25.86,5.02)(0.12,0.23){2}{\line(0,1){0.23}}
\multiput(25.63,4.56)(0.12,0.23){2}{\line(0,1){0.23}}
\multiput(25.41,4.09)(0.11,0.23){2}{\line(0,1){0.23}}
\multiput(25.21,3.62)(0.1,0.24){2}{\line(0,1){0.24}}
\multiput(25.03,3.14)(0.09,0.24){2}{\line(0,1){0.24}}

%amc11
\linethickness{0.3mm}
\put(35.47,2.07){\circle*{2.87}}

%amc12
\linethickness{0.3mm}
\put(25.47,2.07){\circle*{2.87}}

%amc13
\linethickness{0.3mm}
\put(15.47,2.07){\circle*{2.87}}

%amc14
\linethickness{0.3mm}
\put(14.85,40.82){\circle*{2.87}}

%amc15
\linethickness{0.3mm}
\put(24.85,40.69){\circle*{2.87}}

%amc16
\linethickness{0.3mm}
\put(34.85,40.82){\circle*{2.87}}

\end{picture}
\caption{Two $(3,3)$-tangles, one of type $\D^{(2)}$ that is not of type
$\D$, and one of type $\D^{(1)}$ that is not of type $\D^{(2)}$}
\label{pic:rho}\label{pic:tangletypeD}
\end{figure}

\np As only isotopy of the plane is allowed that does not affect the pole, we
forbid tangles to have crossings of strands at the left side
of the pole. Moreover, all pole twists in a tangle diagram are isolated from
each other. So when traversing the pole from top to bottom, the twists of a
tangle are met one by one.  An example of a tangle of type $\D^{(1)}$ is given
in Figure \ref{pic:tangletypeD}.

\begin{figure}[htbp]
%{ch5doubletwist}
\unitlength 0.6mm
\begin{picture}(47.42,49.27)(0,0)
\linethickness{1mm}
\put(8.12,0.62){\line(0,1){20}}
\linethickness{0.3mm}
\multiput(7.5,23.12)(0.49,0.11){1}{\line(1,0){0.49}}
\multiput(7.99,23.23)(0.24,0.06){2}{\line(1,0){0.24}}
\multiput(8.48,23.36)(0.24,0.07){2}{\line(1,0){0.24}}
\multiput(8.97,23.5)(0.24,0.08){2}{\line(1,0){0.24}}
\multiput(9.45,23.67)(0.24,0.09){2}{\line(1,0){0.24}}
\multiput(9.92,23.85)(0.23,0.1){2}{\line(1,0){0.23}}
\multiput(10.38,24.05)(0.15,0.07){3}{\line(1,0){0.15}}
\multiput(10.84,24.27)(0.15,0.08){3}{\line(1,0){0.15}}
\multiput(11.29,24.5)(0.15,0.08){3}{\line(1,0){0.15}}
\multiput(11.73,24.76)(0.14,0.09){3}{\line(1,0){0.14}}
\multiput(12.15,25.02)(0.1,0.07){4}{\line(1,0){0.1}}
\multiput(12.57,25.31)(0.1,0.08){4}{\line(1,0){0.1}}
\multiput(12.98,25.61)(0.1,0.08){4}{\line(1,0){0.1}}
\multiput(13.37,25.93)(0.1,0.08){4}{\line(1,0){0.1}}
\multiput(13.76,26.26)(0.09,0.09){4}{\line(1,0){0.09}}
\multiput(14.13,26.6)(0.09,0.09){4}{\line(0,1){0.09}}
\multiput(14.48,26.96)(0.09,0.09){4}{\line(0,1){0.09}}
\multiput(14.82,27.33)(0.08,0.1){4}{\line(0,1){0.1}}
\multiput(15.15,27.72)(0.08,0.1){4}{\line(0,1){0.1}}
\multiput(15.46,28.12)(0.07,0.1){4}{\line(0,1){0.1}}
\multiput(15.76,28.53)(0.07,0.11){4}{\line(0,1){0.11}}
\multiput(16.04,28.95)(0.09,0.14){3}{\line(0,1){0.14}}
\multiput(16.31,29.38)(0.08,0.15){3}{\line(0,1){0.15}}
\multiput(16.56,29.82)(0.08,0.15){3}{\line(0,1){0.15}}
\multiput(16.79,30.27)(0.07,0.15){3}{\line(0,1){0.15}}
\multiput(17,30.73)(0.1,0.23){2}{\line(0,1){0.23}}
\multiput(17.2,31.19)(0.09,0.24){2}{\line(0,1){0.24}}
\multiput(17.37,31.66)(0.08,0.24){2}{\line(0,1){0.24}}
\multiput(17.53,32.14)(0.07,0.24){2}{\line(0,1){0.24}}
\multiput(17.68,32.63)(0.06,0.25){2}{\line(0,1){0.25}}
\multiput(17.8,33.12)(0.1,0.5){1}{\line(0,1){0.5}}
\multiput(17.9,33.62)(0.08,0.5){1}{\line(0,1){0.5}}
\multiput(17.99,34.11)(0.07,0.5){1}{\line(0,1){0.5}}
\multiput(18.05,34.62)(0.05,0.5){1}{\line(0,1){0.5}}
\multiput(18.1,35.12)(0.03,0.51){1}{\line(0,1){0.51}}

\linethickness{0.3mm}
\multiput(7.02,23)(0.24,0.06){2}{\line(1,0){0.24}}
\multiput(6.56,22.84)(0.23,0.08){2}{\line(1,0){0.23}}
\multiput(6.11,22.64)(0.23,0.1){2}{\line(1,0){0.23}}
\multiput(5.67,22.42)(0.15,0.08){3}{\line(1,0){0.15}}
\multiput(5.25,22.16)(0.14,0.09){3}{\line(1,0){0.14}}
\multiput(4.85,21.87)(0.1,0.07){4}{\line(1,0){0.1}}
\multiput(4.47,21.56)(0.09,0.08){4}{\line(1,0){0.09}}
\multiput(4.12,21.21)(0.09,0.09){4}{\line(1,0){0.09}}
\multiput(3.79,20.85)(0.08,0.09){4}{\line(0,1){0.09}}
\multiput(3.49,20.46)(0.08,0.1){4}{\line(0,1){0.1}}
\multiput(3.22,20.04)(0.09,0.14){3}{\line(0,1){0.14}}
\multiput(2.97,19.62)(0.08,0.14){3}{\line(0,1){0.14}}
\multiput(2.77,19.17)(0.07,0.15){3}{\line(0,1){0.15}}
\multiput(2.59,18.71)(0.09,0.23){2}{\line(0,1){0.23}}
\multiput(2.45,18.24)(0.07,0.24){2}{\line(0,1){0.24}}
\multiput(2.34,17.76)(0.11,0.48){1}{\line(0,1){0.48}}
\multiput(2.26,17.27)(0.07,0.49){1}{\line(0,1){0.49}}
\multiput(2.23,16.78)(0.04,0.49){1}{\line(0,1){0.49}}
\multiput(2.23,16.29)(0,0.49){1}{\line(0,1){0.49}}
\multiput(2.23,16.29)(0.03,-0.49){1}{\line(0,-1){0.49}}
\multiput(2.26,15.8)(0.07,-0.49){1}{\line(0,-1){0.49}}
\multiput(2.33,15.31)(0.1,-0.48){1}{\line(0,-1){0.48}}
\multiput(2.43,14.83)(0.07,-0.24){2}{\line(0,-1){0.24}}
\multiput(2.57,14.35)(0.09,-0.23){2}{\line(0,-1){0.23}}
\multiput(2.75,13.89)(0.07,-0.15){3}{\line(0,-1){0.15}}
\multiput(2.95,13.45)(0.08,-0.14){3}{\line(0,-1){0.14}}
\multiput(3.19,13.01)(0.09,-0.14){3}{\line(0,-1){0.14}}
\multiput(3.46,12.6)(0.07,-0.1){4}{\line(0,-1){0.1}}
\multiput(3.76,12.21)(0.08,-0.09){4}{\line(0,-1){0.09}}
\multiput(4.08,11.84)(0.09,-0.09){4}{\line(1,0){0.09}}
\multiput(4.43,11.49)(0.09,-0.08){4}{\line(1,0){0.09}}
\multiput(4.81,11.18)(0.1,-0.07){4}{\line(1,0){0.1}}
\multiput(5.21,10.89)(0.14,-0.09){3}{\line(1,0){0.14}}

\linethickness{0.3mm}
\multiput(18.08,1.13)(0.05,-0.5){1}{\line(0,-1){0.5}}
\multiput(18,1.63)(0.08,-0.5){1}{\line(0,-1){0.5}}
\multiput(17.9,2.12)(0.1,-0.49){1}{\line(0,-1){0.49}}
\multiput(17.77,2.61)(0.07,-0.24){2}{\line(0,-1){0.24}}
\multiput(17.61,3.09)(0.08,-0.24){2}{\line(0,-1){0.24}}
\multiput(17.42,3.56)(0.09,-0.23){2}{\line(0,-1){0.23}}
\multiput(17.2,4.02)(0.07,-0.15){3}{\line(0,-1){0.15}}
\multiput(16.96,4.46)(0.08,-0.15){3}{\line(0,-1){0.15}}
\multiput(16.7,4.89)(0.09,-0.14){3}{\line(0,-1){0.14}}
\multiput(16.41,5.31)(0.07,-0.1){4}{\line(0,-1){0.1}}
\multiput(16.1,5.71)(0.08,-0.1){4}{\line(0,-1){0.1}}
\multiput(15.76,6.08)(0.08,-0.09){4}{\line(0,-1){0.09}}
\multiput(15.4,6.44)(0.09,-0.09){4}{\line(0,-1){0.09}}
\multiput(15.03,6.78)(0.09,-0.08){4}{\line(1,0){0.09}}
\multiput(14.63,7.09)(0.1,-0.08){4}{\line(1,0){0.1}}
\multiput(14.22,7.39)(0.1,-0.07){4}{\line(1,0){0.1}}
\multiput(13.79,7.65)(0.14,-0.09){3}{\line(1,0){0.14}}
\multiput(13.34,7.89)(0.15,-0.08){3}{\line(1,0){0.15}}
\multiput(12.89,8.11)(0.15,-0.07){3}{\line(1,0){0.15}}
\multiput(12.42,8.3)(0.23,-0.09){2}{\line(1,0){0.23}}
\multiput(11.94,8.46)(0.24,-0.08){2}{\line(1,0){0.24}}
\multiput(11.45,8.6)(0.24,-0.07){2}{\line(1,0){0.24}}
\multiput(10.96,8.7)(0.49,-0.11){1}{\line(1,0){0.49}}

\linethickness{1mm}
\put(8.12,26.25){\line(0,1){9.38}}
\put(23.75,16.88){\makebox(0,0)[cc]{$=$}}

\linethickness{1mm}
\put(36.88,0.62){\line(0,1){6.25}}
\linethickness{0.3mm}
\multiput(39.38,24.38)(0.23,0.1){2}{\line(1,0){0.23}}
\multiput(39.83,24.57)(0.15,0.07){3}{\line(1,0){0.15}}
\multiput(40.27,24.79)(0.15,0.08){3}{\line(1,0){0.15}}
\multiput(40.71,25.02)(0.14,0.08){3}{\line(1,0){0.14}}
\multiput(41.14,25.27)(0.14,0.09){3}{\line(1,0){0.14}}
\multiput(41.55,25.54)(0.1,0.07){4}{\line(1,0){0.1}}
\multiput(41.96,25.83)(0.1,0.08){4}{\line(1,0){0.1}}
\multiput(42.35,26.13)(0.09,0.08){4}{\line(1,0){0.09}}
\multiput(42.73,26.44)(0.09,0.08){4}{\line(1,0){0.09}}
\multiput(43.1,26.78)(0.09,0.09){4}{\line(1,0){0.09}}
\multiput(43.45,27.12)(0.08,0.09){4}{\line(0,1){0.09}}
\multiput(43.79,27.49)(0.08,0.09){4}{\line(0,1){0.09}}
\multiput(44.11,27.86)(0.08,0.1){4}{\line(0,1){0.1}}
\multiput(44.42,28.25)(0.07,0.1){4}{\line(0,1){0.1}}
\multiput(44.71,28.65)(0.09,0.14){3}{\line(0,1){0.14}}
\multiput(44.98,29.06)(0.09,0.14){3}{\line(0,1){0.14}}
\multiput(45.24,29.48)(0.08,0.14){3}{\line(0,1){0.14}}
\multiput(45.48,29.92)(0.07,0.15){3}{\line(0,1){0.15}}
\multiput(45.7,30.36)(0.07,0.15){3}{\line(0,1){0.15}}
\multiput(45.9,30.81)(0.09,0.23){2}{\line(0,1){0.23}}
\multiput(46.09,31.27)(0.08,0.23){2}{\line(0,1){0.23}}
\multiput(46.25,31.73)(0.07,0.24){2}{\line(0,1){0.24}}
\multiput(46.4,32.21)(0.06,0.24){2}{\line(0,1){0.24}}
\multiput(46.53,32.68)(0.11,0.48){1}{\line(0,1){0.48}}
\multiput(46.64,33.17)(0.09,0.49){1}{\line(0,1){0.49}}
\multiput(46.72,33.65)(0.07,0.49){1}{\line(0,1){0.49}}
\multiput(46.79,34.14)(0.05,0.49){1}{\line(0,1){0.49}}
\multiput(46.84,34.64)(0.03,0.49){1}{\line(0,1){0.49}}
\multiput(46.87,35.13)(0.01,0.49){1}{\line(0,1){0.49}}

\linethickness{0.3mm}
\multiput(33.96,22.21)(0.1,0.07){4}{\line(1,0){0.1}}
\multiput(33.57,21.9)(0.1,0.08){4}{\line(1,0){0.1}}
\multiput(33.2,21.56)(0.09,0.09){4}{\line(1,0){0.09}}
\multiput(32.85,21.19)(0.09,0.09){4}{\line(0,1){0.09}}
\multiput(32.54,20.8)(0.08,0.1){4}{\line(0,1){0.1}}
\multiput(32.25,20.39)(0.07,0.1){4}{\line(0,1){0.1}}
\multiput(31.98,19.96)(0.09,0.14){3}{\line(0,1){0.14}}
\multiput(31.75,19.51)(0.08,0.15){3}{\line(0,1){0.15}}
\multiput(31.56,19.05)(0.1,0.23){2}{\line(0,1){0.23}}
\multiput(31.39,18.58)(0.08,0.24){2}{\line(0,1){0.24}}
\multiput(31.26,18.09)(0.07,0.24){2}{\line(0,1){0.24}}
\multiput(31.16,17.6)(0.1,0.49){1}{\line(0,1){0.49}}
\multiput(31.1,17.1)(0.06,0.5){1}{\line(0,1){0.5}}
\multiput(31.07,16.6)(0.03,0.5){1}{\line(0,1){0.5}}
\multiput(31.07,16.6)(0.01,-0.5){1}{\line(0,-1){0.5}}
\multiput(31.07,16.09)(0.04,-0.5){1}{\line(0,-1){0.5}}
\multiput(31.12,15.59)(0.08,-0.5){1}{\line(0,-1){0.5}}
\multiput(31.19,15.09)(0.11,-0.49){1}{\line(0,-1){0.49}}
\multiput(31.3,14.6)(0.07,-0.24){2}{\line(0,-1){0.24}}
\multiput(31.45,14.12)(0.09,-0.24){2}{\line(0,-1){0.24}}
\multiput(31.63,13.65)(0.07,-0.15){3}{\line(0,-1){0.15}}
\multiput(31.84,13.19)(0.08,-0.15){3}{\line(0,-1){0.15}}
\multiput(32.08,12.75)(0.09,-0.14){3}{\line(0,-1){0.14}}
\multiput(32.35,12.33)(0.08,-0.1){4}{\line(0,-1){0.1}}
\multiput(32.66,11.93)(0.08,-0.1){4}{\line(0,-1){0.1}}
\multiput(32.99,11.55)(0.09,-0.09){4}{\line(0,-1){0.09}}
\multiput(33.34,11.19)(0.09,-0.08){4}{\line(1,0){0.09}}
\multiput(33.72,10.86)(0.1,-0.08){4}{\line(1,0){0.1}}
\multiput(34.12,10.55)(0.14,-0.09){3}{\line(1,0){0.14}}
\multiput(34.54,10.28)(0.15,-0.08){3}{\line(1,0){0.15}}
\multiput(34.98,10.03)(0.15,-0.07){3}{\line(1,0){0.15}}
\multiput(35.44,9.82)(0.23,-0.09){2}{\line(1,0){0.23}}
\multiput(35.9,9.64)(0.24,-0.07){2}{\line(1,0){0.24}}
\multiput(36.39,9.49)(0.49,-0.11){1}{\line(1,0){0.49}}

\linethickness{0.3mm}
\multiput(46.68,1.02)(0.08,-0.24){2}{\line(0,-1){0.24}}
\multiput(46.5,1.49)(0.09,-0.23){2}{\line(0,-1){0.23}}
\multiput(46.3,1.95)(0.1,-0.23){2}{\line(0,-1){0.23}}
\multiput(46.09,2.4)(0.07,-0.15){3}{\line(0,-1){0.15}}
\multiput(45.86,2.84)(0.08,-0.15){3}{\line(0,-1){0.15}}
\multiput(45.61,3.27)(0.08,-0.14){3}{\line(0,-1){0.14}}
\multiput(45.35,3.7)(0.09,-0.14){3}{\line(0,-1){0.14}}
\multiput(45.07,4.11)(0.07,-0.1){4}{\line(0,-1){0.1}}
\multiput(44.77,4.51)(0.07,-0.1){4}{\line(0,-1){0.1}}
\multiput(44.46,4.9)(0.08,-0.1){4}{\line(0,-1){0.1}}
\multiput(44.13,5.28)(0.08,-0.09){4}{\line(0,-1){0.09}}
\multiput(43.79,5.64)(0.09,-0.09){4}{\line(0,-1){0.09}}
\multiput(43.43,5.99)(0.09,-0.09){4}{\line(1,0){0.09}}
\multiput(43.06,6.33)(0.09,-0.08){4}{\line(1,0){0.09}}
\multiput(42.68,6.65)(0.1,-0.08){4}{\line(1,0){0.1}}
\multiput(42.29,6.96)(0.1,-0.08){4}{\line(1,0){0.1}}
\multiput(41.88,7.25)(0.1,-0.07){4}{\line(1,0){0.1}}
\multiput(41.47,7.52)(0.14,-0.09){3}{\line(1,0){0.14}}
\multiput(41.04,7.78)(0.14,-0.09){3}{\line(1,0){0.14}}
\multiput(40.6,8.02)(0.15,-0.08){3}{\line(1,0){0.15}}
\multiput(40.15,8.24)(0.15,-0.07){3}{\line(1,0){0.15}}
\multiput(39.7,8.45)(0.15,-0.07){3}{\line(1,0){0.15}}
\multiput(39.24,8.64)(0.23,-0.09){2}{\line(1,0){0.23}}
\multiput(38.77,8.81)(0.23,-0.08){2}{\line(1,0){0.23}}
\multiput(38.29,8.96)(0.24,-0.08){2}{\line(1,0){0.24}}
\multiput(37.81,9.09)(0.24,-0.07){2}{\line(1,0){0.24}}
\multiput(37.32,9.2)(0.49,-0.11){1}{\line(1,0){0.49}}
\multiput(36.83,9.3)(0.49,-0.09){1}{\line(1,0){0.49}}

\linethickness{1mm}
\put(36.88,12.5){\line(0,1){23.12}}
\end{picture}
\caption{The double twist relation}\label{ch5doubletwist}
\end{figure}
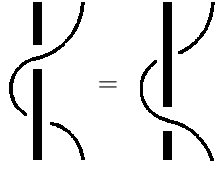

\nl As for the Kauffman tangle algebra described in
\cite{MorWas}, we
define an algebra for the tangles of type $\D^{(1)}$. Let
${\U}_n^{(1)}$ be the monoid of $(n,n)$-tangles of type $\D^{(1)}$
modulo regular isotopy to the right of the pole and Reidemeister~II
for each strand interacting with the pole. 
Similarly, we define the submonoids $\U_n^{(2)}$ and $\U_n$
corresponding to $(n,n)$-tangles of type $\D^{(2)}$ and $\D$, respectively.
Note that the product of two tangles of each type is again of the
same type.

As the tangles contain a pole of order two, the tangles in the monoid
$\U_n^{(1)}$ satisfy the {\em double twist relation}, as shown in
Figure~\ref{ch5doubletwist}. To see this, just compose both sides of
Figure~\ref{ch5doubletwist2} with the right hand side of
Figure~\ref{ch5doubletwist} and use Reidemeister~II around the pole.

Products of tangles with an even number of pole twists have an
even number of pole twists, and the relations corresponding to the
pole of order two preserves this. Therefore the composition of two tangles
of type $\D^{(2)}$ is again of type $\D^{(2)}$.

We now introduce the tangle algebra $\KT(\D_n)^{(1)}$ as a quotient of the
monoid algebra $R[{\U}_n^{(1)}]$. Later, in Definition \ref{df:KTDn}, the
algebra $\KT(\D_n)$ of our prime interest will appear as a subalgebra of
$\KT(\D_n)^{(1)}$.

\begin{Def}\label{ch5main}
\label{df:KTDn}
\rm The {\em tangle algebra} $\KT(\D_n)^{(1)}$ over $R$ is the quotient
algebra obtained from the monoid algebra $R[ {\U}_n^{(1)}]$ by factoring out
the following seven relations.  Here, the pictures indicate tangles which
differ only in the region shown.

\begin{enumerate}[{\rm (i)}]
\item The Kauffman skein relation

\begin{center}
\unitlength 0.5mm

\end{center}

\item The commuting relation

\begin{center}
\unitlength 0.6mm
%
\end{center}

\item The idempotent relation
$$T \cup O = \delta T,$$ where $T \cup O$ is the union of a tangle $T$ and a
closed loop $O$ having no crossings with $T$, no self-intersections, and no
pole twists.

\item The first pole-related self-intersection relation

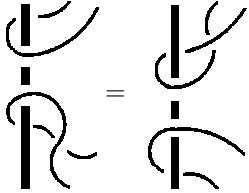
\begin{center}
\unitlength 0.6mm
%
\end{center}

\item The second pole-related self-intersection relation

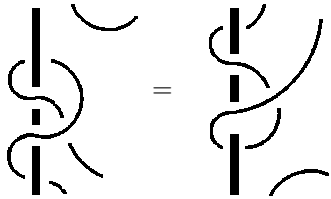
\begin{center}
\unitlength 0.8mm
%
\end{center}

\item The first closed pole loop relation

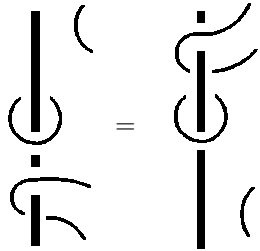
\begin{center}
\unitlength 0.8mm
%
\end{center}
\end{enumerate}

The relations generate a two-sided ideal in $R[\U_n^{(1)}]$.  Thus,
composition of tangles of type $\D^{(1)}$ induces an associative bilinear
multiplication on $\KT(\D_n)^{(1)}$, making $\KT(\D_n)^{(1)}$ an algebra over
$R$.  The subalgebra of $\KT(\D_n)^{(1)}$ generated by all tangles of type
$\D^{(2)}$ is denoted $\KT(\D_n)^{(2)}$.  

For $n\ge 0$ define the {\em Kauffman tangle algebra of type} $\D$ on $n$
nodes, denoted $\KT(\D_n)$, to be the subalgebra of $\KT(\D_n)^{(2)}$
generated by all tangles of type $\D$. 
\end{Def}

\begin{Remarks}\label{rmk:epsetal}
\rm(i). 
Since the relations (i)--(vii) are homogeneous with respect to the parity
of the number of pole twists, $\KT(\D_n)^{(2)}$ is the $R$-linear
span of tangles of type $\D^{(2)}$.  Also, the Kauffman tangle algebra
$\KT(\D_n)$ is easily seen to be the linear span of all tangles in
$\KT(\D_n)^{(2)}$ of type $\D$.

\rm(ii).
In using the pole-related self-intersection relations, care must be taken
to get the correct over crossings versus under crossings.  In both cases, the
correct diagram is obtained by turning the diagram upside down (i.e., turning
the paper 180 degrees around the horizontal axis perpendicular to the pole) so
that the over crossing at the bottom of the left hand side of (v) becomes an
under crossing at the top of the right hand side.  Recall that for pole twists
without a self-intersection, there is no distinction because the pole has
order two.

(iii).
If $S$ is a ring containing $R$ in which $m$ is invertible, then the relations
(v), (vi), and (vii) of Definition \ref{ch5main} for $\KT(\D_n)^{(1)}\otimes_R
S$ follow from the others, see \cite{DAHG} for details.

(iv).
For $n\ge1$, the algebras $\KT(\D_n)$ have the desirable property that
$\KT(\D_{n-1})$ is a natural subalgebra.  In fact, 
addition of a
strand without crossings to the right side of the tangle determines a natural
homomorphism $i : \KT(\D_{n-1})^{(1)} \to \KT(\D_n)^{(1)}$.  We also have a map
$\eps : \KT(\D_n)^{(1)} \to \KT(\D_{n-1})^{(1)}$ defined on tangles $T$ by
$$\eps(T) = \delta^{-1}cl_n(T),$$ where $cl_n : \KT(\D_n)^{(1)} \to
\KT(\D_{n-1})^{(1)}$ is the map defined by connecting the two endpoints in
$K$ on the right, viz.\ $(n,0,0)$ and $(n,0,1)$, of an $(n,n)$-tangle
by a strand with no crossings, self-intersections or pole twists,
see Figure~\ref{pic:cln}. These maps obviously respect regular isotopy and the
defining relations of $\KT(\D_n)^{(1)}$. 
As $\eps \circ i(T) = T$ for $T \in \KT(\D_{n-1})^{(1)}$, we can regard
$\KT(\D_{n-1})^{(1)}$ as a subalgebra of $\KT(\D_{n})^{(1)}$.
It is easily seen that suitable
restrictions lead to embeddings of $\KT(\D_{n-1})^{(2)} $
into $\KT(\D_n)^{(2)}$ and of
$\KT(\D_{n-1})$ into $\KT(\D_n)$.

\begin{figure}[hbtp]
\unitlength 1mm
\unitlength 1mm
\begin{picture}(44.4,20.01)(0,0)
\linethickness{0.3mm}
\put(44,5.63){\line(0,1){12.37}}
\put(8.75,13.75){\makebox(0,0)[cc]{$cl_n$ :}}

\linethickness{1mm}
\put(19.87,2.94){\line(0,1){17.06}}
\linethickness{0.3mm}
\put(16.25,18.05){\line(1,0){23.75}}
\linethickness{0.3mm}
\put(16.25,5.68){\line(0,1){12.37}}
\linethickness{0.3mm}
\put(39.99,5.68){\line(0,1){12.37}}
\linethickness{0.3mm}
\put(16.25,5.68){\line(1,0){23.75}}
\put(27.49,11.88){\makebox(0,0)[cc]{T}}

\linethickness{0.3mm}
\multiput(43.74,18.43)(0.09,-0.14){3}{\line(0,-1){0.14}}
\multiput(43.43,18.82)(0.08,-0.1){4}{\line(0,-1){0.1}}
\multiput(43.07,19.16)(0.09,-0.09){4}{\line(1,0){0.09}}
\multiput(42.67,19.45)(0.1,-0.07){4}{\line(1,0){0.1}}
\multiput(42.23,19.68)(0.15,-0.08){3}{\line(1,0){0.15}}
\multiput(41.77,19.86)(0.23,-0.09){2}{\line(1,0){0.23}}
\multiput(41.28,19.97)(0.49,-0.11){1}{\line(1,0){0.49}}
\multiput(40.79,20.01)(0.5,-0.04){1}{\line(1,0){0.5}}
\multiput(40.29,19.99)(0.5,0.02){1}{\line(1,0){0.5}}
\multiput(39.8,19.89)(0.49,0.09){1}{\line(1,0){0.49}}
\multiput(39.33,19.74)(0.24,0.08){2}{\line(1,0){0.24}}
\multiput(38.88,19.52)(0.15,0.07){3}{\line(1,0){0.15}}
\multiput(38.47,19.24)(0.14,0.09){3}{\line(1,0){0.14}}
\multiput(38.1,18.91)(0.09,0.08){4}{\line(1,0){0.09}}
\multiput(37.77,18.54)(0.08,0.09){4}{\line(0,1){0.09}}
\multiput(37.5,18.12)(0.09,0.14){3}{\line(0,1){0.14}}

\linethickness{0.3mm}
\multiput(38.12,5.62)(0.09,-0.14){3}{\line(0,-1){0.14}}
\multiput(38.38,5.19)(0.08,-0.1){4}{\line(0,-1){0.1}}
\multiput(38.7,4.81)(0.09,-0.08){4}{\line(1,0){0.09}}
\multiput(39.07,4.48)(0.14,-0.09){3}{\line(1,0){0.14}}
\multiput(39.5,4.21)(0.23,-0.1){2}{\line(1,0){0.23}}
\multiput(39.95,4.02)(0.24,-0.06){2}{\line(1,0){0.24}}
\multiput(40.44,3.89)(0.5,-0.05){1}{\line(1,0){0.5}}
\multiput(40.94,3.85)(0.5,0.03){1}{\line(1,0){0.5}}
\multiput(41.44,3.88)(0.49,0.11){1}{\line(1,0){0.49}}
\multiput(41.92,3.99)(0.23,0.09){2}{\line(1,0){0.23}}
\multiput(42.39,4.17)(0.14,0.08){3}{\line(1,0){0.14}}
\multiput(42.82,4.43)(0.1,0.08){4}{\line(1,0){0.1}}
\multiput(43.2,4.75)(0.08,0.09){4}{\line(0,1){0.09}}
\multiput(43.53,5.12)(0.09,0.14){3}{\line(0,1){0.14}}
\multiput(43.8,5.54)(0.1,0.23){2}{\line(0,1){0.23}}
\end{picture}
\caption{The closure of the rightmost strand of a tangle}
\label{pic:cln}
\end{figure}
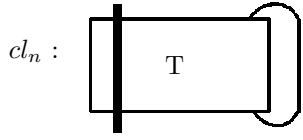

\end{Remarks}

\np We now derive a number of additional relations from the defining relations
concerning small regions of the tangle diagrams containing a part of the
pole. These relations will prove to be extremely useful in the full
understanding of these algebras as they describe the
interaction between the pole and the other strands of the tangles.

The commuting relation (ii) no longer holds if the upper crossing at each side
is changed. Other variations however do hold:

\begin{Lm}\label{lm:crossChanges}
The crossings in (ii), (v), and (vi)
of Definition \ref{ch5main} are all positive.
These relations also hold if the
signs are changed to negative as follows.
\begin{itemize}
\item For (ii): an upper crossing at one side and a lower crossing at the
other side of the equation or all four crossings.
\item For (v) and (vi): both crossings.
\end{itemize}
\end{Lm}

\nl\begin{proof} The
first statement is evident.  For (v),
this follows by application of (i) to the upper crossings at
both sides and identification of the
terms with coefficients $m$ and $-m$ by use of (vii). For (vi), the analogous
procedure works with the additional use of Reidemeister II.

For (ii), application of (i) to
the upper crossing at the left hand side and the lower crossing at the right
hand side and subsequent identification of the tangles with the same
coefficients by means of Reidemeister II and two invocations of (v) will lead
to the version of (ii) in which two crossings have changed signs.
The other version, with left hand lower crossing and right hand upper crossing
changed, is proved similarly.

For all four crossings, application of (i) to all four crossings leads to an
identity of the required kind after suitable application of Reidemeister II,
(vii), (vi), and the newly obtained version of (v) to the terms that are
multiples of $m$.  Alternatively, two applications of the first kind gives the
result.
\end{proof}

\np
One of the newly obtained versions of (ii)
is given in Figure \ref{pic:commuting2rel}.

\begin{figure}[htpb]
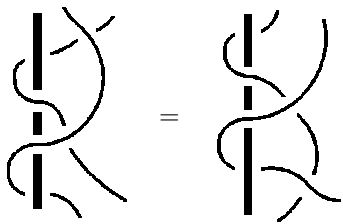

\unitlength 0.7mm

\caption{A second commuting relation, cf.~Lemma \ref{lm:crossChanges}}\label{pic:commuting2rel}
\end{figure}

\np Relation (iii) of Definition \ref{ch5main} shows that in general a
self-intersection of a strand can be replaced by $l^{\pm 1}$. This, however,
is not the case when the strand twists around the pole before intersecting
itself.  In contrast with the self-intersection relation, the two pole-related
self-intersection relations, (v) and (vi), preserve a self-intersection albeit
that the self-intersection may change strands.  The strand involved twists
around the pole and next crosses itself at the part of the strand before the
twist. This combination of twist and self-intersection is called a {\em
pole-related self-intersection}.

By the first pole-related self-intersection relation a twist of that type can
be moved to neighboring twists and then to others.  By the second pole-related
self-intersection relation, a twist of that type can be moved to any segment
of a strand which is accesible to the pole.  This means there is an
unobstructed region between the segment and the pole.  This shows that such a
{pole-related self-intersection} can be moved freely to many strands in
the tangle.  This is true even for a strand not twisting around the pole
provided a segment is accessible to the pole.  It can be given two pole twists
by use of the double twist relation of
Figure~\ref{ch5doubletwist2} (read backwards). Now the pole-related
self-intersection can be moved to this strand using one of the pole-related
self-intersection relations.
This leads to the following observation.

\begin{Remark}\label{rmk:onetwist}
\rm If a tangle has more than one pole-related self-intersection, then by (v)
and (vi) they can all be moved to a single strand. The tangle obtained in this
way with more than one pole-related self-intersection can be rewritten as a
linear combination of tangles with fewer pole-related self-intersections using
the Kauffman skein relations. Thus each element of $\KT(\D_n)^{(1)}$ is a
linear combination of tangles with at most one pole-related self-intersection.
\end{Remark}

\np While working with tangles of type $\D$, we will encounter closed loops
twisting around the pole. Here are some relations involving these loops.

\begin{Prop} \label{prop:addrel}
The tangles of $\KT(\D_n)^{(1)}$ satisfy the following relations.

\begin{enumerate}[{\rm  (i)}] 
\item The second closed pole loop relation,

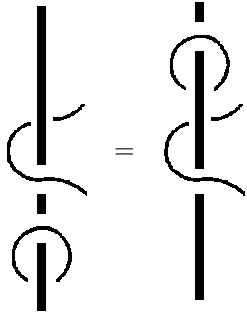
\begin{center}
\unitlength 0.8mm

\end{center}

\item The third closed pole loop relation,
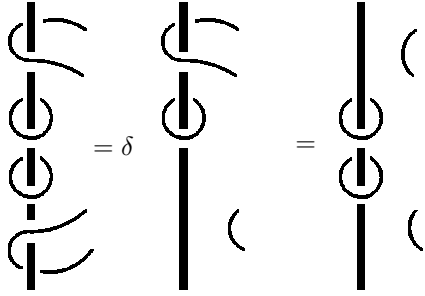
\begin{center}
\unitlength 0.65mm
%
\end{center}
\end{enumerate}
Again, the pictures indicate tangles which differ only in the region shown.
\end{Prop}

\nl
\begin{proof} (i). 
Using Reidemeister II, deform the partial strand and the closed loop in such
a way that, after its pole twist, the partial strand has two over crossings
with the closed loop. Apply the second commuting relation, Lemma
\ref{lm:crossChanges}, to a large enough region containing the closed loop
crossing the partial strand.  Notice the over crossings become under crossings
and again use Reidemeister II to shrink the loop so it does not intersect the
strands.  The loop is now on the other side of the partial strands.  This
gives (i).

\nl (ii).  This follows from the first closed pole loop relation (vii),
applied to partial strands.  To get the second equality, consider the bottom
closed loop on the right hand diagram to be a twist around the pole joined to
an arc to make it a closed loop.  After applying (vii), there is an isolated
closed loop not around the pole which contributes $\delta $ to the middle
picture.  To see that the left hand diagram is equal to the right diagram
in $\KT(\D_n)^{(1)}$, use
the second closed pole loop relation in the left hand diagram to put the bottom
closed twist around the pole below the lower partial strand.  Now distort the
lower twist around the pole so that it has a segment curving upwards before
twisting around the pole so as to be able to apply (vii).  Next use (vii) to
remove the twist at the top and give two twists around the pole to this lower
twist.  As the pole has order two, this is the diagram at the right and
so they are equal in $\KT(\D_n)^{(1)}$.
\end{proof}

\begin{figure}[htbp]
\unitlength 1mm
\begin{picture}(60.33,39)(0,0)
\put(6,2){\makebox(0,0)[cc]{$\Theta$}}

\put(31,2){\makebox(0,0)[cc]{$\Xi^+$}}

\put(56,2){\makebox(0,0)[cc]{$\Xi^{-}$}}

\linethickness{0.3mm}
\multiput(8,7.64)(0.23,0.1){2}{\line(1,0){0.23}}
\multiput(8.46,7.83)(0.15,0.08){3}{\line(1,0){0.15}}
\multiput(8.9,8.07)(0.14,0.09){3}{\line(1,0){0.14}}
\multiput(9.32,8.34)(0.1,0.08){4}{\line(1,0){0.1}}
\multiput(9.71,8.64)(0.09,0.08){4}{\line(1,0){0.09}}
\multiput(10.07,8.98)(0.08,0.09){4}{\line(0,1){0.09}}
\multiput(10.41,9.35)(0.08,0.1){4}{\line(0,1){0.1}}
\multiput(10.71,9.75)(0.09,0.14){3}{\line(0,1){0.14}}
\multiput(10.97,10.17)(0.08,0.15){3}{\line(0,1){0.15}}
\multiput(11.2,10.61)(0.09,0.23){2}{\line(0,1){0.23}}
\multiput(11.39,11.07)(0.07,0.24){2}{\line(0,1){0.24}}
\multiput(11.54,11.55)(0.11,0.49){1}{\line(0,1){0.49}}
\multiput(11.64,12.04)(0.06,0.49){1}{\line(0,1){0.49}}
\multiput(11.7,12.53)(0.02,0.5){1}{\line(0,1){0.5}}
\multiput(11.7,13.53)(0.02,-0.5){1}{\line(0,-1){0.5}}
\multiput(11.63,14.02)(0.07,-0.49){1}{\line(0,-1){0.49}}
\multiput(11.52,14.51)(0.11,-0.49){1}{\line(0,-1){0.49}}
\multiput(11.37,14.98)(0.08,-0.24){2}{\line(0,-1){0.24}}
\multiput(11.18,15.44)(0.1,-0.23){2}{\line(0,-1){0.23}}
\multiput(10.95,15.88)(0.08,-0.15){3}{\line(0,-1){0.15}}
\multiput(10.68,16.3)(0.09,-0.14){3}{\line(0,-1){0.14}}
\multiput(10.38,16.69)(0.08,-0.1){4}{\line(0,-1){0.1}}
\multiput(10.04,17.06)(0.08,-0.09){4}{\line(0,-1){0.09}}
\multiput(9.67,17.4)(0.09,-0.08){4}{\line(1,0){0.09}}
\multiput(9.28,17.7)(0.1,-0.08){4}{\line(1,0){0.1}}
\multiput(8.85,17.97)(0.14,-0.09){3}{\line(1,0){0.14}}
\multiput(8.41,18.19)(0.15,-0.08){3}{\line(1,0){0.15}}
\multiput(7.95,18.38)(0.23,-0.1){2}{\line(1,0){0.23}}
\multiput(7.48,18.53)(0.24,-0.07){2}{\line(1,0){0.24}}
\multiput(6.99,18.64)(0.49,-0.11){1}{\line(1,0){0.49}}
\multiput(6.5,18.71)(0.49,-0.06){1}{\line(1,0){0.49}}
\multiput(6,18.73)(0.5,-0.02){1}{\line(1,0){0.5}}
\multiput(5.5,18.71)(0.5,0.02){1}{\line(1,0){0.5}}
\multiput(5.01,18.64)(0.49,0.06){1}{\line(1,0){0.49}}
\multiput(4.52,18.53)(0.49,0.11){1}{\line(1,0){0.49}}
\multiput(4.05,18.38)(0.24,0.07){2}{\line(1,0){0.24}}
\multiput(3.59,18.19)(0.23,0.1){2}{\line(1,0){0.23}}
\multiput(3.15,17.97)(0.15,0.08){3}{\line(1,0){0.15}}
\multiput(2.72,17.7)(0.14,0.09){3}{\line(1,0){0.14}}
\multiput(2.33,17.4)(0.1,0.08){4}{\line(1,0){0.1}}
\multiput(1.96,17.06)(0.09,0.08){4}{\line(1,0){0.09}}
\multiput(1.62,16.69)(0.08,0.09){4}{\line(0,1){0.09}}
\multiput(1.32,16.3)(0.08,0.1){4}{\line(0,1){0.1}}
\multiput(1.05,15.88)(0.09,0.14){3}{\line(0,1){0.14}}
\multiput(0.82,15.44)(0.08,0.15){3}{\line(0,1){0.15}}
\multiput(0.63,14.98)(0.1,0.23){2}{\line(0,1){0.23}}
\multiput(0.48,14.51)(0.08,0.24){2}{\line(0,1){0.24}}
\multiput(0.37,14.02)(0.11,0.49){1}{\line(0,1){0.49}}
\multiput(0.3,13.53)(0.07,0.49){1}{\line(0,1){0.49}}
\multiput(0.28,13.03)(0.02,0.5){1}{\line(0,1){0.5}}
\multiput(0.28,13.03)(0.02,-0.5){1}{\line(0,-1){0.5}}
\multiput(0.3,12.53)(0.06,-0.49){1}{\line(0,-1){0.49}}
\multiput(0.36,12.04)(0.11,-0.49){1}{\line(0,-1){0.49}}
\multiput(0.46,11.55)(0.07,-0.24){2}{\line(0,-1){0.24}}
\multiput(0.61,11.07)(0.09,-0.23){2}{\line(0,-1){0.23}}
\multiput(0.8,10.61)(0.08,-0.15){3}{\line(0,-1){0.15}}
\multiput(1.03,10.17)(0.09,-0.14){3}{\line(0,-1){0.14}}
\multiput(1.29,9.75)(0.08,-0.1){4}{\line(0,-1){0.1}}
\multiput(1.59,9.35)(0.08,-0.09){4}{\line(0,-1){0.09}}
\multiput(1.93,8.98)(0.09,-0.08){4}{\line(1,0){0.09}}
\multiput(2.29,8.64)(0.1,-0.08){4}{\line(1,0){0.1}}
\multiput(2.68,8.34)(0.14,-0.09){3}{\line(1,0){0.14}}
\multiput(3.1,8.07)(0.15,-0.08){3}{\line(1,0){0.15}}
\multiput(3.54,7.83)(0.23,-0.1){2}{\line(1,0){0.23}}

\linethickness{1mm}
\put(6,21){\line(0,1){11}}
\linethickness{0.3mm}
\multiput(8,23.37)(0.23,0.1){2}{\line(1,0){0.23}}
\multiput(8.46,23.56)(0.15,0.08){3}{\line(1,0){0.15}}
\multiput(8.9,23.8)(0.14,0.09){3}{\line(1,0){0.14}}
\multiput(9.32,24.07)(0.1,0.08){4}{\line(1,0){0.1}}
\multiput(9.71,24.37)(0.09,0.08){4}{\line(1,0){0.09}}
\multiput(10.07,24.71)(0.08,0.09){4}{\line(0,1){0.09}}
\multiput(10.41,25.08)(0.08,0.1){4}{\line(0,1){0.1}}
\multiput(10.71,25.48)(0.09,0.14){3}{\line(0,1){0.14}}
\multiput(10.97,25.9)(0.08,0.15){3}{\line(0,1){0.15}}
\multiput(11.2,26.34)(0.09,0.23){2}{\line(0,1){0.23}}
\multiput(11.39,26.8)(0.07,0.24){2}{\line(0,1){0.24}}
\multiput(11.54,27.28)(0.11,0.49){1}{\line(0,1){0.49}}
\multiput(11.64,27.77)(0.06,0.49){1}{\line(0,1){0.49}}
\multiput(11.7,28.26)(0.02,0.5){1}{\line(0,1){0.5}}
\multiput(11.7,29.26)(0.02,-0.5){1}{\line(0,-1){0.5}}
\multiput(11.63,29.75)(0.07,-0.49){1}{\line(0,-1){0.49}}
\multiput(11.52,30.24)(0.11,-0.49){1}{\line(0,-1){0.49}}
\multiput(11.37,30.71)(0.08,-0.24){2}{\line(0,-1){0.24}}
\multiput(11.18,31.17)(0.1,-0.23){2}{\line(0,-1){0.23}}
\multiput(10.95,31.61)(0.08,-0.15){3}{\line(0,-1){0.15}}
\multiput(10.68,32.03)(0.09,-0.14){3}{\line(0,-1){0.14}}
\multiput(10.38,32.42)(0.08,-0.1){4}{\line(0,-1){0.1}}
\multiput(10.04,32.79)(0.08,-0.09){4}{\line(0,-1){0.09}}
\multiput(9.67,33.13)(0.09,-0.08){4}{\line(1,0){0.09}}
\multiput(9.28,33.43)(0.1,-0.08){4}{\line(1,0){0.1}}
\multiput(8.85,33.7)(0.14,-0.09){3}{\line(1,0){0.14}}
\multiput(8.41,33.92)(0.15,-0.08){3}{\line(1,0){0.15}}
\multiput(7.95,34.11)(0.23,-0.1){2}{\line(1,0){0.23}}
\multiput(7.48,34.26)(0.24,-0.07){2}{\line(1,0){0.24}}
\multiput(6.99,34.37)(0.49,-0.11){1}{\line(1,0){0.49}}
\multiput(6.5,34.44)(0.49,-0.06){1}{\line(1,0){0.49}}
\multiput(6,34.46)(0.5,-0.02){1}{\line(1,0){0.5}}
\multiput(5.5,34.44)(0.5,0.02){1}{\line(1,0){0.5}}
\multiput(5.01,34.37)(0.49,0.06){1}{\line(1,0){0.49}}
\multiput(4.52,34.26)(0.49,0.11){1}{\line(1,0){0.49}}
\multiput(4.05,34.11)(0.24,0.07){2}{\line(1,0){0.24}}
\multiput(3.59,33.92)(0.23,0.1){2}{\line(1,0){0.23}}
\multiput(3.15,33.7)(0.15,0.08){3}{\line(1,0){0.15}}
\multiput(2.72,33.43)(0.14,0.09){3}{\line(1,0){0.14}}
\multiput(2.33,33.13)(0.1,0.08){4}{\line(1,0){0.1}}
\multiput(1.96,32.79)(0.09,0.08){4}{\line(1,0){0.09}}
\multiput(1.62,32.42)(0.08,0.09){4}{\line(0,1){0.09}}
\multiput(1.32,32.03)(0.08,0.1){4}{\line(0,1){0.1}}
\multiput(1.05,31.61)(0.09,0.14){3}{\line(0,1){0.14}}
\multiput(0.82,31.17)(0.08,0.15){3}{\line(0,1){0.15}}
\multiput(0.63,30.71)(0.1,0.23){2}{\line(0,1){0.23}}
\multiput(0.48,30.24)(0.08,0.24){2}{\line(0,1){0.24}}
\multiput(0.37,29.75)(0.11,0.49){1}{\line(0,1){0.49}}
\multiput(0.3,29.26)(0.07,0.49){1}{\line(0,1){0.49}}
\multiput(0.28,28.76)(0.02,0.5){1}{\line(0,1){0.5}}
\multiput(0.28,28.76)(0.02,-0.5){1}{\line(0,-1){0.5}}
\multiput(0.3,28.26)(0.06,-0.49){1}{\line(0,-1){0.49}}
\multiput(0.36,27.77)(0.11,-0.49){1}{\line(0,-1){0.49}}
\multiput(0.46,27.28)(0.07,-0.24){2}{\line(0,-1){0.24}}
\multiput(0.61,26.8)(0.09,-0.23){2}{\line(0,-1){0.23}}
\multiput(0.8,26.34)(0.08,-0.15){3}{\line(0,-1){0.15}}
\multiput(1.03,25.9)(0.09,-0.14){3}{\line(0,-1){0.14}}
\multiput(1.29,25.48)(0.08,-0.1){4}{\line(0,-1){0.1}}
\multiput(1.59,25.08)(0.08,-0.09){4}{\line(0,-1){0.09}}
\multiput(1.93,24.71)(0.09,-0.08){4}{\line(1,0){0.09}}
\multiput(2.29,24.37)(0.1,-0.08){4}{\line(1,0){0.1}}
\multiput(2.68,24.07)(0.14,-0.09){3}{\line(1,0){0.14}}
\multiput(3.1,23.8)(0.15,-0.08){3}{\line(1,0){0.15}}
\multiput(3.54,23.56)(0.23,-0.1){2}{\line(1,0){0.23}}

\linethickness{1mm}
\put(6,37){\line(0,1){2}}
\linethickness{1mm}
\put(5.84,4.38){\line(0,1){12.74}}
\linethickness{1mm}
\put(30,5){\line(0,1){4}}
\linethickness{1mm}
\put(30,13){\line(0,1){7}}
\linethickness{0.3mm}
\multiput(26,13)(0.07,-0.1){4}{\line(0,-1){0.1}}
\multiput(26.29,12.58)(0.08,-0.1){4}{\line(0,-1){0.1}}
\multiput(26.62,12.2)(0.09,-0.09){4}{\line(1,0){0.09}}
\multiput(26.99,11.85)(0.1,-0.08){4}{\line(1,0){0.1}}
\multiput(27.4,11.54)(0.14,-0.09){3}{\line(1,0){0.14}}
\multiput(27.83,11.28)(0.15,-0.07){3}{\line(1,0){0.15}}
\multiput(28.28,11.06)(0.24,-0.09){2}{\line(1,0){0.24}}
\multiput(28.76,10.89)(0.25,-0.06){2}{\line(1,0){0.25}}
\multiput(29.25,10.76)(0.5,-0.07){1}{\line(1,0){0.5}}
\multiput(29.76,10.69)(0.51,-0.02){1}{\line(1,0){0.51}}
\multiput(30.26,10.67)(0.51,0.03){1}{\line(1,0){0.51}}
\multiput(30.77,10.7)(0.5,0.08){1}{\line(1,0){0.5}}
\multiput(31.27,10.78)(0.24,0.07){2}{\line(1,0){0.24}}
\multiput(31.76,10.91)(0.24,0.09){2}{\line(1,0){0.24}}
\multiput(32.23,11.1)(0.15,0.08){3}{\line(1,0){0.15}}
\multiput(32.69,11.33)(0.14,0.09){3}{\line(1,0){0.14}}
\multiput(33.11,11.6)(0.1,0.08){4}{\line(1,0){0.1}}
\multiput(33.51,11.92)(0.09,0.09){4}{\line(1,0){0.09}}
\multiput(33.87,12.27)(0.08,0.1){4}{\line(0,1){0.1}}
\multiput(34.2,12.66)(0.07,0.11){4}{\line(0,1){0.11}}
\multiput(34.48,13.08)(0.08,0.15){3}{\line(0,1){0.15}}
\multiput(34.72,13.53)(0.1,0.23){2}{\line(0,1){0.23}}
\multiput(34.91,14)(0.07,0.24){2}{\line(0,1){0.24}}
\multiput(35.05,14.49)(0.09,0.5){1}{\line(0,1){0.5}}
\multiput(35.15,14.99)(0.04,0.51){1}{\line(0,1){0.51}}
\multiput(35.18,16)(0.01,-0.51){1}{\line(0,-1){0.51}}
\multiput(35.11,16.51)(0.06,-0.5){1}{\line(0,-1){0.5}}
\multiput(35,17)(0.11,-0.49){1}{\line(0,-1){0.49}}

\linethickness{0.3mm}
\multiput(27.56,17.77)(0.15,0.08){3}{\line(1,0){0.15}}
\multiput(27.15,17.49)(0.1,0.07){4}{\line(1,0){0.1}}
\multiput(26.77,17.16)(0.09,0.08){4}{\line(1,0){0.09}}
\multiput(26.45,16.78)(0.08,0.09){4}{\line(0,1){0.09}}
\multiput(26.18,16.36)(0.09,0.14){3}{\line(0,1){0.14}}
\multiput(25.96,15.91)(0.07,0.15){3}{\line(0,1){0.15}}
\multiput(25.81,15.44)(0.08,0.24){2}{\line(0,1){0.24}}
\multiput(25.72,14.95)(0.09,0.49){1}{\line(0,1){0.49}}
\multiput(25.69,14.45)(0.03,0.5){1}{\line(0,1){0.5}}
\multiput(25.69,14.45)(0.04,-0.5){1}{\line(0,-1){0.5}}
\multiput(25.73,13.96)(0.1,-0.49){1}{\line(0,-1){0.49}}
\multiput(25.83,13.47)(0.08,-0.23){2}{\line(0,-1){0.23}}

\linethickness{0.3mm}
\multiput(32,18)(0.49,0.08){1}{\line(1,0){0.49}}
\multiput(32.49,18.08)(0.24,0.06){2}{\line(1,0){0.24}}
\multiput(32.98,18.21)(0.24,0.09){2}{\line(1,0){0.24}}
\multiput(33.45,18.39)(0.15,0.07){3}{\line(1,0){0.15}}
\multiput(33.9,18.6)(0.14,0.08){3}{\line(1,0){0.14}}
\multiput(34.33,18.85)(0.1,0.07){4}{\line(1,0){0.1}}
\multiput(34.74,19.15)(0.09,0.08){4}{\line(1,0){0.09}}
\multiput(35.12,19.47)(0.09,0.09){4}{\line(0,1){0.09}}
\multiput(35.47,19.83)(0.08,0.1){4}{\line(0,1){0.1}}
\multiput(35.78,20.22)(0.09,0.14){3}{\line(0,1){0.14}}
\multiput(36.06,20.64)(0.08,0.15){3}{\line(0,1){0.15}}
\multiput(36.3,21.08)(0.1,0.23){2}{\line(0,1){0.23}}
\multiput(36.49,21.54)(0.08,0.24){2}{\line(0,1){0.24}}
\multiput(36.65,22.02)(0.11,0.49){1}{\line(0,1){0.49}}
\multiput(36.76,22.5)(0.07,0.5){1}{\line(0,1){0.5}}
\multiput(36.83,23)(0.02,0.5){1}{\line(0,1){0.5}}
\multiput(36.83,24)(0.02,-0.5){1}{\line(0,-1){0.5}}
\multiput(36.76,24.5)(0.07,-0.5){1}{\line(0,-1){0.5}}
\multiput(36.65,24.98)(0.11,-0.49){1}{\line(0,-1){0.49}}
\multiput(36.49,25.46)(0.08,-0.24){2}{\line(0,-1){0.24}}
\multiput(36.3,25.92)(0.1,-0.23){2}{\line(0,-1){0.23}}
\multiput(36.06,26.36)(0.08,-0.15){3}{\line(0,-1){0.15}}
\multiput(35.78,26.78)(0.09,-0.14){3}{\line(0,-1){0.14}}
\multiput(35.47,27.17)(0.08,-0.1){4}{\line(0,-1){0.1}}
\multiput(35.12,27.53)(0.09,-0.09){4}{\line(0,-1){0.09}}
\multiput(34.74,27.85)(0.09,-0.08){4}{\line(1,0){0.09}}
\multiput(34.33,28.15)(0.1,-0.07){4}{\line(1,0){0.1}}
\multiput(33.9,28.4)(0.14,-0.08){3}{\line(1,0){0.14}}
\multiput(33.45,28.61)(0.15,-0.07){3}{\line(1,0){0.15}}
\multiput(32.98,28.79)(0.24,-0.09){2}{\line(1,0){0.24}}
\multiput(32.49,28.92)(0.24,-0.06){2}{\line(1,0){0.24}}
\multiput(32,29)(0.49,-0.08){1}{\line(1,0){0.49}}

\linethickness{1mm}
\put(30,24.12){\line(0,1){14.88}}
\linethickness{0.3mm}
\multiput(27.55,28.8)(0.22,0.1){2}{\line(1,0){0.22}}
\multiput(27.13,28.54)(0.14,0.09){3}{\line(1,0){0.14}}
\multiput(26.75,28.23)(0.09,0.08){4}{\line(1,0){0.09}}
\multiput(26.42,27.87)(0.08,0.09){4}{\line(0,1){0.09}}
\multiput(26.14,27.46)(0.07,0.1){4}{\line(0,1){0.1}}
\multiput(25.91,27.02)(0.07,0.15){3}{\line(0,1){0.15}}
\multiput(25.75,26.56)(0.08,0.23){2}{\line(0,1){0.23}}
\multiput(25.65,26.08)(0.1,0.48){1}{\line(0,1){0.48}}
\multiput(25.62,25.59)(0.03,0.49){1}{\line(0,1){0.49}}
\multiput(25.62,25.59)(0.04,-0.49){1}{\line(0,-1){0.49}}
\multiput(25.66,25.09)(0.1,-0.48){1}{\line(0,-1){0.48}}
\multiput(25.76,24.61)(0.08,-0.23){2}{\line(0,-1){0.23}}
\multiput(25.93,24.15)(0.08,-0.15){3}{\line(0,-1){0.15}}
\multiput(26.16,23.71)(0.07,-0.1){4}{\line(0,-1){0.1}}
\multiput(26.44,23.31)(0.08,-0.09){4}{\line(0,-1){0.09}}
\multiput(26.78,22.95)(0.1,-0.08){4}{\line(1,0){0.1}}
\multiput(27.16,22.65)(0.14,-0.08){3}{\line(1,0){0.14}}
\multiput(27.59,22.39)(0.23,-0.1){2}{\line(1,0){0.23}}
\multiput(28.04,22.2)(0.24,-0.07){2}{\line(1,0){0.24}}
\multiput(28.51,22.07)(0.49,-0.07){1}{\line(1,0){0.49}}

\linethickness{0.3mm}
\multiput(32.67,20.4)(0.08,-0.1){4}{\line(0,-1){0.1}}
\multiput(32.3,20.75)(0.09,-0.09){4}{\line(1,0){0.09}}
\multiput(31.9,21.07)(0.1,-0.08){4}{\line(1,0){0.1}}
\multiput(31.46,21.35)(0.14,-0.09){3}{\line(1,0){0.14}}
\multiput(31,21.58)(0.15,-0.08){3}{\line(1,0){0.15}}
\multiput(30.52,21.76)(0.24,-0.09){2}{\line(1,0){0.24}}
\multiput(30.02,21.89)(0.25,-0.07){2}{\line(1,0){0.25}}
\multiput(29.51,21.97)(0.51,-0.08){1}{\line(1,0){0.51}}
\multiput(29,22)(0.51,-0.03){1}{\line(1,0){0.51}}

\linethickness{1mm}
\put(54,5){\line(0,1){4}}
\linethickness{1mm}
\put(54,13){\line(0,1){7}}
\linethickness{0.3mm}
\multiput(50,13)(0.08,-0.1){4}{\line(0,-1){0.1}}
\multiput(50.31,12.61)(0.09,-0.09){4}{\line(0,-1){0.09}}
\multiput(50.66,12.26)(0.09,-0.08){4}{\line(1,0){0.09}}
\multiput(51.04,11.93)(0.1,-0.07){4}{\line(1,0){0.1}}
\multiput(51.44,11.64)(0.14,-0.08){3}{\line(1,0){0.14}}
\multiput(51.87,11.39)(0.15,-0.07){3}{\line(1,0){0.15}}
\multiput(52.32,11.18)(0.23,-0.08){2}{\line(1,0){0.23}}
\multiput(52.79,11.01)(0.24,-0.06){2}{\line(1,0){0.24}}
\multiput(53.27,10.89)(0.49,-0.08){1}{\line(1,0){0.49}}
\multiput(53.76,10.8)(0.5,-0.04){1}{\line(1,0){0.5}}
\multiput(54.25,10.77)(0.5,0.01){1}{\line(1,0){0.5}}
\multiput(54.75,10.78)(0.49,0.05){1}{\line(1,0){0.49}}
\multiput(55.25,10.83)(0.49,0.1){1}{\line(1,0){0.49}}
\multiput(55.73,10.93)(0.24,0.07){2}{\line(1,0){0.24}}
\multiput(56.21,11.07)(0.23,0.09){2}{\line(1,0){0.23}}
\multiput(56.67,11.25)(0.15,0.08){3}{\line(1,0){0.15}}
\multiput(57.11,11.48)(0.14,0.09){3}{\line(1,0){0.14}}
\multiput(57.54,11.74)(0.1,0.08){4}{\line(1,0){0.1}}
\multiput(57.93,12.04)(0.09,0.08){4}{\line(1,0){0.09}}
\multiput(58.3,12.38)(0.08,0.09){4}{\line(0,1){0.09}}
\multiput(58.63,12.75)(0.07,0.1){4}{\line(0,1){0.1}}
\multiput(58.93,13.14)(0.09,0.14){3}{\line(0,1){0.14}}
\multiput(59.19,13.57)(0.07,0.15){3}{\line(0,1){0.15}}
\multiput(59.42,14.01)(0.09,0.23){2}{\line(0,1){0.23}}
\multiput(59.6,14.47)(0.07,0.24){2}{\line(0,1){0.24}}
\multiput(59.74,14.95)(0.1,0.49){1}{\line(0,1){0.49}}
\multiput(59.83,15.44)(0.05,0.49){1}{\line(0,1){0.49}}
\multiput(59.88,15.93)(0.01,0.5){1}{\line(0,1){0.5}}
\multiput(59.85,16.93)(0.04,-0.5){1}{\line(0,-1){0.5}}
\multiput(59.77,17.42)(0.08,-0.49){1}{\line(0,-1){0.49}}
\multiput(59.64,17.9)(0.06,-0.24){2}{\line(0,-1){0.24}}
\multiput(59.47,18.36)(0.09,-0.23){2}{\line(0,-1){0.23}}
\multiput(59.25,18.81)(0.07,-0.15){3}{\line(0,-1){0.15}}
\multiput(59,19.24)(0.08,-0.14){3}{\line(0,-1){0.14}}

\linethickness{0.3mm}
\multiput(51.56,17.77)(0.15,0.08){3}{\line(1,0){0.15}}
\multiput(51.15,17.49)(0.1,0.07){4}{\line(1,0){0.1}}
\multiput(50.77,17.16)(0.09,0.08){4}{\line(1,0){0.09}}
\multiput(50.45,16.78)(0.08,0.09){4}{\line(0,1){0.09}}
\multiput(50.18,16.36)(0.09,0.14){3}{\line(0,1){0.14}}
\multiput(49.96,15.91)(0.07,0.15){3}{\line(0,1){0.15}}
\multiput(49.81,15.44)(0.08,0.24){2}{\line(0,1){0.24}}
\multiput(49.72,14.95)(0.09,0.49){1}{\line(0,1){0.49}}
\multiput(49.69,14.45)(0.03,0.5){1}{\line(0,1){0.5}}
\multiput(49.69,14.45)(0.04,-0.5){1}{\line(0,-1){0.5}}
\multiput(49.73,13.96)(0.1,-0.49){1}{\line(0,-1){0.49}}
\multiput(49.83,13.47)(0.08,-0.23){2}{\line(0,-1){0.23}}

\linethickness{0.3mm}
\multiput(59.62,21.74)(0.08,0.15){3}{\line(0,1){0.15}}
\multiput(59.86,22.19)(0.1,0.24){2}{\line(0,1){0.24}}
\multiput(60.06,22.66)(0.07,0.24){2}{\line(0,1){0.24}}
\multiput(60.2,23.15)(0.09,0.5){1}{\line(0,1){0.5}}
\multiput(60.3,23.65)(0.04,0.51){1}{\line(0,1){0.51}}
\multiput(60.31,24.67)(0.02,-0.51){1}{\line(0,-1){0.51}}
\multiput(60.24,25.17)(0.07,-0.51){1}{\line(0,-1){0.51}}
\multiput(60.12,25.67)(0.06,-0.25){2}{\line(0,-1){0.25}}
\multiput(59.94,26.15)(0.09,-0.24){2}{\line(0,-1){0.24}}
\multiput(59.71,26.6)(0.08,-0.15){3}{\line(0,-1){0.15}}
\multiput(59.44,27.03)(0.09,-0.14){3}{\line(0,-1){0.14}}
\multiput(59.12,27.43)(0.08,-0.1){4}{\line(0,-1){0.1}}
\multiput(58.76,27.79)(0.07,-0.07){5}{\line(0,-1){0.07}}
\multiput(58.36,28.11)(0.1,-0.08){4}{\line(1,0){0.1}}
\multiput(57.93,28.39)(0.14,-0.09){3}{\line(1,0){0.14}}
\multiput(57.48,28.62)(0.15,-0.08){3}{\line(1,0){0.15}}
\multiput(57,28.8)(0.24,-0.09){2}{\line(1,0){0.24}}
\multiput(56.5,28.93)(0.25,-0.06){2}{\line(1,0){0.25}}
\multiput(56,29)(0.5,-0.07){1}{\line(1,0){0.5}}

\linethickness{1mm}
\put(54,24.12){\line(0,1){14.88}}
\linethickness{0.3mm}
\multiput(51.55,28.8)(0.22,0.1){2}{\line(1,0){0.22}}
\multiput(51.13,28.54)(0.14,0.09){3}{\line(1,0){0.14}}
\multiput(50.75,28.23)(0.09,0.08){4}{\line(1,0){0.09}}
\multiput(50.42,27.87)(0.08,0.09){4}{\line(0,1){0.09}}
\multiput(50.14,27.46)(0.07,0.1){4}{\line(0,1){0.1}}
\multiput(49.91,27.02)(0.07,0.15){3}{\line(0,1){0.15}}
\multiput(49.75,26.56)(0.08,0.23){2}{\line(0,1){0.23}}
\multiput(49.65,26.08)(0.1,0.48){1}{\line(0,1){0.48}}
\multiput(49.62,25.59)(0.03,0.49){1}{\line(0,1){0.49}}
\multiput(49.62,25.59)(0.04,-0.49){1}{\line(0,-1){0.49}}
\multiput(49.66,25.09)(0.1,-0.48){1}{\line(0,-1){0.48}}
\multiput(49.76,24.61)(0.08,-0.23){2}{\line(0,-1){0.23}}
\multiput(49.93,24.15)(0.08,-0.15){3}{\line(0,-1){0.15}}
\multiput(50.16,23.71)(0.07,-0.1){4}{\line(0,-1){0.1}}
\multiput(50.44,23.31)(0.08,-0.09){4}{\line(0,-1){0.09}}
\multiput(50.78,22.95)(0.1,-0.08){4}{\line(1,0){0.1}}
\multiput(51.16,22.65)(0.14,-0.08){3}{\line(1,0){0.14}}
\multiput(51.59,22.39)(0.23,-0.1){2}{\line(1,0){0.23}}
\multiput(52.04,22.2)(0.24,-0.07){2}{\line(1,0){0.24}}
\multiput(52.51,22.07)(0.49,-0.07){1}{\line(1,0){0.49}}

\linethickness{0.3mm}
\multiput(58.72,19.65)(0.07,-0.1){4}{\line(0,-1){0.1}}
\multiput(58.4,20.04)(0.08,-0.1){4}{\line(0,-1){0.1}}
\multiput(58.06,20.4)(0.09,-0.09){4}{\line(0,-1){0.09}}
\multiput(57.68,20.73)(0.09,-0.08){4}{\line(1,0){0.09}}
\multiput(57.28,21.03)(0.1,-0.07){4}{\line(1,0){0.1}}
\multiput(56.86,21.29)(0.14,-0.09){3}{\line(1,0){0.14}}
\multiput(56.41,21.52)(0.15,-0.08){3}{\line(1,0){0.15}}
\multiput(55.95,21.71)(0.23,-0.09){2}{\line(1,0){0.23}}
\multiput(55.48,21.86)(0.24,-0.08){2}{\line(1,0){0.24}}
\multiput(54.99,21.97)(0.49,-0.11){1}{\line(1,0){0.49}}
\multiput(54.49,22.04)(0.49,-0.07){1}{\line(1,0){0.49}}
\multiput(54,22.07)(0.5,-0.03){1}{\line(1,0){0.5}}
\multiput(53.5,22.06)(0.5,0.01){1}{\line(1,0){0.5}}
\multiput(53,22)(0.5,0.06){1}{\line(1,0){0.5}}

\linethickness{0.3mm}
\multiput(55.87,18.62)(0.16,0.08){8}{\line(1,0){0.16}}
\end{picture}
\caption{Three $(0,0)$-tangles of type $\D^{(2)}$} \label{ch5KT0tangles}
\end{figure}
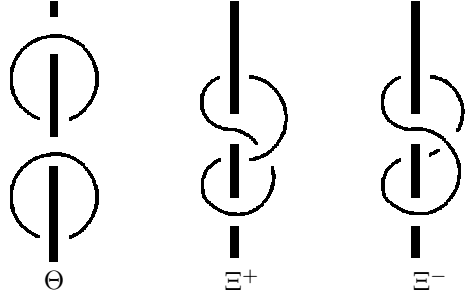

\nl The closed pole loop relations illustrate that when a tangle contains one
closed loop around the pole, all other pole twists can be moved freely between
all strands for which a segment is accessible to the pole. In particular, all
these twists can be moved to a single strand. By use of the double twist
relation, all but at most one twist can be removed.  So every tangle with
closed loops and with a pole twist can be transformed to a tangle
containing this closed loop with a twist around the pole and at most one other
strand with a pole twist.

Furthermore, as there are an even number of twists around
the pole in tangles from $\KT(\D_n)^{(2)}$,
we can assume there are two closed loops around the pole.

\np We now turn our attention to certain closed strands with pole twists which
we will use extensively. Denote by $\Theta$ the $(0,0)$-tangle of type
$\D^{(2)}$ consisting of only two separate loops each of which twists around
the pole, as shown in Figure~\ref{ch5KT0tangles}. Denote by $\Xi^+$ and
$\Xi^-$ the $(0,0)$-tangle of type $\D^{(2)}$ consisting of precisely one
closed loop with two pole twists and a positive, respectively, negative
self-intersection between these two twists.  These $(0,0)$-tangles have some
very useful properties.

\begin{Lm}\label{ch5KT0}
The $(0,0)$-tangles $\Theta$, $\Xi^{+}$, and $\Xi^{-}$ satisfy the
following relations in $\KT(\D_0)^{(2)}$.

\begin{eqnarray}
\Xi^+ - \Xi^- &=& m(\Theta-\delta) \label{ch5eQ1}, \\
(\Xi^+)^2 &=& \delta^2 - m\delta\Xi^+ + ml^{-1} \delta\Theta,
\label{ch5eQ2} \\
\Xi^+\Theta = \Theta\Xi^+ &=& \delta l^{-1}\Theta, \label{ch5eQ3} \\
\Theta^2 &=& \delta^2 \Theta. \label{ch5eQ4}
\end{eqnarray}
\end{Lm}

\begin{proof}
\nl Applying the Kauffman skein relation to the single crossing in
$\Xi^+$ gives $\Xi^+ + m\delta = \Xi^- + m\Theta$ and
$(\ref{ch5eQ1})$ follows from this, showing that $\Xi^-$ can be
expressed as a linear combination of $\Xi^+$ and $\Theta$.

\nl Observe that $(\Xi^+)^2$ is equal to $\delta$ times the
$(0,0)$-tangle containing one closed loop around the pole with two
pole-related self-intersections, as shown in Figure~\ref{ch5xixi1}. Property
$(\ref{ch5eQ2})$ is a direct result of applying the Kauffman
skein relation to one of the two pole-related self-intersections.

\begin{figure}[htbp]
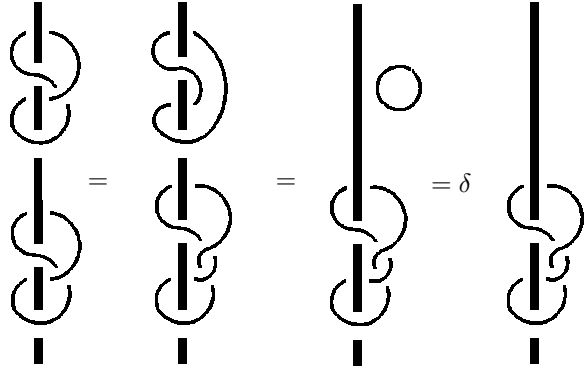

\unitlength 0.8mm

\caption{Equating $(\Xi^+)^2$  to $\delta$ times a closed loop with
two self-intersections} \label{ch5xixi1}
\end{figure}

Now $(\ref{ch5eQ3})$ follows from applying the first pole-related
sef-intersection relation to $\Xi^+\Theta$, moving the pole-related
self-crossing to the nearest loop.

Also, $(\ref{ch5eQ4})$ is obtained by use of the third pole-related
self-intesection relations to move two twists to the same loop, resulting in
two loops without pole twists to give $\delta^2\Theta$.
\end{proof}

\np The tangle $\Xi^-$ is not the inverse of $\Xi^+$.  However, the first
pole-related self-intersection relation can be used to verify that the tangle
$\delta^{-2}\Xi^-$ is the inverse of $\Xi^+$.

\begin{Lm}\label{ch5lmcommute}\label{xicommutes}
The tangles $\Theta$, $\Xi^{+}$ commute with every tangle
of type $\D^{(1)}$.
In particular, $\KT(\D_n)^{(1)}$ can be viewed as an
algebra over $R[\Theta, \Xi^+]$.
\end{Lm}

\begin{proof}
Both tangles obviously commute with any tangle of type $\D^{(1)}$ which
contains no pole twists. If we can show that both
tangles commute with every twist around the pole we are done.

\nl For a closed loop with a twist around the pole this holds by
Proposition~\ref{prop:addrel} (i), the second closed pole loop relation. Hence,
$\Theta$ commutes with every tangle of type $\D^{(1)}$.

\begin{figure}[hbtp]
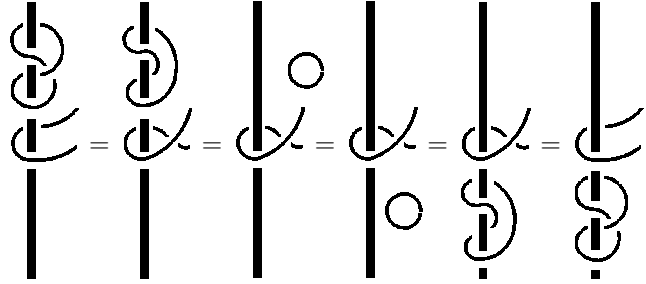

\unitlength 0.6mm

\caption{Commuting $\Xi^+$ with another
twist around the pole} \label{ch5commutexi}
\end{figure}

It remains to prove that $\Xi^+$ commutes with every pole twist. We illustrate
the argument in Figure~\ref{ch5commutexi}.  Using the first pole-related
self-intersection relation, we move the pole-related self-intersection of
$\Xi^+$ to the other strands twisting around the pole. The two pole twists of
the closed loop without self-intersection can now be removed as the pole has
order two. As a closed loop can move freely through the tangle by Reidemeister
move II, we can move it to the other side of the twist. Finally, apply the
reverse of the procedure just described to bring the pole-related
self-intersection back in the closed loop.
\end{proof}

\np
The importance of the lemma is to allow us to isolate closed strands from
the tangle by commuting any $\Xi^{\pm}$ or $\Theta$ away from other
parts of the tangle. 

Also $\KT(\D_n)^{(2)}$ is an algebra over $R[\Xi^+,\Theta]$.  But $\KT(\D_n)$
is not, as the multiple of the identity element by $\Xi^+$ does not belong to
it (see Definitions \ref{df:tangletype} and \ref{df:KTDn}).

\np
Our goal is to establish that $\KT(\D_n)$ is the image of an algebra
homomorphism $\phi: \BMW(\D_n)\to \KT(\D_n)^{(2)}$ for $n\ge 2$.
To set up the homomorphism $\phi$ 
the generators of $\BMW(\D_n)$ are to be mapped
onto simple tangles which contain at most one crossing.
We introduce
$2n$ of these simple tangles, which we denote by $G_i$ and $E_i$ for $i = 1 ,
\ldots ,n$.

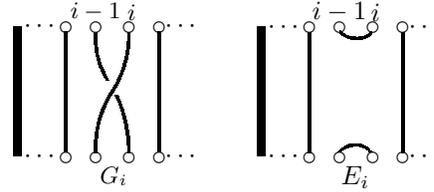
\begin{figure}[htbp]
\unitlength 0.4mm
\begin{picture}(140,60)(0,0)
\linethickness{1mm}
\put(6,7){\line(0,1){43.06}}
\put(60.97,50.19){\makebox(0,0)[cc]{$\cdots$}}
\put(60.97,7.07){\makebox(0,0)[cc]{$\cdots$}}
\put(14.1,7.07){\makebox(0,0)[cc]{$\cdots$}}
\put(14.1,50.19){\makebox(0,0)[cc]{$\cdots$}}

\linethickness{0.3mm}
\put(53.03,8.39){\line(0,1){40}}
\put(22.03,8.39){\line(0,1){40}}
\put(53.03,49.97){\circle{3.16}}
\put(43.03,49.97){\circle{3.16}}
\put(32.03,49.97){\circle{3.16}}
\put(22.03,49.97){\circle{3.16}}
\put(53.03,6.81){\circle{3.16}}

\put(43.03,6.81){\circle{3.16}}
\put(32.03,6.81){\circle{3.16}}
\put(22.03,6.81){\circle{3.16}}
\put(32.03,55){\makebox(0,0)[cc]{\small $i-1$}}
\put(44.03,55){\makebox(0,0)[cc]{\small $i$}}
\put(38.03,0.39){\makebox(0,0)[cc]{\small $G_i$}}

\multiput(42.97,8.19)(0,0.5){1}{\line(0,1){0.5}}
\multiput(42.97,9.19)(0,-0.5){1}{\line(0,-1){0.5}}
\multiput(42.97,9.7)(0.01,-0.5){1}{\line(0,-1){0.5}}
\multiput(42.95,10.2)(0.01,-0.5){1}{\line(0,-1){0.5}}
\multiput(42.94,10.7)(0.02,-0.5){1}{\line(0,-1){0.5}}
\multiput(42.91,11.2)(0.02,-0.5){1}{\line(0,-1){0.5}}
\multiput(42.88,11.7)(0.03,-0.5){1}{\line(0,-1){0.5}}
\multiput(42.85,12.2)(0.04,-0.5){1}{\line(0,-1){0.5}}
\multiput(42.8,12.7)(0.04,-0.5){1}{\line(0,-1){0.5}}
\multiput(42.76,13.2)(0.05,-0.5){1}{\line(0,-1){0.5}}
\multiput(42.7,13.7)(0.05,-0.5){1}{\line(0,-1){0.5}}
\multiput(42.65,14.2)(0.06,-0.5){1}{\line(0,-1){0.5}}
\multiput(42.58,14.7)(0.06,-0.5){1}{\line(0,-1){0.5}}
\multiput(42.51,15.2)(0.07,-0.5){1}{\line(0,-1){0.5}}
\multiput(42.44,15.69)(0.08,-0.5){1}{\line(0,-1){0.5}}
\multiput(42.36,16.19)(0.08,-0.5){1}{\line(0,-1){0.5}}
\multiput(42.27,16.68)(0.09,-0.49){1}{\line(0,-1){0.49}}
\multiput(42.18,17.17)(0.09,-0.49){1}{\line(0,-1){0.49}}
\multiput(42.08,17.67)(0.1,-0.49){1}{\line(0,-1){0.49}}
\multiput(41.98,18.16)(0.1,-0.49){1}{\line(0,-1){0.49}}
\multiput(41.87,18.65)(0.11,-0.49){1}{\line(0,-1){0.49}}
\multiput(41.75,19.14)(0.11,-0.49){1}{\line(0,-1){0.49}}
\multiput(41.63,19.62)(0.06,-0.24){2}{\line(0,-1){0.24}}
\multiput(41.51,20.11)(0.06,-0.24){2}{\line(0,-1){0.24}}
\multiput(41.38,20.6)(0.07,-0.24){2}{\line(0,-1){0.24}}
\multiput(41.24,21.08)(0.07,-0.24){2}{\line(0,-1){0.24}}
\multiput(41.1,21.56)(0.07,-0.24){2}{\line(0,-1){0.24}}
\multiput(40.95,22.04)(0.07,-0.24){2}{\line(0,-1){0.24}}
\multiput(40.8,22.52)(0.08,-0.24){2}{\line(0,-1){0.24}}
\multiput(40.64,22.99)(0.08,-0.24){2}{\line(0,-1){0.24}}
\multiput(40.47,23.47)(0.08,-0.24){2}{\line(0,-1){0.24}}
\multiput(40.3,23.94)(0.08,-0.24){2}{\line(0,-1){0.24}}
\multiput(40.13,24.41)(0.09,-0.24){2}{\line(0,-1){0.24}}
\multiput(39.95,24.88)(0.09,-0.23){2}{\line(0,-1){0.23}}
\multiput(39.76,25.35)(0.09,-0.23){2}{\line(0,-1){0.23}}
\multiput(39.57,25.81)(0.1,-0.23){2}{\line(0,-1){0.23}}
\multiput(39.38,26.27)(0.1,-0.23){2}{\line(0,-1){0.23}}
\multiput(39.18,26.73)(0.07,-0.15){3}{\line(0,-1){0.15}}
\multiput(38.97,27.19)(0.07,-0.15){3}{\line(0,-1){0.15}}

\linethickness{0.3mm}
\multiput(31.95,47.68)(0.02,0.51){1}{\line(0,1){0.51}}
\multiput(31.93,47.18)(0.01,0.51){1}{\line(0,1){0.51}}
\multiput(31.93,46.67)(0,0.51){1}{\line(0,1){0.51}}
\multiput(31.93,46.67)(0,-0.51){1}{\line(0,-1){0.51}}
\multiput(31.93,46.16)(0.01,-0.51){1}{\line(0,-1){0.51}}
\multiput(31.95,45.65)(0.02,-0.51){1}{\line(0,-1){0.51}}
\multiput(31.97,45.15)(0.03,-0.51){1}{\line(0,-1){0.51}}
\multiput(32,44.64)(0.04,-0.51){1}{\line(0,-1){0.51}}
\multiput(32.05,44.14)(0.05,-0.5){1}{\line(0,-1){0.5}}
\multiput(32.1,43.63)(0.06,-0.5){1}{\line(0,-1){0.5}}
\multiput(32.16,43.13)(0.07,-0.5){1}{\line(0,-1){0.5}}
\multiput(32.23,42.63)(0.08,-0.5){1}{\line(0,-1){0.5}}
\multiput(32.3,42.12)(0.09,-0.5){1}{\line(0,-1){0.5}}
\multiput(32.39,41.62)(0.1,-0.5){1}{\line(0,-1){0.5}}
\multiput(32.49,41.13)(0.11,-0.5){1}{\line(0,-1){0.5}}
\multiput(32.59,40.63)(0.11,-0.49){1}{\line(0,-1){0.49}}
\multiput(32.71,40.14)(0.06,-0.25){2}{\line(0,-1){0.25}}
\multiput(32.83,39.64)(0.07,-0.24){2}{\line(0,-1){0.24}}
\multiput(32.96,39.15)(0.07,-0.24){2}{\line(0,-1){0.24}}
\multiput(33.1,38.67)(0.07,-0.24){2}{\line(0,-1){0.24}}
\multiput(33.25,38.18)(0.08,-0.24){2}{\line(0,-1){0.24}}
\multiput(33.41,37.7)(0.08,-0.24){2}{\line(0,-1){0.24}}
\multiput(33.58,37.22)(0.09,-0.24){2}{\line(0,-1){0.24}}
\multiput(33.75,36.74)(0.09,-0.24){2}{\line(0,-1){0.24}}
\multiput(33.94,36.27)(0.1,-0.23){2}{\line(0,-1){0.23}}
\multiput(34.13,35.8)(0.07,-0.16){3}{\line(0,-1){0.16}}
\multiput(34.33,35.34)(0.07,-0.15){3}{\line(0,-1){0.15}}
\multiput(34.54,34.88)(0.07,-0.15){3}{\line(0,-1){0.15}}
\multiput(34.76,34.42)(0.08,-0.15){3}{\line(0,-1){0.15}}
\multiput(34.98,33.96)(0.08,-0.15){3}{\line(0,-1){0.15}}
\multiput(35.22,33.51)(0.08,-0.15){3}{\line(0,-1){0.15}}
\multiput(35.46,33.07)(0.08,-0.15){3}{\line(0,-1){0.15}}
\multiput(35.71,32.63)(0.09,-0.15){3}{\line(0,-1){0.15}}

\linethickness{0.3mm}
\multiput(37.97,29.19)(0.08,0.14){3}{\line(0,1){0.14}}
\multiput(38.22,29.62)(0.08,0.14){3}{\line(0,1){0.14}}
\multiput(38.47,30.05)(0.08,0.15){3}{\line(0,1){0.15}}
\multiput(38.71,30.48)(0.08,0.15){3}{\line(0,1){0.15}}
\multiput(38.95,30.92)(0.08,0.15){3}{\line(0,1){0.15}}
\multiput(39.18,31.37)(0.07,0.15){3}{\line(0,1){0.15}}
\multiput(39.4,31.81)(0.07,0.15){3}{\line(0,1){0.15}}
\multiput(39.62,32.26)(0.07,0.15){3}{\line(0,1){0.15}}
\multiput(39.83,32.71)(0.07,0.15){3}{\line(0,1){0.15}}
\multiput(40.03,33.16)(0.1,0.23){2}{\line(0,1){0.23}}
\multiput(40.23,33.62)(0.1,0.23){2}{\line(0,1){0.23}}
\multiput(40.42,34.08)(0.09,0.23){2}{\line(0,1){0.23}}
\multiput(40.61,34.54)(0.09,0.23){2}{\line(0,1){0.23}}
\multiput(40.78,35.01)(0.09,0.23){2}{\line(0,1){0.23}}
\multiput(40.95,35.48)(0.08,0.24){2}{\line(0,1){0.24}}
\multiput(41.12,35.95)(0.08,0.24){2}{\line(0,1){0.24}}
\multiput(41.28,36.42)(0.08,0.24){2}{\line(0,1){0.24}}
\multiput(41.43,36.89)(0.07,0.24){2}{\line(0,1){0.24}}
\multiput(41.57,37.37)(0.07,0.24){2}{\line(0,1){0.24}}
\multiput(41.71,37.85)(0.07,0.24){2}{\line(0,1){0.24}}
\multiput(41.84,38.33)(0.06,0.24){2}{\line(0,1){0.24}}
\multiput(41.96,38.81)(0.12,0.48){1}{\line(0,1){0.48}}
\multiput(42.08,39.3)(0.11,0.49){1}{\line(0,1){0.49}}
\multiput(42.19,39.78)(0.1,0.49){1}{\line(0,1){0.49}}
\multiput(42.29,40.27)(0.1,0.49){1}{\line(0,1){0.49}}
\multiput(42.39,40.76)(0.09,0.49){1}{\line(0,1){0.49}}
\multiput(42.47,41.25)(0.08,0.49){1}{\line(0,1){0.49}}
\multiput(42.56,41.74)(0.07,0.49){1}{\line(0,1){0.49}}
\multiput(42.63,42.23)(0.07,0.49){1}{\line(0,1){0.49}}
\multiput(42.7,42.73)(0.06,0.49){1}{\line(0,1){0.49}}
\multiput(42.76,43.22)(0.05,0.5){1}{\line(0,1){0.5}}
\multiput(42.81,43.71)(0.05,0.5){1}{\line(0,1){0.5}}
\multiput(42.86,44.21)(0.04,0.5){1}{\line(0,1){0.5}}
\multiput(42.9,44.71)(0.03,0.5){1}{\line(0,1){0.5}}
\multiput(42.93,45.2)(0.02,0.5){1}{\line(0,1){0.5}}
\multiput(42.95,45.7)(0.02,0.5){1}{\line(0,1){0.5}}
\multiput(42.97,46.2)(0.01,0.5){1}{\line(0,1){0.5}}
\multiput(42.98,46.7)(0,0.5){1}{\line(0,1){0.5}}
\multiput(42.98,47.69)(0,-0.5){1}{\line(0,-1){0.5}}
\multiput(42.97,48.19)(0.01,-0.5){1}{\line(0,-1){0.5}}

\linethickness{0.3mm}
\multiput(37.69,28.77)(0.07,0.1){4}{\line(0,1){0.1}}
\multiput(37.41,28.35)(0.09,0.14){3}{\line(0,1){0.14}}
\multiput(37.14,27.92)(0.09,0.14){3}{\line(0,1){0.14}}
\multiput(36.88,27.49)(0.09,0.14){3}{\line(0,1){0.14}}
\multiput(36.62,27.06)(0.09,0.14){3}{\line(0,1){0.14}}
\multiput(36.37,26.62)(0.08,0.15){3}{\line(0,1){0.15}}
\multiput(36.13,26.18)(0.08,0.15){3}{\line(0,1){0.15}}
\multiput(35.89,25.73)(0.08,0.15){3}{\line(0,1){0.15}}
\multiput(35.66,25.28)(0.08,0.15){3}{\line(0,1){0.15}}
\multiput(35.43,24.83)(0.07,0.15){3}{\line(0,1){0.15}}
\multiput(35.21,24.37)(0.07,0.15){3}{\line(0,1){0.15}}
\multiput(35,23.91)(0.07,0.15){3}{\line(0,1){0.15}}
\multiput(34.8,23.45)(0.07,0.15){3}{\line(0,1){0.15}}
\multiput(34.6,22.99)(0.1,0.23){2}{\line(0,1){0.23}}
\multiput(34.41,22.52)(0.1,0.23){2}{\line(0,1){0.23}}
\multiput(34.23,22.05)(0.09,0.24){2}{\line(0,1){0.24}}
\multiput(34.05,21.58)(0.09,0.24){2}{\line(0,1){0.24}}
\multiput(33.88,21.1)(0.09,0.24){2}{\line(0,1){0.24}}
\multiput(33.71,20.62)(0.08,0.24){2}{\line(0,1){0.24}}
\multiput(33.56,20.14)(0.08,0.24){2}{\line(0,1){0.24}}
\multiput(33.41,19.66)(0.07,0.24){2}{\line(0,1){0.24}}
\multiput(33.27,19.18)(0.07,0.24){2}{\line(0,1){0.24}}
\multiput(33.13,18.69)(0.07,0.24){2}{\line(0,1){0.24}}
\multiput(33,18.2)(0.06,0.24){2}{\line(0,1){0.24}}
\multiput(32.88,17.71)(0.06,0.25){2}{\line(0,1){0.25}}
\multiput(32.77,17.22)(0.11,0.49){1}{\line(0,1){0.49}}
\multiput(32.66,16.73)(0.11,0.49){1}{\line(0,1){0.49}}
\multiput(32.56,16.23)(0.1,0.5){1}{\line(0,1){0.5}}
\multiput(32.47,15.73)(0.09,0.5){1}{\line(0,1){0.5}}
\multiput(32.39,15.24)(0.08,0.5){1}{\line(0,1){0.5}}
\multiput(32.31,14.74)(0.08,0.5){1}{\line(0,1){0.5}}
\multiput(32.24,14.24)(0.07,0.5){1}{\line(0,1){0.5}}
\multiput(32.18,13.74)(0.06,0.5){1}{\line(0,1){0.5}}
\multiput(32.12,13.23)(0.06,0.5){1}{\line(0,1){0.5}}
\multiput(32.07,12.73)(0.05,0.5){1}{\line(0,1){0.5}}
\multiput(32.03,12.23)(0.04,0.5){1}{\line(0,1){0.5}}
\multiput(32,11.72)(0.03,0.5){1}{\line(0,1){0.5}}
\multiput(31.97,11.22)(0.03,0.5){1}{\line(0,1){0.5}}
\multiput(31.95,10.71)(0.02,0.5){1}{\line(0,1){0.5}}
\multiput(31.94,10.21)(0.01,0.5){1}{\line(0,1){0.5}}
\multiput(31.94,9.7)(0,0.51){1}{\line(0,1){0.51}}
\multiput(31.94,9.7)(0,-0.51){1}{\line(0,-1){0.51}}
\multiput(31.94,9.2)(0.01,-0.5){1}{\line(0,-1){0.5}}
\multiput(31.95,8.69)(0.02,-0.5){1}{\line(0,-1){0.5}}

\linethickness{1mm}
\put(87,6.94){\line(0,1){43.06}}
\linethickness{0.3mm}
\multiput(113.28,48.16)(0.1,-0.08){4}{\line(1,0){0.1}}
\multiput(113.67,47.84)(0.1,-0.07){4}{\line(1,0){0.1}}
\multiput(114.07,47.54)(0.14,-0.09){3}{\line(1,0){0.14}}
\multiput(114.5,47.26)(0.15,-0.08){3}{\line(1,0){0.15}}
\multiput(114.94,47.02)(0.15,-0.07){3}{\line(1,0){0.15}}
\multiput(115.39,46.8)(0.23,-0.09){2}{\line(1,0){0.23}}
\multiput(115.86,46.61)(0.24,-0.08){2}{\line(1,0){0.24}}
\multiput(116.34,46.45)(0.24,-0.06){2}{\line(1,0){0.24}}
\multiput(116.83,46.33)(0.5,-0.1){1}{\line(1,0){0.5}}
\multiput(117.32,46.23)(0.5,-0.06){1}{\line(1,0){0.5}}
\multiput(117.82,46.17)(0.5,-0.03){1}{\line(1,0){0.5}}
\put(118.33,46.13){\line(1,0){0.5}}
\multiput(118.83,46.13)(0.5,0.03){1}{\line(1,0){0.5}}
\multiput(119.34,46.17)(0.5,0.06){1}{\line(1,0){0.5}}
\multiput(119.84,46.23)(0.5,0.1){1}{\line(1,0){0.5}}
\multiput(120.33,46.33)(0.24,0.06){2}{\line(1,0){0.24}}
\multiput(120.82,46.45)(0.24,0.08){2}{\line(1,0){0.24}}
\multiput(121.3,46.61)(0.23,0.09){2}{\line(1,0){0.23}}
\multiput(121.77,46.8)(0.15,0.07){3}{\line(1,0){0.15}}
\multiput(122.22,47.02)(0.15,0.08){3}{\line(1,0){0.15}}
\multiput(122.66,47.26)(0.14,0.09){3}{\line(1,0){0.14}}
\multiput(123.09,47.54)(0.1,0.07){4}{\line(1,0){0.1}}
\multiput(123.49,47.84)(0.1,0.08){4}{\line(1,0){0.1}}

\linethickness{0.3mm}
\multiput(123.22,8.76)(0.09,-0.09){4}{\line(0,-1){0.09}}
\multiput(122.85,9.09)(0.09,-0.08){4}{\line(1,0){0.09}}
\multiput(122.46,9.4)(0.1,-0.08){4}{\line(1,0){0.1}}
\multiput(122.04,9.67)(0.14,-0.09){3}{\line(1,0){0.14}}
\multiput(121.61,9.92)(0.14,-0.08){3}{\line(1,0){0.14}}
\multiput(121.16,10.14)(0.15,-0.07){3}{\line(1,0){0.15}}
\multiput(120.7,10.33)(0.23,-0.09){2}{\line(1,0){0.23}}
\multiput(120.22,10.48)(0.24,-0.08){2}{\line(1,0){0.24}}
\multiput(119.74,10.6)(0.24,-0.06){2}{\line(1,0){0.24}}
\multiput(119.25,10.69)(0.49,-0.09){1}{\line(1,0){0.49}}
\multiput(118.75,10.74)(0.5,-0.05){1}{\line(1,0){0.5}}
\multiput(118.26,10.76)(0.5,-0.02){1}{\line(1,0){0.5}}
\multiput(117.76,10.74)(0.5,0.02){1}{\line(1,0){0.5}}
\multiput(117.26,10.69)(0.5,0.05){1}{\line(1,0){0.5}}
\multiput(116.77,10.6)(0.49,0.09){1}{\line(1,0){0.49}}
\multiput(116.29,10.48)(0.24,0.06){2}{\line(1,0){0.24}}
\multiput(115.81,10.33)(0.24,0.08){2}{\line(1,0){0.24}}
\multiput(115.35,10.14)(0.23,0.09){2}{\line(1,0){0.23}}
\multiput(114.9,9.92)(0.15,0.07){3}{\line(1,0){0.15}}
\multiput(114.47,9.67)(0.14,0.08){3}{\line(1,0){0.14}}
\multiput(114.05,9.4)(0.14,0.09){3}{\line(1,0){0.14}}
\multiput(113.66,9.09)(0.1,0.08){4}{\line(1,0){0.1}}
\multiput(113.29,8.76)(0.09,0.08){4}{\line(1,0){0.09}}
\multiput(112.94,8.4)(0.09,0.09){4}{\line(0,1){0.09}}

\linethickness{0.3mm}
\put(134.06,8.24){\line(0,1){40}}
\linethickness{0.3mm}
\put(103.06,8.24){\line(0,1){40}}
\linethickness{0.3mm}
\put(134.06,49.82){\circle{3.16}}

\linethickness{0.3mm}
\put(124.06,49.82){\circle{3.16}}

\linethickness{0.3mm}
\put(113.06,49.82){\circle{3.16}}

\linethickness{0.3mm}
\put(103.06,49.82){\circle{3.16}}

\linethickness{0.3mm}
\put(134.06,6.66){\circle{3.16}}

\linethickness{0.3mm}
\put(124.06,6.66){\circle{3.16}}

\linethickness{0.3mm}
\put(113.06,6.66){\circle{3.16}}

\linethickness{0.3mm}
\put(103.06,6.66){\circle{3.16}}

\put(113.06,55){\makebox(0,0)[cc]{$i-1$}}

\put(125.06,55){\makebox(0,0)[cc]{$i$}}

\put(118.26,0.05){\makebox(0,0)[cc]{$E_i$}}

\put(142,50.04){\makebox(0,0)[cc]{$\cdots$}}

\put(142,6.92){\makebox(0,0)[cc]{$\cdots$}}

\put(95.13,6.92){\makebox(0,0)[cc]{$\cdots$}}

\put(95.13,50.04){\makebox(0,0)[cc]{$\cdots$}}

\end{picture}
\caption{The tangles $G_i$ and $E_i$ for $i = 2,\ldots,n$}\label{picsGiEi}
\end{figure}

\nl For $i \ne 1$ we define the $G_i$ to be just the simple
tangles where the $(i-1)$-st and $i$-th node are connected by two
strands with a positive crossing. All other nodes are connected by
straight lines without crossings. These tangles do not have any
pole twists, see the left side of Figure \ref{picsGiEi}.

\nl The tangle $E_i$, where $1< i \le n$, connects the $(i-1)$-st
and $i$-th node by horizontal strands. All other nodes are again
connected by straight lines without crossings. These tangles have
no pole twists, see the right side Figure \ref{picsGiEi}.

\nl The two tangles $G_1$ and $E_1$ are tangles with pole twists. The tangle
$G_1$ is obtained from $G_2$ in a natural way by twisting the two strands
connecting the first and second node around the pole. Similarly, the tangle
$E_1$ is obtained from $E_2$ by twisting the two strands connecting the first
and second node around the pole, see Figure \ref{ch5G0E0}.

\begin{figure}[htbp]
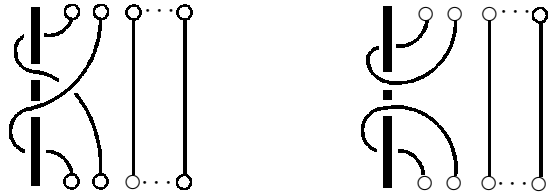

\unitlength 0.6mm

\caption{The additional generators $G_1$ and $E_1$}
\label{ch5G0E0}
\end{figure}

\begin{Remarks}\label{def:Gi}
\rm (i). The $G_i$ and $E_i$ for $i \ge 2$ generate a subalgebra of
$\KT(\D_n)$ that is isomorphic with $\KT(\A_n)$, the Kauffman Tangle algebra
as defined by \cite{MorWas}. The isomorphism is simply defined by removing or
adding the pole, which does not affect the tangles in the subalgebra generated
by $G_i$ and $E_i$ for $i \ge 2$.

\nl (ii). Define $G_0$ to be the tangle in Figure~\ref{ch5G0}.
It is of type $\D^{(1)}$ but not of type $\D$.

\begin{figure}[hbtp]
\unitlength 0.6mm
\begin{picture}(39.57,46.03)(0,0)
\linethickness{0.3mm}
\multiput(7.52,27.57)(0.15,0.08){3}{\line(1,0){0.15}}
\multiput(7.98,27.8)(0.15,0.08){3}{\line(1,0){0.15}}
\multiput(8.42,28.05)(0.14,0.09){3}{\line(1,0){0.14}}
\multiput(8.85,28.32)(0.1,0.07){4}{\line(1,0){0.1}}
\multiput(9.26,28.62)(0.1,0.08){4}{\line(1,0){0.1}}
\multiput(9.66,28.93)(0.1,0.08){4}{\line(1,0){0.1}}
\multiput(10.04,29.27)(0.09,0.09){4}{\line(1,0){0.09}}
\multiput(10.41,29.63)(0.09,0.09){4}{\line(0,1){0.09}}
\multiput(10.75,30)(0.08,0.1){4}{\line(0,1){0.1}}
\multiput(11.08,30.39)(0.08,0.1){4}{\line(0,1){0.1}}
\multiput(11.38,30.8)(0.07,0.11){4}{\line(0,1){0.11}}
\multiput(11.66,31.23)(0.09,0.15){3}{\line(0,1){0.15}}
\multiput(11.92,31.67)(0.08,0.15){3}{\line(0,1){0.15}}
\multiput(12.15,32.12)(0.07,0.15){3}{\line(0,1){0.15}}
\multiput(12.36,32.58)(0.09,0.24){2}{\line(0,1){0.24}}
\multiput(12.55,33.06)(0.08,0.24){2}{\line(0,1){0.24}}
\multiput(12.71,33.54)(0.07,0.25){2}{\line(0,1){0.25}}
\multiput(12.84,34.03)(0.11,0.5){1}{\line(0,1){0.5}}
\multiput(12.95,34.53)(0.08,0.5){1}{\line(0,1){0.5}}
\multiput(13.03,35.03)(0.06,0.51){1}{\line(0,1){0.51}}
\multiput(13.09,35.54)(0.03,0.51){1}{\line(0,1){0.51}}
\multiput(13.12,36.04)(0,0.51){1}{\line(0,1){0.51}}
\multiput(13.1,37.06)(0.02,-0.51){1}{\line(0,-1){0.51}}

\linethickness{0.3mm}
\multiput(12.97,2.47)(0.03,0.5){1}{\line(0,1){0.5}}
\multiput(13,2.96)(0.01,0.5){1}{\line(0,1){0.5}}
\multiput(13,3.96)(0,-0.5){1}{\line(0,-1){0.5}}
\multiput(12.98,4.46)(0.02,-0.5){1}{\line(0,-1){0.5}}
\multiput(12.94,4.95)(0.04,-0.5){1}{\line(0,-1){0.5}}
\multiput(12.89,5.45)(0.05,-0.49){1}{\line(0,-1){0.49}}
\multiput(12.82,5.94)(0.07,-0.49){1}{\line(0,-1){0.49}}
\multiput(12.73,6.43)(0.09,-0.49){1}{\line(0,-1){0.49}}
\multiput(12.63,6.92)(0.1,-0.49){1}{\line(0,-1){0.49}}
\multiput(12.51,7.4)(0.12,-0.48){1}{\line(0,-1){0.48}}
\multiput(12.38,7.88)(0.07,-0.24){2}{\line(0,-1){0.24}}
\multiput(12.23,8.35)(0.08,-0.24){2}{\line(0,-1){0.24}}
\multiput(12.06,8.82)(0.08,-0.23){2}{\line(0,-1){0.23}}
\multiput(11.88,9.29)(0.09,-0.23){2}{\line(0,-1){0.23}}
\multiput(11.68,9.74)(0.1,-0.23){2}{\line(0,-1){0.23}}
\multiput(11.47,10.19)(0.07,-0.15){3}{\line(0,-1){0.15}}
\multiput(11.24,10.64)(0.08,-0.15){3}{\line(0,-1){0.15}}
\multiput(11,11.07)(0.08,-0.15){3}{\line(0,-1){0.15}}
\multiput(10.75,11.5)(0.09,-0.14){3}{\line(0,-1){0.14}}
\multiput(10.48,11.92)(0.09,-0.14){3}{\line(0,-1){0.14}}
\multiput(10.2,12.33)(0.07,-0.1){4}{\line(0,-1){0.1}}
\multiput(9.9,12.73)(0.07,-0.1){4}{\line(0,-1){0.1}}
\multiput(9.59,13.12)(0.08,-0.1){4}{\line(0,-1){0.1}}
\multiput(9.27,13.5)(0.08,-0.09){4}{\line(0,-1){0.09}}
\multiput(8.94,13.87)(0.08,-0.09){4}{\line(0,-1){0.09}}
\multiput(8.59,14.22)(0.09,-0.09){4}{\line(0,-1){0.09}}
\multiput(8.23,14.57)(0.09,-0.09){4}{\line(1,0){0.09}}
\multiput(7.86,14.9)(0.09,-0.08){4}{\line(1,0){0.09}}
\multiput(7.48,15.22)(0.09,-0.08){4}{\line(1,0){0.09}}
\multiput(7.09,15.53)(0.1,-0.08){4}{\line(1,0){0.1}}
\multiput(6.69,15.83)(0.1,-0.07){4}{\line(1,0){0.1}}
\multiput(6.28,16.11)(0.1,-0.07){4}{\line(1,0){0.1}}
\multiput(5.86,16.38)(0.14,-0.09){3}{\line(1,0){0.14}}
\multiput(5.44,16.63)(0.14,-0.08){3}{\line(1,0){0.14}}
\multiput(5,16.88)(0.15,-0.08){3}{\line(1,0){0.15}}

\linethickness{0.3mm}
\multiput(2.72,26.59)(0.1,0.07){4}{\line(1,0){0.1}}
\multiput(2.34,26.26)(0.09,0.08){4}{\line(1,0){0.09}}
\multiput(1.99,25.9)(0.09,0.09){4}{\line(0,1){0.09}}
\multiput(1.68,25.51)(0.08,0.1){4}{\line(0,1){0.1}}
\multiput(1.4,25.1)(0.09,0.14){3}{\line(0,1){0.14}}
\multiput(1.16,24.66)(0.08,0.15){3}{\line(0,1){0.15}}
\multiput(0.96,24.2)(0.1,0.23){2}{\line(0,1){0.23}}
\multiput(0.81,23.72)(0.08,0.24){2}{\line(0,1){0.24}}
\multiput(0.7,23.23)(0.11,0.49){1}{\line(0,1){0.49}}
\multiput(0.63,22.74)(0.07,0.5){1}{\line(0,1){0.5}}
\multiput(0.61,22.24)(0.02,0.5){1}{\line(0,1){0.5}}
\multiput(0.61,22.24)(0.02,-0.5){1}{\line(0,-1){0.5}}
\multiput(0.63,21.74)(0.07,-0.5){1}{\line(0,-1){0.5}}
\multiput(0.7,21.25)(0.11,-0.49){1}{\line(0,-1){0.49}}
\multiput(0.82,20.76)(0.08,-0.24){2}{\line(0,-1){0.24}}
\multiput(0.98,20.28)(0.07,-0.15){3}{\line(0,-1){0.15}}
\multiput(1.18,19.83)(0.08,-0.15){3}{\line(0,-1){0.15}}
\multiput(1.42,19.39)(0.07,-0.1){4}{\line(0,-1){0.1}}
\multiput(1.7,18.97)(0.08,-0.1){4}{\line(0,-1){0.1}}
\multiput(2.01,18.59)(0.09,-0.09){4}{\line(0,-1){0.09}}
\multiput(2.37,18.23)(0.1,-0.08){4}{\line(1,0){0.1}}
\multiput(2.75,17.91)(0.1,-0.07){4}{\line(1,0){0.1}}
\multiput(3.16,17.62)(0.14,-0.08){3}{\line(1,0){0.14}}
\multiput(3.59,17.37)(0.15,-0.07){3}{\line(1,0){0.15}}
\multiput(4.04,17.16)(0.24,-0.08){2}{\line(1,0){0.24}}
\multiput(4.52,17)(0.24,-0.06){2}{\line(1,0){0.24}}

\linethickness{1mm}
\put(5,19.38){\line(0,1){20.62}}
\linethickness{1mm}
\put(5,0){\line(0,1){13.75}}
\linethickness{0.2mm}
\put(11.4,0.7){\line(0,1){0.49}}
\multiput(11.4,0.7)(0.08,-0.23){2}{\line(0,-1){0.23}}
\multiput(11.55,0.23)(0.07,-0.1){4}{\line(0,-1){0.1}}
\multiput(11.84,-0.17)(0.1,-0.07){4}{\line(1,0){0.1}}
\multiput(12.24,-0.46)(0.24,-0.08){2}{\line(1,0){0.24}}
\put(12.71,-0.61){\line(1,0){0.49}}
\multiput(13.21,-0.61)(0.24,0.08){2}{\line(1,0){0.24}}
\multiput(13.68,-0.46)(0.1,0.07){4}{\line(1,0){0.1}}
\multiput(14.08,-0.17)(0.07,0.1){4}{\line(0,1){0.1}}
\multiput(14.37,0.23)(0.08,0.23){2}{\line(0,1){0.23}}
\put(14.52,0.7){\line(0,1){0.49}}
\multiput(14.37,1.66)(0.08,-0.23){2}{\line(0,-1){0.23}}
\multiput(14.08,2.06)(0.07,-0.1){4}{\line(0,-1){0.1}}
\multiput(13.68,2.35)(0.1,-0.07){4}{\line(1,0){0.1}}
\multiput(13.21,2.5)(0.24,-0.08){2}{\line(1,0){0.24}}
\put(12.71,2.5){\line(1,0){0.49}}
\multiput(12.24,2.35)(0.24,0.08){2}{\line(1,0){0.24}}
\multiput(11.84,2.06)(0.1,0.07){4}{\line(1,0){0.1}}
\multiput(11.55,1.66)(0.07,0.1){4}{\line(0,1){0.1}}
\multiput(11.4,1.19)(0.08,0.23){2}{\line(0,1){0.23}}

\linethickness{0.2mm}
\put(17.82,0.7){\line(0,1){0.49}}
\multiput(17.82,0.7)(0.08,-0.23){2}{\line(0,-1){0.23}}
\multiput(17.97,0.23)(0.07,-0.1){4}{\line(0,-1){0.1}}
\multiput(18.26,-0.17)(0.1,-0.07){4}{\line(1,0){0.1}}
\multiput(18.66,-0.46)(0.24,-0.08){2}{\line(1,0){0.24}}
\put(19.13,-0.61){\line(1,0){0.49}}
\multiput(19.63,-0.61)(0.24,0.08){2}{\line(1,0){0.24}}
\multiput(20.1,-0.46)(0.1,0.07){4}{\line(1,0){0.1}}
\multiput(20.5,-0.17)(0.07,0.1){4}{\line(0,1){0.1}}
\multiput(20.79,0.23)(0.08,0.23){2}{\line(0,1){0.23}}
\put(20.94,0.7){\line(0,1){0.49}}
\multiput(20.79,1.66)(0.08,-0.23){2}{\line(0,-1){0.23}}
\multiput(20.5,2.06)(0.07,-0.1){4}{\line(0,-1){0.1}}
\multiput(20.1,2.35)(0.1,-0.07){4}{\line(1,0){0.1}}
\multiput(19.63,2.5)(0.24,-0.08){2}{\line(1,0){0.24}}
\put(19.13,2.5){\line(1,0){0.49}}
\multiput(18.66,2.35)(0.24,0.08){2}{\line(1,0){0.24}}
\multiput(18.26,2.06)(0.1,0.07){4}{\line(1,0){0.1}}
\multiput(17.97,1.66)(0.07,0.1){4}{\line(0,1){0.1}}
\multiput(17.82,1.19)(0.08,0.23){2}{\line(0,1){0.23}}

\linethickness{0.2mm}
\put(11.53,38.32){\line(0,1){0.49}}
\multiput(11.53,38.32)(0.08,-0.24){2}{\line(0,-1){0.24}}
\multiput(11.68,37.85)(0.07,-0.1){4}{\line(0,-1){0.1}}
\multiput(11.97,37.45)(0.1,-0.07){4}{\line(1,0){0.1}}
\multiput(12.37,37.16)(0.23,-0.08){2}{\line(1,0){0.23}}
\put(12.84,37.01){\line(1,0){0.49}}
\multiput(13.33,37.01)(0.23,0.08){2}{\line(1,0){0.23}}
\multiput(13.8,37.16)(0.1,0.07){4}{\line(1,0){0.1}}
\multiput(14.2,37.45)(0.07,0.1){4}{\line(0,1){0.1}}
\multiput(14.49,37.85)(0.08,0.24){2}{\line(0,1){0.24}}
\put(14.64,38.32){\line(0,1){0.49}}
\multiput(14.49,39.29)(0.08,-0.24){2}{\line(0,-1){0.24}}
\multiput(14.2,39.69)(0.07,-0.1){4}{\line(0,-1){0.1}}
\multiput(13.8,39.98)(0.1,-0.07){4}{\line(1,0){0.1}}
\multiput(13.33,40.13)(0.23,-0.08){2}{\line(1,0){0.23}}
\put(12.84,40.13){\line(1,0){0.49}}
\multiput(12.37,39.98)(0.23,0.08){2}{\line(1,0){0.23}}
\multiput(11.97,39.69)(0.1,0.07){4}{\line(1,0){0.1}}
\multiput(11.68,39.29)(0.07,0.1){4}{\line(0,1){0.1}}
\multiput(11.53,38.82)(0.08,0.24){2}{\line(0,1){0.24}}

\linethickness{0.2mm}
\put(17.95,38.32){\line(0,1){0.49}}
\multiput(17.95,38.32)(0.08,-0.24){2}{\line(0,-1){0.24}}
\multiput(18.1,37.85)(0.07,-0.1){4}{\line(0,-1){0.1}}
\multiput(18.39,37.45)(0.1,-0.07){4}{\line(1,0){0.1}}
\multiput(18.79,37.16)(0.23,-0.08){2}{\line(1,0){0.23}}
\put(19.26,37.01){\line(1,0){0.49}}
\multiput(19.75,37.01)(0.23,0.08){2}{\line(1,0){0.23}}
\multiput(20.22,37.16)(0.1,0.07){4}{\line(1,0){0.1}}
\multiput(20.62,37.45)(0.07,0.1){4}{\line(0,1){0.1}}
\multiput(20.91,37.85)(0.08,0.24){2}{\line(0,1){0.24}}
\put(21.06,38.32){\line(0,1){0.49}}
\multiput(20.91,39.29)(0.08,-0.24){2}{\line(0,-1){0.24}}
\multiput(20.62,39.69)(0.07,-0.1){4}{\line(0,-1){0.1}}
\multiput(20.22,39.98)(0.1,-0.07){4}{\line(1,0){0.1}}
\multiput(19.75,40.13)(0.23,-0.08){2}{\line(1,0){0.23}}
\put(19.26,40.13){\line(1,0){0.49}}
\multiput(18.79,39.98)(0.23,0.08){2}{\line(1,0){0.23}}
\multiput(18.39,39.69)(0.1,0.07){4}{\line(1,0){0.1}}
\multiput(18.1,39.29)(0.07,0.1){4}{\line(0,1){0.1}}
\multiput(17.95,38.82)(0.08,0.24){2}{\line(0,1){0.24}}

\linethickness{0.3mm}
\put(26.7,2.49){\line(0,1){34.38}}
\linethickness{0.2mm}
\put(28.26,38.17){\line(0,1){0.49}}
\multiput(28.11,39.14)(0.08,-0.24){2}{\line(0,-1){0.24}}
\multiput(27.82,39.54)(0.07,-0.1){4}{\line(0,-1){0.1}}
\multiput(27.42,39.83)(0.1,-0.07){4}{\line(1,0){0.1}}
\multiput(26.95,39.98)(0.23,-0.08){2}{\line(1,0){0.23}}
\put(26.46,39.98){\line(1,0){0.49}}
\multiput(25.99,39.83)(0.23,0.08){2}{\line(1,0){0.23}}
\multiput(25.59,39.54)(0.1,0.07){4}{\line(1,0){0.1}}
\multiput(25.3,39.14)(0.07,0.1){4}{\line(0,1){0.1}}
\multiput(25.15,38.67)(0.08,0.24){2}{\line(0,1){0.24}}
\put(25.15,38.17){\line(0,1){0.49}}
\multiput(25.15,38.17)(0.08,-0.24){2}{\line(0,-1){0.24}}
\multiput(25.3,37.7)(0.07,-0.1){4}{\line(0,-1){0.1}}
\multiput(25.59,37.3)(0.1,-0.07){4}{\line(1,0){0.1}}
\multiput(25.99,37.01)(0.23,-0.08){2}{\line(1,0){0.23}}
\put(26.46,36.86){\line(1,0){0.49}}
\multiput(26.95,36.86)(0.23,0.08){2}{\line(1,0){0.23}}
\multiput(27.42,37.01)(0.1,0.07){4}{\line(1,0){0.1}}
\multiput(27.82,37.3)(0.07,0.1){4}{\line(0,1){0.1}}
\multiput(28.11,37.7)(0.08,0.24){2}{\line(0,1){0.24}}

\linethickness{0.2mm}
\put(26.58,0.79){\circle{3.16}}

\put(32.96,38.75){\makebox(0,0)[cc]{$\cdots$}}

\put(32.33,0.62){\makebox(0,0)[cc]{$\cdots$}}

\linethickness{0.3mm}
\put(38,2.5){\line(0,1){34.37}}
\linethickness{0.2mm}
\put(39.55,38.17){\line(0,1){0.49}}
\multiput(39.4,39.14)(0.08,-0.24){2}{\line(0,-1){0.24}}
\multiput(39.11,39.54)(0.07,-0.1){4}{\line(0,-1){0.1}}
\multiput(38.71,39.83)(0.1,-0.07){4}{\line(1,0){0.1}}
\multiput(38.24,39.98)(0.23,-0.08){2}{\line(1,0){0.23}}
\put(37.75,39.98){\line(1,0){0.49}}
\multiput(37.28,39.83)(0.23,0.08){2}{\line(1,0){0.23}}
\multiput(36.88,39.54)(0.1,0.07){4}{\line(1,0){0.1}}
\multiput(36.59,39.14)(0.07,0.1){4}{\line(0,1){0.1}}
\multiput(36.44,38.67)(0.08,0.24){2}{\line(0,1){0.24}}
\put(36.44,38.17){\line(0,1){0.49}}
\multiput(36.44,38.17)(0.08,-0.24){2}{\line(0,-1){0.24}}
\multiput(36.59,37.7)(0.07,-0.1){4}{\line(0,-1){0.1}}
\multiput(36.88,37.3)(0.1,-0.07){4}{\line(1,0){0.1}}
\multiput(37.28,37.01)(0.23,-0.08){2}{\line(1,0){0.23}}
\put(37.75,36.86){\line(1,0){0.49}}
\multiput(38.24,36.86)(0.23,0.08){2}{\line(1,0){0.23}}
\multiput(38.71,37.01)(0.1,0.07){4}{\line(1,0){0.1}}
\multiput(39.11,37.3)(0.07,0.1){4}{\line(0,1){0.1}}
\multiput(39.4,37.7)(0.08,0.24){2}{\line(0,1){0.24}}

\linethickness{0.2mm}
\put(39.43,0.55){\line(0,1){0.49}}
\multiput(39.28,1.51)(0.08,-0.23){2}{\line(0,-1){0.23}}
\multiput(38.99,1.91)(0.07,-0.1){4}{\line(0,-1){0.1}}
\multiput(38.59,2.2)(0.1,-0.07){4}{\line(1,0){0.1}}
\multiput(38.12,2.35)(0.24,-0.08){2}{\line(1,0){0.24}}
\put(37.62,2.35){\line(1,0){0.49}}
\multiput(37.15,2.2)(0.24,0.08){2}{\line(1,0){0.24}}
\multiput(36.75,1.91)(0.1,0.07){4}{\line(1,0){0.1}}
\multiput(36.46,1.51)(0.07,0.1){4}{\line(0,1){0.1}}
\multiput(36.31,1.04)(0.08,0.23){2}{\line(0,1){0.23}}
\put(36.31,0.55){\line(0,1){0.49}}
\multiput(36.31,0.55)(0.08,-0.23){2}{\line(0,-1){0.23}}
\multiput(36.46,0.08)(0.07,-0.1){4}{\line(0,-1){0.1}}
\multiput(36.75,-0.32)(0.1,-0.07){4}{\line(1,0){0.1}}
\multiput(37.15,-0.61)(0.24,-0.08){2}{\line(1,0){0.24}}
\put(37.62,-0.76){\line(1,0){0.49}}
\multiput(38.12,-0.76)(0.24,0.08){2}{\line(1,0){0.24}}
\multiput(38.59,-0.61)(0.1,0.07){4}{\line(1,0){0.1}}
\multiput(38.99,-0.32)(0.07,0.1){4}{\line(0,1){0.1}}
\multiput(39.28,0.08)(0.08,0.23){2}{\line(0,1){0.23}}
\linethickness{0.3mm}
\put(19.38,2.5){\line(0,1){34.37}}
\end{picture}
\caption{The element $G_0$} \label{ch5G0}
\end{figure}
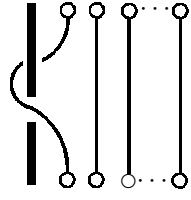

As the pole has order two, it is an involution.  Conjugation by $G_0$ is an
automorphism of $\KT(\D_n)^{(1)}$ as it leaves invariant the set of defining
relations for $\KT(\D_n)^{(1)}$; for instance, it interchanges the first two
pole-related self-intersection relations (v) and (vi). It also leaves the
subalgebras $\KT(\D_n)^{(2)}$ and $\KT(\D_n)$ invariant.  Moreover, $G_1 =
G_0G_2G_0$, and $E_1 = G_0E_2G_0$, whereas $G_0G_iG_0=G_i$ for $i\geq 3$, so
conjugation by $G_0$ is also an automorphism of the subalgebra of
$\KT(\D_n)^{(1)}$ generated by the $ G_i$ and $E_i$ for $1\le i \le n$.
\end{Remarks}

\begin{Prop}\label{homom}
If $n\ge 2$, there is a homomorphism $\phi: \BMW(\D_n) \to \KT(\D_n)$ of
$R$-algebras determined by $\phi(g_i) = G_i$ and $\phi(e_i)=E_i$ for all 
$i\in\{1,\ldots,n\}$.

\end{Prop}

\begin{proof}
We need to check that the defining relations of the BMW algebra $\BMW(\D_n)$,
given in Definition~\ref{BMW-def}, are respected by the tangles.  If the
indices of the generator symbols do not include $1$, the pole is not involved
and the equalities follow from \cite{MorWas}.  Also, the defining relation
(D1) is covered by the Kauffman skein relation (i).
If one of the indices is $1$ and the second not $2$, conjugate by $G_0$ to get
$E_2$ or $G_2$.  The relations can be obtained this way except for ones
involving both $i=1$ and $j=2$.  
The relation (B1), $G_1G_2 = G_2G_1$, is straightforward from the commuting
relation (i).
The relation (HCer), $E_1G_2=G_2E_1$, follows from the first pole-related
self-intersection and conjugation by $G_0$.
The relation $E_2G_1=G_1E_2$ follows by conjugation with $G_0$.
Finally, the relation (HCee),
$E_1E_2=E_2E_1$, follows from the first closed pole loop relation and the pole
being of order two.
\end{proof}

Later, in Theorem~\ref{th:surjhomo}, we will show that $\phi$ is surjective.

\section{Totally Descending Tangles}
\label{TotDescending} \label{sec:totallyDescTangles}
In this section we identify a spanning set of tangles for $\KT(\D_n)$.  The
result is Proposition \ref{tottangle} below. We will restrict our attention to
certain tangles, called totally descending, as in \cite{MorWas}.

\begin{Def}\label{df:descending}
\rm Given a tangle $T$, choose a sequence of base points: firstly an endpoint
(in $K$) for each full non-closed strand, and, secondly a point on each closed
strand. Subsequently provide each strand with an orientation determined by a
direction starting at the base point.  We say $T$ is {\em totally descending}
(with respect to the ordered base points and orientations) if, on traversing
all the strands from $T$ in order of the base points, we meet each crossing
for the first time as an over crossing. Such a crossing is called {\em
descending}.
\end{Def}

\begin{Lm}\label{totspan}
$\KT(\D_n)$ is spanned by totally descending tangles.
\end{Lm}

\begin{proof}
The proof is similar to the proof of Theorem 2.6 in \cite{MorWas}
and is done by induction on first the number of crossings, then on
the number of non-descending crossings.

\nl Let $T$ be a tangle in $\KT(\D_n)$. Choose a sequence of base points for
$T$. Follow the strands of $T$ in this determined order. At the first
non-descending crossing apply the Kauffman skein relation (i). This results in
a linear combination of a tangle with this particular crossing changed to
a descending one and two tangles with the crossing removed.
Induction shows we can
write every tangle in $\KT(\D_n)$ as a linear combination of
totally descending crossings this way.
\end{proof}

Each strand of a totally descending tangle lies entirely above the strands
that appear later in the order.  The importance of totally descending tangles
is that Reidemeister moves can always be made.  In the case of Reidemeister
III the move is with respect to at least one of the crossings.  This is
because one of the strands is first in the order given by Definition
\ref{df:descending} and one is last.  This means one is above the other two,
and one below the other two and stays that way.  For Reidemeister II one is
above the other.

The next result involves eight kinds of regions always to the right of the
pole.

\begin{Def}\label{def:typeE}
\rm Let $T$ be a tangle in $\U_n^{(1)}$.
A region of $T$ is understood to be a part of the
$x,z$-plane $\Pi$ in which $T$ is projected by means of the
natural projection along the $y$-axis; it is bounded by segments of strands
and segments from the borders of the diagram, with the understanding that the
West border of $T$ consists of the pole; so parts of the strands that
are left of the pole never occur as parts of a region. In particular, a region
can be crossed by strands.  Crossings of strands will be interpreted as real
crossings in the plane $\Pi$ where the boundaries of regions are concerned,
even though in $T$ one strand of the crossing passes above the other.

A region of $T$ is called of type E if one of the following
cases occurs.

\begin{itemize}
\item[(1)]  The region is bounded by exactly
 three strands as in Reidemeister III.

\item[(2)] The region is bounded by exactly two strands as in Reidemeister II
or by just one self-intersecting strand as in Reidemeister I; or the region is
enclosed by a strand twisting around the pole exactly once with one
self-intersection as in a pole-related self-intersection, in which case the
region is bounded by the part of the pole where the strand twists around the
pole and the segments of the strand until they cross to the right of the pole.

\item[(3)] The region is bounded by the East border of $T$ and a single
strand, which starts at $n$ on the top (that is, endpoint $(n,0,1)$) and
finishes at $n$ on the bottom (endpoint $(n,0,0)$).

\item[(4)] The region is bounded by the segment of the North border between
endpoints $i$ and $i+1$, and by a single strand starting at $i$ on the top and
ending at $i+1$ on the top or by the two strands starting at the top endpoints
$i$ and $i+1$, respectively.  Similarly, the same
description using the South border.

\item[(5)] The region is bounded by a segment of the pole and the segment of a
strand between two pole twists if no other strand twists around the pole
between these two twists.

\item[(6)] The region is bounded by the topmost segment of the pole,
the leftmost segment at the top between the pole and the strand starting at
the endpoint $1$,
and the segment of the strand starting at top endpoint 1
and its first twist around the pole
if no other strand has twisted around the pole in this region.

\item[(7)] The region is enclosed by a closed strand without
self-intersections.  The strand is either entirely to the right of the pole or
twists around the pole exactly once, in which case the region is bounded by
the part of the pole where the strand twists around the pole and the segment
of the strand to the right of the pole.

\item[(8)] The region is bounded by the pole and by two strands each of which
twists around the pole while the two strands cross to the right of the pole
without any pole twists in between.
\end{itemize}
\end{Def}

\begin{Lm}\label{evacuate} %3.5
Suppose that $Q$ is a region of type E of a totally descending tangle in
$\U_n^{(1)}$ without closed strands.  Then $Q$ can be evacuated in the sense
that all strands entering the region can be removed by means of Reidemeister
moves.  The resulting tangle is totally descending, represents the same
element of $\KT(\D_n)^{(1)}$, and has no strands in the interior of $Q$.
\end{Lm}

\begin{proof}
We begin with $Q$ being a region of type (1) or (2).
We will use induction on
the total number of crossings in $Q$, including those in its boundary.
If $Q$ is not already evacuated, there are
strands which enter $Q$ and subsequently leave it.
Let $s$ be such a strand. By induction, we can evacuate any region bordered by
a self-intersection of $s$ within $Q$ and subsequently apply Reidemeister I.
Therefore, we may assume that $s$ has no self-intersections within $Q$.  If it
enters and leaves across the same strand segment of the boundary of $Q$, it
creates a new region of type (2) and within this region there are fewer total
crossings than $Q$ has and, by the induction hypothesis we can evacuate this
region; using Reidemeister II we can take $s$ completely away from $Q$ and use
induction again.  Assume therefore that $s$ leaves through a different strand
in the boundary of $Q$.  If $Q$ is of type (2), there are two new regions of
type (1).  Each has fewer total crossings as at least one of the original
crossings is not in the new region.  Use induction again to evacuate one of
the new regions.  If $Q$ is of type (1), at least one of the two new regions
of the dissection of $Q$ by $s$ is of type (1).  Again the new regions have
fewer total crossings and so, by induction, can be evacuated. In both cases,
we use of Reidemeister III to remove $s$ from $Q$.  This shows that the result
holds for regions of type (1) and type (2).

Now suppose that $Q$ is a region of another type.
If a strand enters $Q$, it
must also leave it.
By following a strand from when it enters $Q$ until it first leaves $Q$, we
obtain a new region of type (1) or type (2).  If there is just one strand
bordering the region it is certainly of type (2).  When there are two strands
it could be of type (1).  Now use the result above for regions of these types
to remove the strand from $Q$.  Continue doing this until all extraneous
strands are removed from $Q$.
\end{proof}

\begin{Prop}\label{closedloops} %3.6
Suppose that $T$ is a totally descending tangle in $\U_n^{(1)}$ and let $Q$ be
a region of type E. Then $Q$ can be evacuated in the sense that the tangle can
be rewritten to a tangle that is totally descending, represents the same
element of $\KT(\D_n)^{(1)}$, and has no strands in the interior of $Q$.
Moreover, every closed strand to the right of the pole can be removed with the
introduction of powers of $\delta$ and $l^{\pm 1}$ as coefficients of $T$ in
$\KT(\D_n)^{(1)}$.
\end{Prop}

\begin{proof}
By the definition of type E, there are no closed strands twisting around the
pole inside $Q$.  If there are no closed strands inside $Q$, then by Lemma
\ref{evacuate} we can evacuate $Q$ as required.  So assume $q$ is a closed
strand inside $Q$, not enclosing another closed strand.  The region enclosed
by $q$ is made up of a finite number of regions entirely bounded by a single
segment of $q$. Such a region is called a {\em
closed component}.  Notice that each of the closed components can be evacuated
by Lemma~\ref{evacuate} and removed by Reidemeister I with the introduction of
a power of $l$ until there is just one closed component which can also be
evacuated.  Such a closed component can be shrunk
to a very small one which can be moved by Reidemeister II across any
strand and can be completely separated from the rest of the
diagram.  It can then be removed by multiplying by $\delta$ using the
idempotent relation (iv). 
\end{proof}

In view of Proposition \ref{closedloops}, we can evacuate certain regions of a
tangle in $\U_n^{(1)}$ while representing the same element of
$\KT(\D_n)^{(1)}$.  In order to keep track of the form of tangles needed for a
spanning set of $\KT(\D_n)$, we introduce the following two notions.

\begin{Def}\label{df:complexity} %3.7
\rm The {\em complexity} of a tangle $T$ in $\U_n^{(1)}$ is the sum of the
number of pole twists and the number of crossings appearing in $T$.
Such a tangle is called {\em reduced} if, as a member of $\KT(\D_n)^{(1)}$, it
is totally descending and cannot be written as a linear combination of tangles
with strictly lower complexities.
\end{Def}

In our search for a spanning set of $\KT(\D_n)$ we need only consider reduced
tangles.  In a reduced tangle, no self-intersecting strands without pole
twists occur; for otherwise a region of type E(2) would occur, which can be
evacuated by Proposition \ref{closedloops}, so that Reidemeister I could be
applied to reduce the tangle's complexity. Similarly, by Reidemeister II, no
two strands cross twice without twisting around the pole between the two
crossings.  Similarly, up to scaling by factors of $\delta$, we can assume
that no closed strands occur to the right of the pole.  As discussed in Remark
\ref{rmk:onetwist}, a reduced tangle has at most one pole-related
self-intersection.

\begin{Lm}\label{twotwists} %3.8
Let $T$ be a tangle in $\U_n^{(1)}$ with a strand $q$ having two pole twists
without pole-related self-intersections. Assume that there are no further
twists of $q$ around the pole in between these two.  Let $Q$ be the region
bounded by the pole between the two twists and the segment of $q$ between the
two twists.  If $Q$ has no closed strands, then $T$ is not reduced.
\end{Lm}

\begin{proof}
Assume that $T$ is a counterexample of smallest complexity.  Let $t_1$ and
$t_2$ be the two consecutive twists of $q$ around the pole and bordering $Q$.
If there are no strands of $T$ inside $Q$, the twists $t_1$ and $t_2$ are
adjacent along the pole and, as the pole has order~$2$, these twists can be
removed. As $T$ is reduced any strand entering $Q$ must twist around the pole
before leaving $Q$.  So without loss of generality we may assume there is a
further pole twist, say $t_3$ of a strand $q'$ between $t_1$ and
$t_2$.  Since, by assumption, there are no closed strands within $Q$, the
strand $q'$ enters and leaves $Q$.  If $q'$ twists around the pole twice,
there is a smaller region as in the hypotheses, and the result follows from
the minimality of the complexity of $T$. If one of the twists of $q'$ were a
pole-related self-intersection within $Q$, it could be moved to $q$ at $t_1$
or $t_2$ by Lemma \ref{lm:crossChanges}.  Therefore, we may assume that each
strand entering $Q$ twists once inside the region.

Without loss of generality, we may assume that $t_3$ is the first twist
occurring in between $t_1$ and $t_2$ when going down the pole.  This implies
that the region bounded by the segment of the pole from $t_1$ to $t_3$ and by
$q$ and $q'$ has type E(8) and can be evacuated as in Lemma
\ref{evacuate}.  Let $Q'$ be the region whose corners are the twist $t_3$ and
the two adjacent crossings of $q$ and $q'$. If $Q'$ is also evacuated, the
second commuting relation allows us to remove the twist from $Q$, a
contradiction.

We are left with the case where $Q'$ is not evacuated. So there is a strand
$q''$ that enters $Q'$ and a region $S$ of type E(1) whose corners are a
crossing of $q$ and $q'$, a crossing of $q$ and $q''$, and a crossing of $q'$
and $q''$.  Evacuate $S$ and apply Reidemeister III so as to remove $q''$ from
$Q'$.  By induction on the number of crossings in $Q'$, we can evacuate all of
$Q'$ in this way.  This leads us to the previous case and hence to the final
contradiction.
\end{proof}

\begin{Lm} \label{ch5once} %3.9
Let $T$ be a tangle in $\U_n^{(1)}$ for which there are no closed strands
which twist around the pole. Suppose that some strand of $T$ twists around the
pole two or more times without a pole-related self-intersection.  Then $T$ is
not reduced.
\end{Lm}

\begin{proof}
Suppose that $T$ is reduced.
If there is a strand in $T$ without a pole-related self-intersection that
twists around the pole more than once, by the assumption that there are no
closed strands, Lemma~\ref{twotwists} allows us to rewrite $T$ to a linear
combination of tangles of smaller complexities.
\end{proof}

We next deal with closed strands which twist around the pole.  Examples are
the closed tangles $\Theta$, $\Xi^{+}$, and $\Xi^{-}$ which occurred in
Lemma~\ref{ch5KT0} and Figure~$\ref{ch5KT0tangles}$.  When there are no other
strands crossing such a closed strand, it can be
moved into the coefficient ring $R[\Xi^+,\Theta]$, cf.~Lemma \ref{xicommutes}.
We will show that when there is such a closed strand, the tangle can be
rewritten so that the strand has at most two twists around the pole.

\begin{Lm}\label{reducetoXi} %3.10
Suppose that $T$ is a tangle in $\U_n^{(1)}$ containing a closed strand $q$
twisting around the pole three or more times.  Assume there is no other closed
strand twisting around the pole between any two of these twists.  Then $T$ is
not reduced.
\end{Lm}

\begin{proof}
Let $T$ be a counterexample of minimal complexity.  By Remark
\ref{rmk:onetwist}, of all twists around the pole in $T$, at most one is a
pole-related self-intersection, and this one can be moved to the South-most
twist, say $t_3$, of $q$ around the pole.  Consider the top twist $t_1$ of $q$
around the pole.  The two segments of $q$ beginning at $t_1$ do not cross as
these segments could not possibly both end at $t_3$ for otherwise they
would close the strand and $q$ would have at most two twists. Let $t_2$ be the
next twist of $q$ from the top. The bottom
segment of $q$ starting at $t_1$ must be joined to the top segment of $q$
starting at $t_2$, for otherwise, we would be able to produce another
pole-related self-intersection by evacuation of a region of type E
(cf.\ Proposition \ref{closedloops}).  Now the region bounded by this segment of $q$ and the part
of the pole between $t_1$ and $t_2$ satisfies the conditions of Lemma
\ref{twotwists}.  By that lemma, $T$ is not reduced, a contradiction.
\end{proof}

\begin{Lm}\label{isolate}
Suppose that $T$ is a reduced tangle of type $\D^{(1)}$ containing
a closed strand $q$ that twists around the pole. If $q$ twists around the pole
more than once, assume that there are two twists such that no closed strand
twists around the pole between them.  Then $T$ can be rewritten in such
a way that $q$ is one of $\Theta$, $\Xi^+$, or $\Xi^-$, and no other strands
cross it.
\end{Lm}

\begin{proof}  If $q$ has only one twist, the
strand is a closed circle around the pole and so, by Proposition
\ref{prop:addrel}(ii) and the fact there are an even number of pole twists, we
can replace $q$ by $\Theta$ up to multiplication by a power of $\delta$.

Therefore we assume that there are no closed strands with at most one twist
around the pole. In particular, $q$ has at least two pole twists, say $t_1$
and $t_2$. As the number of closed strands is finite, without loss of
generality, we may assume that there is no closed strand entirely contained in
the region enclosed by $q$ and the part of the pole between $t_1$ and
$t_2$. Now, by restricting to a suitable region entirely containing $q$ and
isotopy, we find a tangle satisfying the conditions of Lemma \ref{reducetoXi}.
Applying the lemma and continuing this way, we obtain the required assertion.
\end{proof}

\begin{Lm}\label{(0,0)-tangles}  Suppose
that $T$ is a reduced tangle in $\U_n^{(2)}$.  
Then $T$, viewed as an element of $\KT(\D_n)^{(2)}$,
is an $R[\Theta,\Xi^+]$-linear combination of tangles without closed strands.
In particular, any $(0,0)$-tangle can be
written in terms of the $R$-algebra $R[\Theta,\Xi^+]$.
\end{Lm}

\begin{proof} 
Recall that $\Xi^{\pm}$ and $\Theta$ can be commuted out of $T$ by
Lemma~\ref{ch5lmcommute}.

Let $q$ be a closed strand. If it has no pole twists, it can be replaced by
the scalar $\delta$. If $q$ has a single pole twist, it can be commuted out by
the second closed pole loop relation (i) of Proposition \ref{prop:addrel}.  As
the number of pole twists is even, there must be another pole twist, which we
use to apply the third closed pole loop relation, Proposition
\ref{prop:addrel}(ii), and to extract a factor $\delta^{-1}\Theta$.  

Suppose, therefore, that $q$ has at least two pole twists. By Lemma
\ref{isolate}, we may assume that if $t_1$ and $t_2$ are a consecutive pair of
pole twists, there is a closed strand twisting around the pole between them,
which we can rewrite as required by induction on the number of pole twists of
other strands in the region enclosed by $q$.
\end{proof}

We have developed enough properties for a standard expression
in terms of closed strands and twists around the pole.  

\begin{Prop} \label{ch5rewrite}
Let $T$ be a reduced tangle in $\KT(\D_n)^{(2)}$.

\begin{enumerate}[{\rm (i)}] 
\item If $T$ contains a pole-related
self-intersection, then $T = \delta^{-1}\Xi^{\pm} T'$, where $T'$ is the
tangle obtained from $T$ by removing the pole-related self-intersection.

\item If $T$ contains a closed strand $q$ without self-intersections twisting
around the pole, then $T = \delta^{-1}\Theta T'$ where $T'$ is the tangle
obtained from $T$ by removing $q$ and all twists of other strands around the
pole.
\end{enumerate}
\end{Prop}

\begin{proof}
(i).  By assumption, part of the tangle $T$ is similar to one of the partial
diagrams shown in the first or second pole-related self-intersection relations
(v), (vi) of Definition~\ref{ch5main}.

For the second pole-related self-intersection relation (vi), consider
Figure~\ref{ch5xidef}. A closed loop with no pole twists can be brought into
the tangle by applying the idempotent relation (iv) backwards. Now the
self-intersection relations (iii) allow moving the pole-related
self-intersection inside this loop, obtaining the tangle $T'$ as described
above.

For the first pole-related self-intersection relation (v), take the closed
loop not twisting around the pole and let it twist around the pole twice and a
similar picture gives the result.

\nl (ii). Besides the single pole twist of $q$, there are an odd number, say
$2r+1$, of pole twists in $T$. As in the proof of Lemma
\ref{(0,0)-tangles}, the second equality of the third closed pole relation,
Proposition~\ref{prop:addrel}(ii), allows us to replace one of the $2r+1$
twists by $\delta^{-1}$ and a second loop around the pole.  Now the further
equalities of the third closed pole relation can be used to remove the $2r$
remaining pole twists and we obtain the desired tangle.
\end{proof}

\begin{figure}[htbp]
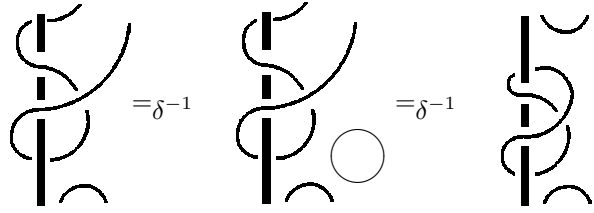

\unitlength 0.7mm

\caption{Replacing a pole-related  self-intersection by $\delta^{-1}\Xi^+$} 
\label{ch5xidef}
\end{figure}

\begin{Thm} \label{tottangle}
As an $R$-module, $\KT(\D_n)$ is spanned by reduced tangles $T$
with each strand twisting around the pole at most once
satisfying one of the following three properties.

\nl
\begin{enumerate}[{\rm  (i)}] \item The tangle $T$ contains no closed
strands at all.
\item $T= \Xi^{\pm}T'$ where $T'$
is a tangle with a horizontal strand but no closed strand.
\item $T = \Theta T'$ where $T'$ is a tangle with a horizontal strand
containing neither a closed strand nor a strand twisting around the pole.
\end{enumerate}
\end{Thm}

\begin{proof}
By induction on complexity, there is a spanning set consisting of reduced
tangles.  Let $T$ be a member of a minimal spanning set of reduced tangles.
If there are no crossings or twists, the tangle is of the required form.  The
three distinct shapes (i), (ii), (iii) are due to the rewrite rules of
Proposition~\ref{ch5rewrite} and the requirement that the tangle be of type
$\D$.

Without loss of generality we can assume that the base point and orientation
on $\Xi^{+}$ are chosen in such a way that the positive self-intersection is
totally descending. Other configurations of closed loops around the pole are
already shown to be rewritable to the ones listed here. So we only have to
show that all these tangles have no closed strands without pole twists.
But such closed strands can
be removed as we saw in Proposition \ref{closedloops}.
\end{proof}

We are now ready for one of the main results which is half of
Theorem~\ref{th:main} (\ref{BMW2KTiso}).

\begin{Thm}\label{th:surjhomo}
The map $\phi:\BMW(\D_n)\to \KT(\D_n)$ of Proposition~\ref{homom} is
surjective.
\end{Thm}

\begin{proof}
Recall that $n\geq 2$.  We need to show that an arbitrary reduced tangle
is a monomial in the generators $G_i^{\pm 1}$, $E_i$ for $i = 1, \ldots ,n$.

Suppose that the result is proved for reduced tangles without closed strands
and let $T$ be as in (ii) or (iii) Theorem \ref{tottangle}. Then the tangle
$T'$ in this part of the theorem is in the image of $\phi$.  We can conjugate
such a $T'$ by a suitable series of $G_i$'s so a horizontal strand joins the
endpoints $1$ to $2$ at the top.  By Remark \ref{def:Gi}(ii), conjugation by
$G_0$ is an isomorphism of the algebra generated by these generators, so we
can conjugate by $G_0$ if necessary to reduce to the case where the strand
does not go around the pole (here we use the fact, given by the theorem, that
each strand of $T$ twists around the pole at most once).

Now premultipy by $E_2G_1E_2$ or $E_2E_1E_2$ to get the factor
$\Xi^{\pm}$ or $\Theta$ and so $T$ is also in the image of $\phi$.

Therefore, by Theorem \ref{tottangle}, we may assume that $T$ has no closed
strands.  Suppose first that there is a vertical strand going from the top to
the bottom.  If it happens to twist around the pole, it goes around the pole
exactly once, by Theorem \ref{tottangle}. Now multiply on the left and the
right by products of $G_i$, with $i\ge 2$ to get that $1$ on the top goes to
$2$ on the bottom and of course still around the pole.  Now conjugate by $G_0$
to get a vertical strand from $1$ on the top to $2$ on the bottom that does
not twist around the pole.

Suppose then that there is a vertical strand that does not twist around the
pole.  Now multiply again on the left and the right by products of $G_i$, with
$i\ge 2$ to get that $n$ on the top goes to $n$ on the bottom.  Now this
strand forms a region of type E(3), cf.~Definition \ref{def:typeE}, and so can
be evacuated by Lemma \ref{evacuate}.  The result is an $(n-1,n-1)$-tangle 
$T_1$ with $T = \eps(T_1)$ and
we can use induction unless $n=2$.  Suppose $n=2$.  As there are an even
number of twists around the pole we see the other strand must join $1$ on the
top to $1$ on the bottom without twisting around the pole.  As there are no
horizontal strands and the element is in $\KT(\D_n)$ there are no closed loops
twisting around the pole and so $T$ is the identity, 
and so belongs to the image of $\phi$.

The only other possibility is that all strands are horizontal.  Suppose that
there are horizontal strands on the top and on the bottom that do not go
around the pole.  Act by $G_i$, $i\ge 2$ to get strands from the endpoint
$n-1$ on the bottom to $n$ on the bottom and the same for the top.  Now these
strands bound regions of type E(4) and can be evacuated by Lemma
\ref{evacuate}.  This leaves an $(n-2,n-2)$-tangle $T_2$ with $T = T_2 E_n$,
and we use induction on $n$ for $T_2$ to conclude that $T_2$ and hence $T$
belongs to the image of $\phi$.

Suppose that all horizontal strands on the top twist around the pole.  Pick
the strand that twists around the pole closest to the top and let $k$ be the
endpoint on the top where it starts.  Multiplying from the left by suitable
$G_i$ as before, we can move the $k$ to $1$.  Now the region between the
strand from $1$ on the top to the pole is a region of type E(6) and can be
evacuated by Lemma \ref{evacuate}.  Now conjugate $T$ by $G_0$ to obtain that
the horizontal strand from $1$ on the top does not twist around the pole.  If
$n=2$, then $T$, having an even number of pole twists, must be $E_2$ and we
are done. Suppose therefore, $n\ge3$.  If now all strands on the bottom twist
around the pole, we multiply from the right by suitable $G_i$ to obtain a
horizontal strand from the endpoint 1 at the bottom joint to the first pole
twist from the bottom. After evacuation as before, in order not to relapse
into the first horizontal strand at the top having a pole twist, we shift this
strand away from $1$, by premultiplying $T$ with $G_2G_3$ before conjugation
by $G_0$. We now have horizontal lines on the top and bottom not going around
the pole and we can apply the results of the previous paragraph.
\end{proof}

\begin{Remarks}\label{rmks:AD}
\rm (i).
The full algebra $\KT(\D_n)^{(2)}$ is generated by all $G_i$, $E_i$, for
$i = 1,\ldots n$, and the element $\rho$, consisting of one closed loop around
the pole and vertical strands without crossings of which the one at the far
left also has a pole twist.  See Figure~\ref{pic:rho}.

\rm (ii).  The algebra $\KT(\D_n)$ contains the algebra $\KT(\A_{n-1})$
consisting of tangles which do not go around the pole.  The elements $g_i,\
e_i$ of $\BMW(\D_n)$ for $i\geq 2$ satisfy the relations for $\BMW(\A_{n-1})$
with the usual index increased by~$1$.  The proof of Theorem~\ref{th:surjhomo}
also shows that the map $\phi$ when restricted to $\BMW(\A_{n-1})$ is a
surjective homomorphism onto $\KT(\A_{n-1})$. This was proved by similar means
in \cite{MorWas}.

\rm (iii).  By Proposition \ref{ch5rewrite} and Lemma \ref{ch5KT0},
we can rewrite all closed tangles in $\KT(\D_0)^{(2)}$ to
$\Z[\delta^{\pm1}]$-linear combinations of the identity, $\Xi^+$, and
$\Theta$.
\end{Remarks}

\section{The Brauer diagram algebra of type $\D$}
\label{sec:BrauerTypeD}
In this section we introduce a variation of the original Brauer algebra using
$(n,n)$-tangles of type $\D$.  It involves a variation of the $n$-connectors,
known from the Brauer algebras in their classical sense, cf.~\cite{Bra}, which
we will recall first. The Brauer diagrams of type $\D_n$, to be introduced in
Definition \ref{df:BrDalg}, encompass both the standard diagrammatic
description of elements of the Weyl group $W(\D_n)$ and the diagrammatic
elements of the (generalized) Temperley-Lieb algebra of type $\D_n$ introduced
by Green in \cite{Gre}. Our goals are to show that the algebra defined on
linear combinations of Brauer diagrams of type $\D_n$ is isomorphic to the
Brauer algebra of type $\D_n$ as defined in \cite{CFW} and to utilize this
result to prove Theorem \ref{th:main}. For most of this final section, we use
tangles and related diagrams of type $\D_n$, but at the end we discuss an
alternative approach with the larger class of $(n,n)$ tangles of type
$\D^{(2)}$.

\begin{Def}\label{def:n-connector}
\rm
An $n$-{\em connector} is a pairing on $2n$ points into $n$ disjoint
pairs.
\end{Def}

\np These were used by Brauer in \cite{Bra} to define an algebra over
$\Z[\delta^{\pm1}]$, here called the Brauer diagram algebra.  We take a basis
of these $n$-connectors.  An $n$-connector can be pictured by a diagram.  The
$2n$ points are divided into two sets $\{1,2,\ldots,n\}$ and $\{\bar 1,\bar
2,\ldots,\bar n\}$ of points in the plane with each set on a horizontal line
and point $i\in\{1,\ldots,n\}$ above $\bar i$.  The points are connected by
$n$ strands in the plane as determined by the pairing.

The product $a_1a_2$, for two $n$-connectors $a_1$ and $a_2$, is defined by
stacking their diagrams with $a_1$ on top of $a_2$.  Now identify node $\bar
i$ of the bottom of $a_1$ with node $i$ of the top set of $a_2$, thus
connecting the strands of $a_1$ and $a_2$. We find a new pairing between the
nodes of the top of $a_1$ and the bottom of $a_2$.  Besides a pairing, this
can result in closed loops not connected to the top or bottom.  When a
composition gives $r$ closed loops, the new $n$-connector has coefficient
$\delta^r$.  So $a_1a_2=\delta^ra$ with $a$ the resulting $n$-connector.  This
multiplication is associative.  In fact, these diagrams can be considered to be
$(n,n)$-tangles in which the differences between over and under crossings are
completely ignored.  The multiplication is the same as the multiplication of
the tangle after identification of over and under crossings in the resulting
tangle. 
The classical Brauer diagram algebra, ${\BrD }(\A_{n-1})$, over
$\Z[\delta^{\pm1}]$ is the free $\Z[\delta^{\pm 1}]$-module on the set
of $n$-connectors with the described multiplication.

We will now begin with the analogue for type $\D$.

\begin{Def}\label{def:decoratednconnector}
\rm
A {\em decorated $n$-connector} is an $n$-connector in which an even
number of pairs is labelled $1$. All other pairs are labelled $0$.
A pair labelled $1$ will be called {\em decorated}.
Denote $T_n$ the set of all decorated $n$-connectors.
Denote $T^0_n$ the subset of $T_n$ of decorated $n$-connectors
with no decorations and denote $T^=_n$ the subset of $T_n$ of
decorated $n$-connectors with at least one horizontal pairing.
\end{Def}

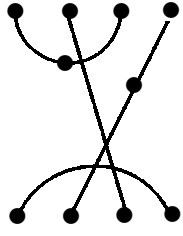
\begin{figure}[htbp]
\unitlength 0.7mm
\begin{picture}(32.25,40)(0,0)
\linethickness{0.3mm}
\put(21.74,1.21){\circle*{2.87}}

\linethickness{0.3mm}
\put(30.81,1.32){\circle*{2.87}}

\linethickness{0.3mm}
\put(21.12,39.96){\circle*{2.87}}

\linethickness{0.3mm}
\put(30.56,40.06){\circle*{2.88}}

\linethickness{0.3mm}
\multiput(1,39)(0.07,-0.49){1}{\line(0,-1){0.49}}
\multiput(1.07,38.51)(0.09,-0.49){1}{\line(0,-1){0.49}}
\multiput(1.16,38.02)(0.12,-0.48){1}{\line(0,-1){0.48}}
\multiput(1.27,37.53)(0.07,-0.24){2}{\line(0,-1){0.24}}
\multiput(1.41,37.05)(0.08,-0.24){2}{\line(0,-1){0.24}}
\multiput(1.57,36.58)(0.09,-0.23){2}{\line(0,-1){0.23}}
\multiput(1.76,36.12)(0.07,-0.15){3}{\line(0,-1){0.15}}
\multiput(1.97,35.67)(0.08,-0.15){3}{\line(0,-1){0.15}}
\multiput(2.2,35.23)(0.08,-0.14){3}{\line(0,-1){0.14}}
\multiput(2.45,34.8)(0.09,-0.14){3}{\line(0,-1){0.14}}
\multiput(2.72,34.38)(0.07,-0.1){4}{\line(0,-1){0.1}}
\multiput(3.02,33.98)(0.08,-0.1){4}{\line(0,-1){0.1}}
\multiput(3.33,33.59)(0.08,-0.09){4}{\line(0,-1){0.09}}
\multiput(3.66,33.22)(0.09,-0.09){4}{\line(0,-1){0.09}}
\multiput(4.01,32.86)(0.09,-0.08){4}{\line(1,0){0.09}}
\multiput(4.38,32.52)(0.1,-0.08){4}{\line(1,0){0.1}}
\multiput(4.76,32.21)(0.1,-0.07){4}{\line(1,0){0.1}}
\multiput(5.16,31.91)(0.14,-0.09){3}{\line(1,0){0.14}}
\multiput(5.57,31.63)(0.14,-0.09){3}{\line(1,0){0.14}}
\multiput(6,31.37)(0.15,-0.08){3}{\line(1,0){0.15}}
\multiput(6.44,31.13)(0.15,-0.07){3}{\line(1,0){0.15}}
\multiput(6.89,30.92)(0.23,-0.1){2}{\line(1,0){0.23}}
\multiput(7.35,30.73)(0.23,-0.08){2}{\line(1,0){0.23}}
\multiput(7.82,30.56)(0.24,-0.07){2}{\line(1,0){0.24}}
\multiput(8.29,30.41)(0.24,-0.06){2}{\line(1,0){0.24}}
\multiput(8.78,30.29)(0.49,-0.1){1}{\line(1,0){0.49}}
\multiput(9.26,30.19)(0.49,-0.07){1}{\line(1,0){0.49}}
\multiput(9.76,30.12)(0.5,-0.05){1}{\line(1,0){0.5}}
\multiput(10.25,30.07)(0.5,-0.02){1}{\line(1,0){0.5}}
\put(10.75,30.04){\line(1,0){0.5}}
\multiput(11.25,30.04)(0.5,0.02){1}{\line(1,0){0.5}}
\multiput(11.75,30.07)(0.5,0.05){1}{\line(1,0){0.5}}
\multiput(12.24,30.12)(0.49,0.07){1}{\line(1,0){0.49}}
\multiput(12.74,30.19)(0.49,0.1){1}{\line(1,0){0.49}}
\multiput(13.22,30.29)(0.24,0.06){2}{\line(1,0){0.24}}
\multiput(13.71,30.41)(0.24,0.07){2}{\line(1,0){0.24}}
\multiput(14.18,30.56)(0.23,0.08){2}{\line(1,0){0.23}}
\multiput(14.65,30.73)(0.23,0.1){2}{\line(1,0){0.23}}
\multiput(15.11,30.92)(0.15,0.07){3}{\line(1,0){0.15}}
\multiput(15.56,31.13)(0.15,0.08){3}{\line(1,0){0.15}}
\multiput(16,31.37)(0.14,0.09){3}{\line(1,0){0.14}}
\multiput(16.43,31.63)(0.14,0.09){3}{\line(1,0){0.14}}
\multiput(16.84,31.91)(0.1,0.07){4}{\line(1,0){0.1}}
\multiput(17.24,32.21)(0.1,0.08){4}{\line(1,0){0.1}}
\multiput(17.62,32.52)(0.09,0.08){4}{\line(1,0){0.09}}
\multiput(17.99,32.86)(0.09,0.09){4}{\line(0,1){0.09}}
\multiput(18.34,33.22)(0.08,0.09){4}{\line(0,1){0.09}}
\multiput(18.67,33.59)(0.08,0.1){4}{\line(0,1){0.1}}
\multiput(18.98,33.98)(0.07,0.1){4}{\line(0,1){0.1}}
\multiput(19.28,34.38)(0.09,0.14){3}{\line(0,1){0.14}}
\multiput(19.55,34.8)(0.08,0.14){3}{\line(0,1){0.14}}
\multiput(19.8,35.23)(0.08,0.15){3}{\line(0,1){0.15}}
\multiput(20.03,35.67)(0.07,0.15){3}{\line(0,1){0.15}}
\multiput(20.24,36.12)(0.09,0.23){2}{\line(0,1){0.23}}
\multiput(20.43,36.58)(0.08,0.24){2}{\line(0,1){0.24}}
\multiput(20.59,37.05)(0.07,0.24){2}{\line(0,1){0.24}}
\multiput(20.73,37.53)(0.12,0.48){1}{\line(0,1){0.48}}
\multiput(20.84,38.02)(0.09,0.49){1}{\line(0,1){0.49}}
\multiput(20.93,38.51)(0.07,0.49){1}{\line(0,1){0.49}}

\linethickness{0.3mm}
\multiput(11,39)(0.08,-0.27){138}{\line(0,-1){0.27}}
\linethickness{0.3mm}
\multiput(12,2)(0.08,0.16){225}{\line(0,1){0.16}}
\linethickness{0.3mm}
\multiput(29.76,2.44)(0.08,-0.15){3}{\line(0,-1){0.15}}
\multiput(29.51,2.87)(0.08,-0.14){3}{\line(0,-1){0.14}}
\multiput(29.24,3.29)(0.09,-0.14){3}{\line(0,-1){0.14}}
\multiput(28.96,3.7)(0.09,-0.14){3}{\line(0,-1){0.14}}
\multiput(28.67,4.1)(0.07,-0.1){4}{\line(0,-1){0.1}}
\multiput(28.37,4.5)(0.08,-0.1){4}{\line(0,-1){0.1}}
\multiput(28.05,4.88)(0.08,-0.1){4}{\line(0,-1){0.1}}
\multiput(27.72,5.26)(0.08,-0.09){4}{\line(0,-1){0.09}}
\multiput(27.38,5.62)(0.09,-0.09){4}{\line(0,-1){0.09}}
\multiput(27.03,5.97)(0.09,-0.09){4}{\line(0,-1){0.09}}
\multiput(26.67,6.31)(0.09,-0.09){4}{\line(1,0){0.09}}
\multiput(26.29,6.64)(0.09,-0.08){4}{\line(1,0){0.09}}
\multiput(25.91,6.96)(0.1,-0.08){4}{\line(1,0){0.1}}
\multiput(25.52,7.27)(0.1,-0.08){4}{\line(1,0){0.1}}
\multiput(25.11,7.56)(0.1,-0.07){4}{\line(1,0){0.1}}
\multiput(24.7,7.84)(0.1,-0.07){4}{\line(1,0){0.1}}
\multiput(24.28,8.11)(0.14,-0.09){3}{\line(1,0){0.14}}
\multiput(23.85,8.36)(0.14,-0.08){3}{\line(1,0){0.14}}
\multiput(23.42,8.6)(0.15,-0.08){3}{\line(1,0){0.15}}
\multiput(22.97,8.83)(0.15,-0.08){3}{\line(1,0){0.15}}
\multiput(22.52,9.04)(0.15,-0.07){3}{\line(1,0){0.15}}
\multiput(22.07,9.24)(0.23,-0.1){2}{\line(1,0){0.23}}
\multiput(21.6,9.42)(0.23,-0.09){2}{\line(1,0){0.23}}
\multiput(21.14,9.59)(0.23,-0.08){2}{\line(1,0){0.23}}
\multiput(20.66,9.75)(0.24,-0.08){2}{\line(1,0){0.24}}
\multiput(20.18,9.89)(0.24,-0.07){2}{\line(1,0){0.24}}
\multiput(19.7,10.01)(0.24,-0.06){2}{\line(1,0){0.24}}
\multiput(19.21,10.12)(0.49,-0.11){1}{\line(1,0){0.49}}
\multiput(18.73,10.21)(0.49,-0.09){1}{\line(1,0){0.49}}
\multiput(18.23,10.29)(0.49,-0.08){1}{\line(1,0){0.49}}
\multiput(17.74,10.35)(0.49,-0.06){1}{\line(1,0){0.49}}
\multiput(17.24,10.4)(0.5,-0.05){1}{\line(1,0){0.5}}
\multiput(16.75,10.43)(0.5,-0.03){1}{\line(1,0){0.5}}
\multiput(16.25,10.45)(0.5,-0.02){1}{\line(1,0){0.5}}
\put(15.75,10.45){\line(1,0){0.5}}
\multiput(15.25,10.43)(0.5,0.02){1}{\line(1,0){0.5}}
\multiput(14.76,10.4)(0.5,0.03){1}{\line(1,0){0.5}}
\multiput(14.26,10.35)(0.5,0.05){1}{\line(1,0){0.5}}
\multiput(13.77,10.29)(0.49,0.06){1}{\line(1,0){0.49}}
\multiput(13.27,10.21)(0.49,0.08){1}{\line(1,0){0.49}}
\multiput(12.79,10.12)(0.49,0.09){1}{\line(1,0){0.49}}
\multiput(12.3,10.01)(0.49,0.11){1}{\line(1,0){0.49}}
\multiput(11.82,9.89)(0.24,0.06){2}{\line(1,0){0.24}}
\multiput(11.34,9.75)(0.24,0.07){2}{\line(1,0){0.24}}
\multiput(10.86,9.59)(0.24,0.08){2}{\line(1,0){0.24}}
\multiput(10.4,9.42)(0.23,0.08){2}{\line(1,0){0.23}}
\multiput(9.93,9.24)(0.23,0.09){2}{\line(1,0){0.23}}
\multiput(9.48,9.04)(0.23,0.1){2}{\line(1,0){0.23}}
\multiput(9.03,8.83)(0.15,0.07){3}{\line(1,0){0.15}}
\multiput(8.58,8.6)(0.15,0.08){3}{\line(1,0){0.15}}
\multiput(8.15,8.36)(0.15,0.08){3}{\line(1,0){0.15}}
\multiput(7.72,8.11)(0.14,0.08){3}{\line(1,0){0.14}}
\multiput(7.3,7.84)(0.14,0.09){3}{\line(1,0){0.14}}
\multiput(6.89,7.56)(0.1,0.07){4}{\line(1,0){0.1}}
\multiput(6.48,7.27)(0.1,0.07){4}{\line(1,0){0.1}}
\multiput(6.09,6.96)(0.1,0.08){4}{\line(1,0){0.1}}
\multiput(5.71,6.64)(0.1,0.08){4}{\line(1,0){0.1}}
\multiput(5.33,6.31)(0.09,0.08){4}{\line(1,0){0.09}}
\multiput(4.97,5.97)(0.09,0.09){4}{\line(1,0){0.09}}
\multiput(4.62,5.62)(0.09,0.09){4}{\line(0,1){0.09}}
\multiput(4.28,5.26)(0.09,0.09){4}{\line(0,1){0.09}}
\multiput(3.95,4.88)(0.08,0.09){4}{\line(0,1){0.09}}
\multiput(3.63,4.5)(0.08,0.1){4}{\line(0,1){0.1}}
\multiput(3.33,4.1)(0.08,0.1){4}{\line(0,1){0.1}}
\multiput(3.04,3.7)(0.07,0.1){4}{\line(0,1){0.1}}
\multiput(2.76,3.29)(0.09,0.14){3}{\line(0,1){0.14}}
\multiput(2.49,2.87)(0.09,0.14){3}{\line(0,1){0.14}}
\multiput(2.24,2.44)(0.08,0.14){3}{\line(0,1){0.14}}
\multiput(2,2)(0.08,0.15){3}{\line(0,1){0.15}}

\linethickness{0.3mm}
\put(11.56,1){\circle*{2.87}}

\linethickness{0.3mm}
\put(1.44,1){\circle*{2.87}}

\linethickness{0.3mm}
\put(23.56,26){\circle*{2.87}}

\linethickness{0.3mm}
\put(11.43,40){\circle*{2.87}}

\linethickness{0.3mm}
\put(1,40){\circle*{2.87}}

\linethickness{0.3mm}
\put(10.43,30){\circle*{2.87}}

\end{picture}
\caption{A decorated $4$-connector with $2$
decorated strands} \label{ch5brauerex}
\end{figure}

We shall shortly (Definition~\ref{df:BrDalg}) describe an algebra,
$\BrD(\D_n)$, using certain decorated $n$-connectors together with some new
parameters; we call it the Brauer diagram algebra of type $\D_n$. This algebra
will turn out to be the image under the map $\psi$ assigning Brauer diagrams
to tangles of type $\D_n$ by identifying over and under crossings,
cf.~Proposition \ref{homomKTDn2CDn}.  The definition will involve some
collections of decorated $n$-connectors.  For these purposes we need the sets
$T^0_n$ and $T^=_n$ defined above.  Recall that $n!! = (2n-1)(2n-3) \cdots 3
\cdot 1$, the product of the first $n$ odd integers.

\begin{Lm} \label{ch5Fn}
The sizes of the sets just defined are as follows.
$|T_n| = 2^{n-1}n!!$,
$|T^0_n| = n!!$, and $|T^=_n| = 2^{n-1}(n!!-n!)$. Moreover,
$|T^=_n \cap T^0_n| = n!!-n!$.
\end{Lm}

\begin{proof}
The $2n$ points can be paired in $n!!$ ways, so $|T^0_n| = n!!$.
Each strand can randomly be labelled $0$ or $1$, except for the
last one which is obliged to make the total number of decorations
even. Hence, $|T_n| = 2^{n-1}n!!$.

\np To obtain a decorated $n$-connector without horizontal strands the $2n$
points can be paired in $n!$ ways. So there are $n!!-n!$ pairings which have
at least one horizontal pair. These are all decorated $n$-connectors with at
least one horizontal strand but without decorations, so $|T^=_n \cap T^0_n| =
n!!-n!$. Again, these pairings can be decorated in $2^{n-1}$ ways, so
$|T^=_n|=2^{n-1}(n!!-n!)$. \end{proof}

For Brauer diagrams of type $\D$, the role of the group $\la\delta^{\pm1}\ra$
for Brauer diagrams of type $\A$ will to some extent be taken over by the
commutative monoid with presentation
$$\Lambda = \la\delta^{\pm1}, \xi,\theta\mid \xi^2=\delta^2, \xi\theta =
\delta\theta, \theta^2 = \delta^2\theta \ra.$$ In particular, we allow for
decorated $n$-connectors to be multiplied with coefficients (central elements)
from $\Lambda$.
Notice $\Lambda = \la\delta^{\pm1}\ra\{1,\xi,\theta\}$.

\begin{Def}\label{df:BrDalg}
\rm A {\em Brauer diagram} of type $\D_n$ is the scalar multiple of a
decorated $n$-connector with an element of $\Lambda$ belonging to $\la
\delta^{\pm1}\ra \left(T_n\cup \xi
T^=_n\cup\theta (T^0_n\cap T^=_n)\right)$.
\label{df:BrD}
The {\em Brauer diagram algebra} of type $\D_n$, notation $\BrD(\D_n)$, is the
$\Z[\delta^{\pm1}]$-linear span of all Brauer diagrams of type $\D_n$ with
multiplication defined by $\Z[\delta^{\pm1}]$-bilinear extension of the
multiplication of two Brauer diagrams $\lambda_1 f_1$ and $\lambda_2 f_2$,
with $\lambda_1$, $\lambda_2\in\{1,\xi,\theta\}$ and $f_1$, $f_2$ decorated
$n$-connectors, determined by the following five steps. Here, the product
$\lambda_1 f_1\lambda_2f_2$ is of the form $\lambda f$ where $f$ is a
  decorated $n$-connector and $\lambda\in\Lambda$.

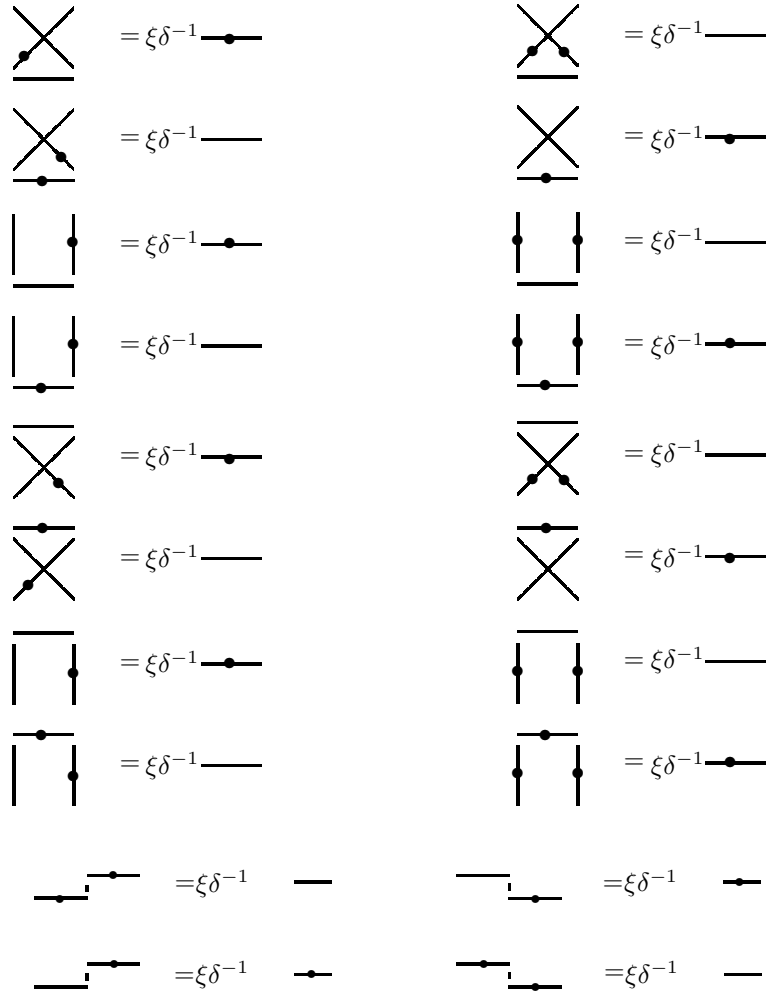
\begin{figure}[htbp]
\unitlength .8mm
\begin{picture}(138.88,63.38)(0,0)
\linethickness{0.3mm}
\put(1.23,18.62){\line(0,1){10}}
\linethickness{0.3mm}
\put(11.23,18.62){\line(0,1){10}}
\linethickness{0.3mm}
\put(1.23,16.74){\line(1,0){10}}
\linethickness{0.3mm}
\put(11,24){\circle*{1.7}}

\linethickness{0.3mm}
\multiput(1.23,53)(0.08,0.08){125}{\line(1,0){0.08}}
\linethickness{0.3mm}
\multiput(1.23,63)(0.08,-0.08){125}{\line(1,0){0.08}}
\linethickness{0.3mm}
\put(1.23,51.12){\line(1,0){10}}
\linethickness{0.3mm}
\put(32.49,58){\line(1,0){10}}
\put(20.61,58){\makebox(0,0)[cc]{$=$}}

\linethickness{0.3mm}
\put(37.09,57.77){\circle*{1.7}}

\linethickness{0.3mm}
\put(3,55){\circle*{1.7}}

\linethickness{0.3mm}
\multiput(1.23,36.12)(0.08,0.08){125}{\line(1,0){0.08}}
\linethickness{0.3mm}
\multiput(1.23,46.12)(0.08,-0.08){125}{\line(1,0){0.08}}
\linethickness{0.3mm}
\put(1.23,34.24){\line(1,0){10.01}}
\linethickness{0.3mm}
\put(32.49,41.12){\line(1,0){10}}
\put(20.61,41.12){\makebox(0,0)[cc]{$=$}}

\linethickness{0.3mm}
\put(6.01,34.25){\circle*{1.7}}

\linethickness{0.3mm}
\put(9.15,38.15){\circle*{1.7}}

\put(27.61,58.12){\makebox(0,0)[cc]{$\xi\delta^{-1}$}}

\linethickness{0.3mm}
\put(27.63,41.12){\makebox(0,0)[cc]{$\xi\delta^{-1}$}}

\put(27.63,6.75){\makebox(0,0)[cc]{$\xi\delta^{-1}$}}

\linethickness{0.3mm}
\put(1.23,1.74){\line(0,1){10}}
\linethickness{0.3mm}
\put(11.23,1.74){\line(0,1){10}}
\linethickness{0.3mm}
\put(1.23,-0.13){\line(1,0){10}}
\linethickness{0.3mm}
\put(32.49,6.74){\line(1,0){10}}
\put(20.61,6.74){\makebox(0,0)[cc]{$=$}}

\linethickness{0.3mm}
\put(5.84,-0.13){\circle*{1.7}}

\linethickness{0.3mm}
\put(11.15,7.15){\circle*{1.7}}

\linethickness{0.3mm}
\put(32.49,23.62){\line(1,0){10}}
\put(20.61,23.62){\makebox(0,0)[cc]{$=$}}

\linethickness{0.3mm}
\put(37.09,23.85){\circle*{1.7}}

\linethickness{0.3mm}
\put(27.63,23.62){\makebox(0,0)[cc]{$\xi\delta^{-1}$}}

\linethickness{0.3mm}
\put(85,19){\line(0,1){10}}
\linethickness{0.3mm}
\put(95,19){\line(0,1){10}}
\linethickness{0.3mm}
\put(85,17.12){\line(1,0){10}}
\linethickness{0.3mm}
\put(85,24.4){\circle*{1.7}}

\linethickness{0.3mm}
\multiput(85,53.38)(0.08,0.08){125}{\line(1,0){0.08}}
\linethickness{0.3mm}
\multiput(85,63.38)(0.08,-0.08){125}{\line(1,0){0.08}}
\linethickness{0.3mm}
\put(85,51.5){\line(1,0){10}}
\linethickness{0.3mm}
\put(116.25,58.38){\line(1,0){10}}
\put(104.38,58.38){\makebox(0,0)[cc]{$=$}}

\linethickness{0.3mm}
\put(87.5,55.88){\circle*{1.71}}

\linethickness{0.3mm}
\put(92.73,55.65){\circle*{1.7}}

\linethickness{0.3mm}
\multiput(85,36.5)(0.08,0.08){125}{\line(1,0){0.08}}
\linethickness{0.3mm}
\multiput(85,46.5)(0.08,-0.08){125}{\line(1,0){0.08}}
\linethickness{0.3mm}
\put(85,34.62){\line(1,0){10}}
\linethickness{0.3mm}
\put(116.25,41.5){\line(1,0){10}}
\put(104.38,41.5){\makebox(0,0)[cc]{$=$}}

\linethickness{0.3mm}
\put(89.78,34.63){\circle*{1.71}}

\linethickness{0.3mm}
\put(120.23,41.27){\circle*{1.7}}

\linethickness{0.3mm}
\put(116.25,24){\line(1,0){10}}
\put(104.38,24){\makebox(0,0)[cc]{$=$}}

\linethickness{0.3mm}
\put(95,24.4){\circle*{1.7}}

\linethickness{0.3mm}
\put(95,7.52){\circle*{1.7}}

\linethickness{0.3mm}
\put(120.31,7.29){\circle*{1.7}}

\linethickness{0.3mm}
\put(85,2.12){\line(0,1){10}}
\linethickness{0.3mm}
\put(95,2.12){\line(0,1){10}}
\linethickness{0.3mm}
\put(85,0.25){\line(1,0){10}}
\linethickness{0.3mm}
\put(116.25,7.12){\line(1,0){10}}
\put(104.38,7.12){\makebox(0,0)[cc]{$=$}}

\linethickness{0.3mm}
\put(89.6,0.25){\circle*{1.7}}
\put(85,7.52){\circle*{1.7}}
\put(111.59,58.62){\makebox(0,0)[cc]{$\xi\delta^{-1}$}}

\linethickness{0.3mm}

\put(111.59,41.12){\makebox(0,0)[cc]{$\xi\delta^{-1}$}}
\put(111.59,24.25){\makebox(0,0)[cc]{$\xi\delta^{-1}$}}
\put(111.59,6.75){\makebox(0,0)[cc]{$\xi\delta^{-1}$}}

\end{picture}

\medskip\bigskip

\unitlength 0.8mm
\begin{picture}(138.88,63.75)(0,0)
\linethickness{0.3mm}
\put(32.49,58){\line(1,0){10}}
\put(20.61,58){\makebox(0,0)[cc]{$=$}}
\linethickness{0.3mm}
\put(37.09,57.77){\circle*{1.7}}
\linethickness{0.3mm}
\put(32.49,41.12){\line(1,0){10}}
\put(20.61,41.12){\makebox(0,0)[cc]{$=$}}

\put(27.61,58.12){\makebox(0,0)[cc]{$\xi\delta^{-1}$}}

\linethickness{0.3mm}
\put(27.63,41.12){\makebox(0,0)[cc]{$\xi\delta^{-1}$}}

\put(27.63,6.75){\makebox(0,0)[cc]{$\xi\delta^{-1}$}}

\linethickness{0.3mm}
\put(32.49,6.74){\line(1,0){10}}
\put(20.61,6.74){\makebox(0,0)[cc]{$=$}}

\linethickness{0.3mm}
\put(32.49,23.62){\line(1,0){10}}
\put(20.61,23.62){\makebox(0,0)[cc]{$=$}}

\linethickness{0.3mm}
\put(37.09,23.85){\circle*{1.71}}

\linethickness{0.3mm}
\put(27.63,23.62){\makebox(0,0)[cc]{$\xi\delta^{-1}$}}

\linethickness{0.3mm}
\put(116.25,41.5){\line(1,0){10}}
\put(104.38,41.5){\makebox(0,0)[cc]{$=$}}

\linethickness{0.3mm}
\put(120.23,41.27){\circle*{1.7}}

\linethickness{0.3mm}
\put(116.25,24){\line(1,0){10}}
\put(104.38,24){\makebox(0,0)[cc]{$=$}}

\linethickness{0.3mm}
\put(120.31,7.29){\circle*{1.7}}

\linethickness{0.3mm}
\put(116.25,7.12){\line(1,0){10}}
\put(104.38,7.12){\makebox(0,0)[cc]{$=$}}

\linethickness{0.3mm}
\put(111.59,41.12){\makebox(0,0)[cc]{$\xi\delta^{-1}$}}

\put(111.59,24.25){\makebox(0,0)[cc]{$\xi\delta^{-1}$}}

\put(111.59,6.75){\makebox(0,0)[cc]{$\xi\delta^{-1}$}}

\linethickness{0.3mm}
\put(85,29){\line(1,0){10}}
\linethickness{0.3mm}
\put(85,22.4){\circle*{1.7}}

\linethickness{0.3mm}
\put(95,22.4){\circle*{1.7}}

\linethickness{0.3mm}
\put(85,17){\line(0,1){10}}
\linethickness{0.3mm}
\put(95,17){\line(0,1){10}}
\linethickness{0.3mm}
\multiput(85,34.38)(0.08,0.08){125}{\line(1,0){0.08}}
\linethickness{0.3mm}
\multiput(85,44.38)(0.08,-0.08){125}{\line(1,0){0.08}}
\linethickness{0.3mm}
\put(89.77,46.26){\circle*{1.7}}

\linethickness{0.3mm}
\put(85,46.25){\line(1,0){10}}
\linethickness{0.3mm}
\put(95,5.4){\circle*{1.7}}

\linethickness{0.3mm}
\put(85,5.4){\circle*{1.7}}

\linethickness{0.3mm}
\put(85,0){\line(0,1){10}}
\linethickness{0.3mm}
\put(95,0){\line(0,1){10}}
\linethickness{0.3mm}
\put(89.6,11.87){\circle*{1.7}}

\linethickness{0.3mm}
\put(85,11.88){\line(1,0){10}}
\linethickness{0.3mm}
\put(85,63.75){\line(1,0){10}}
\linethickness{0.3mm}
\put(87.5,54.37){\circle*{1.71}}

\linethickness{0.3mm}
\put(92.73,54.14){\circle*{1.7}}

\linethickness{0.3mm}
\multiput(85,51.88)(0.08,0.08){125}{\line(1,0){0.08}}
\linethickness{0.3mm}
\multiput(85,61.87)(0.08,-0.08){125}{\line(1,0){0.08}}
\linethickness{0.3mm}
\put(5.86,11.87){\circle*{1.7}}

\linethickness{0.3mm}
\put(1.25,11.88){\line(1,0){10}}
\linethickness{0.3mm}
\put(1.25,28.75){\line(1,0){10}}
\linethickness{0.3mm}
\put(6.02,46.26){\circle*{1.7}}

\linethickness{0.3mm}
\put(1.24,46.25){\line(1,0){10.01}}
\linethickness{0.3mm}
\put(1.25,63.12){\line(1,0){10}}
\linethickness{0.3mm}
\put(116.25,58.38){\line(1,0){10}}
\put(104.38,58.38){\makebox(0,0)[cc]{$=$}}

\put(111.59,58.62){\makebox(0,0)[cc]{$\xi\delta^{-1}$}}

\linethickness{0.3mm}
\put(11.15,22.15){\circle*{1.7}}

\linethickness{0.3mm}
\put(1.25,16.88){\line(0,1){10}}
\linethickness{0.3mm}
\put(11.25,16.88){\line(0,1){10}}
\linethickness{0.3mm}
\put(3.69,36.79){\circle*{1.7}}

\linethickness{0.3mm}
\multiput(1.24,34.38)(0.08,0.08){125}{\line(1,0){0.08}}
\linethickness{0.3mm}
\multiput(1.24,44.38)(0.08,-0.08){125}{\line(1,0){0.08}}
\linethickness{0.3mm}
\put(8.69,53.79){\circle*{1.7}}

\linethickness{0.3mm}
\multiput(1.25,51.25)(0.08,0.08){125}{\line(1,0){0.08}}
\linethickness{0.3mm}
\multiput(1.25,61.25)(0.08,-0.08){125}{\line(1,0){0.08}}
\linethickness{0.3mm}
\put(11.15,5){\circle*{1.7}}

\linethickness{0.3mm}
\put(1.25,0){\line(0,1){10}}
\linethickness{0.3mm}
\put(11.25,0){\line(0,1){10}}
\end{picture}
\label{ch5brauerxi}
\unitlength 0.7mm
\begin{picture}(147.69,43.54)(0,0)
\linethickness{0.3mm}
\put(133.75,29.6){\circle*{1.7}}

\linethickness{0.3mm}
\put(4.63,26.5){\circle*{1.7}}

\linethickness{0.3mm}
\put(14.85,30.88){\circle*{1.71}}

\linethickness{0.3mm}
\put(-0.15,26.5){\line(1,0){10}}
\linethickness{0.3mm}
\put(9.85,30.88){\line(1,0){10}}
\linethickness{0.3mm}
\put(9.85,27.75){\line(0,1){1.25}}
\linethickness{0.3mm}
\put(15,14){\circle*{1.71}}

\linethickness{0.3mm}
\put(-0.15,9.62){\line(1,0){10}}
\linethickness{0.3mm}
\put(9.85,14){\line(1,0){10}}
\linethickness{0.3mm}
\put(9.85,10.88){\line(0,1){1.24}}
\put(28.35,11.59){\makebox(0,0)[cc]{$=$}}

\linethickness{0.3mm}
\put(52.5,12.1){\circle*{1.7}}

\linethickness{0.3mm}
\put(35.56,12){\makebox(0,0)[cc]{$\xi\delta^{-1}$}}

\linethickness{0.3mm}
\put(49.25,12){\line(1,0){7}}
\put(28.42,29){\makebox(0,0)[cc]{$=$}}

\linethickness{0.3mm}
\put(35.62,29.33){\makebox(0,0)[cc]{$\xi\delta^{-1}$}}

\linethickness{0.3mm}
\put(49.32,29.57){\line(1,0){7}}
\put(109.81,29.09){\makebox(0,0)[cc]{$=$}}

\linethickness{0.3mm}
\put(117.02,29.33){\makebox(0,0)[cc]{$\xi\delta^{-1}$}}
\linethickness{0.3mm}
\put(130.71,29.57){\line(1,0){7}}
\put(110,12){\makebox(0,0)[cc]{$=$}}

\linethickness{0.3mm}
\put(117.21,12){\makebox(0,0)[cc]{$\xi\delta^{-1}$}}

\linethickness{0.3mm}
\put(130.9,12){\line(1,0){7}}
\linethickness{0.3mm}
\put(90,28){\line(0,1){1.25}}
\linethickness{0.3mm}
\put(95,26.47){\circle*{1.71}}
\linethickness{0.3mm}
\put(90,26.47){\line(1,0){10}}
\linethickness{0.3mm}
\put(80,30.85){\line(1,0){10}}
\linethickness{0.3mm}
\put(90,11.23){\line(0,1){1.25}}
\linethickness{0.3mm}
\put(95,9.6){\circle*{1.71}}

\linethickness{0.3mm}
\put(90,9.6){\line(1,0){10}}
\linethickness{0.3mm}
\put(85.23,13.97){\circle*{1.7}}

\linethickness{0.3mm}
\put(80,13.97){\line(1,0){10}}
\end{picture}
\caption{Rules for the composition of two vertical pairs and a horizontal pair
in a product of two Brauer diagrams}
\label{ch5brauerxi3big}
\end{figure}

\begin{enumerate}[{\rm  (i)}] 
\item As before, draw the diagrams $f_1$ and $f_2$, stack them, and identify
point $\bar i$ at the bottom of $f_1$ with $i$ at the top of $f_2$.  The
points at the top of $f_1$ will be identified with those for $f$ and similarly
for the points at the bottom of $f_2$.

\item Determine the pairing of $f$ as before: for a point at the top of $f_1$
or the bottom of $f_2$, follow the strand until it ends in a point at the top
of $f_1$ or the bottom of $f_2$. This results in a new pair for $f$.

\item Set $\lambda=\lambda_1\lambda_2$. For each straightening step in a
concatenation of pairs as carried out in the previous step, check if the
pattern shrunk to a straight horizontal line segment occurs as the left hand
side of an equality in Figure~\ref{ch5brauerxi3big}. If so, multiply $\lambda$
by $\xi\delta^{-1}$; otherwise, $\lambda$ is not changed.  (Compare with the
left hand picture of Figure \ref{ch5brauerbasica}; this pattern as well as
each triple of straight line segments forming a shape appearing in Figure
\ref{ch5brauerxi3big} but whose decoration pattern does not appear in Figure
\ref{ch5brauerxi3big}, does not change $\lambda$.)

\item At this stage, only closed loops remain.  Closed loops come from strands
which have no endpoints in $f$.  First simplify loops by removing crossings as
in (iii), i.e. by use of Figure \ref{ch5brauerxi3big} (again, the
configurations not appearing in the figure do not give $\xi\delta^{-1}$) and
shrink them using the rules on the two bottom lines of Figure
\ref{ch5brauerxi3big} (at this stage, factors $\xi\delta^{-1}$ may emerge).
Next replace each closed loop without decoration by $\delta$ (that is, remove
the loop and multiply $\lambda$ by $\delta$) and each pair of disjoint closed
decorated loops by $\theta$.  As the number of decorated pairs is even, what
might remain is a simple decorated loop in the presence of a decorated pair;
if so, undecorate the pair (i.e., give it label $0$), remove the decorated
loop and multiply $\lambda$ by $\theta\delta^{-1}$.  (Compare with the right
hand picture of Figure \ref{ch5brauerbasica}.)

\begin{figure} [hbtp]
\unitlength 0.9mm
\begin{picture}(115.64,17)(0,5)
\put(23,11){\makebox(0,0)[cc]{$=$}}
\linethickness{0.3mm}
\put(13,8){\circle*{1.71}}
\put(13,14){\circle*{1.71}}
\put(13,6.5){\line(0,1){10}}
\put(33,6.5){\line(0,1){10}}
\put(67.89,11){\circle*{1.71}}
\put(70.96,11){\circle{6.37}}
\put(76.65,6.5){\line(0,1){10}}
\put(76.64,11){\circle*{1.71}}
\put(83,11){\makebox(0,0)[cc]{$=$}}
\put(91.25,11){\makebox(0,0)[cc]{$\theta\delta^{-1}$}}
\put(99.24,6.5){\line(0,1){10}}
\end{picture}
\caption{Two rules for decorated $n$-connectors}\label{ch5brauerbasica}
\end{figure}

\item If $\theta$ is a factor of $\lambda$, remove all decorations from
$f$.
\end{enumerate}
\end{Def}

Thus, $\BrD(\D_n)$ is a free $\Z[\delta^{\pm1}]$-module.  However, there is
some indeterminacy in the description of the multiplication due to the order
in which patterns are rewritten in (iii) and (iv). For instance in (iii),
instead of applying the fifth rule from the top of the right column of Figure
\ref{ch5brauerxi3big} to the top three strands of the right hand picture for
$\xi$ in Figure \ref{ch5closedloop}, we could have applied the first rule of
the same column of Figure \ref{ch5brauerxi3big} to the bottom three strands.

\begin{figure}[htbp]
\unitlength 1mm
\begin{picture}(26,15)(0,0)
\linethickness{0.3mm}
\put(10,9.41){\circle{11.18}}
\put(20,10){\makebox(0,0)[cc]{$=$}}
\put(26,10){\makebox(0,0)[cc]{$\delta$,}}
\end{picture}
$\qquad$
\unitlength 1mm
\begin{picture}(36.59,14.18)(0,0)
\put(30.59,9.18){\makebox(0,0)[cc]{$=$}}
\put(36.59,9.18){\makebox(0,0)[cc]{$\theta$,}}
\linethickness{0.3mm}
\put(20.59,8.59){\circle{11.18}}
\put(21.09,3.18){\circle*{2}}
\put(7,8.5){\circle{11.18}}
\put(7.5,3.09){\circle*{2}}
\end{picture}
$\qquad$
\unitlength 1mm
\begin{picture}(27,15)(0,0)
\linethickness{0.3mm}
\put(8.4,12){\circle*{2}}
\put(9,8){\circle*{2}}
\multiput(6.25,5.62)(0.08,0.08){100}{\line(0,1){0.09}}
\multiput(6.25,13.75)(0.08,-0.08){96}{\line(1,0){0.09}}
\put(20,10){\makebox(0,0)[cc]{$=$}}
\put(27,10){\makebox(0,0)[cc]{$\xi$}}
\multiput(6.25,13.75)(0.08,0.00){94}{\line(0,1){0.09}}
\multiput(6.25,5.62)(0.08,0.00){94}{\line(0,1){0.09}}
\end{picture}
\caption{Three closed loops occurring in
Brauer diagrams} \label{ch5closedloop}
\end{figure}

The proposition below shows that the multiplication on $\BrD(\D_n)$ is well
defined by means of a $\Z[\delta^{\pm1}]$-linear map $\nu :
\Br(\D_n)\to\BrD(\D_n)$ which we introduce first.

Recall from \cite{CFW} the set $\AO_0$ of highest elements of $W$-orbits in
$\AO$, the set of all admissible sets of mutually orthogonal positive roots.
We will use the action of $\Br(\D_n)$ on $\AO$, denoted $\sigma$ in
\cite[Theorem 3.6(i)]{CFW}.  For $X\in \AO_0$, we found a set $C_{WX}$ of
nodes of $\D_n$ whose corresponding simple roots are orthogonal to $X$. We
also defined $D_Y$ for $Y\in \AO$ as a set of minimal length coset
representatives of $N_W(Y)$ in $W$.  By $\Phi$ we denote the root system of
type $\D_n$. Its positive roots will be taken to be the vectors $\eps_j\pm
\eps_i$ with $j>i$ in $\R^n$ with orthonormal basis
$\eps_1,\ldots,\eps_n$. The simple roots are $\a_1 = \eps_2+\eps_1$ and $\a_i
= \eps_i-\eps_{i-1}$ for $i = 2,\ldots,n$.  If $\a$ is a root of $\Phi$, then
$\a^*$ denotes its orthogonal mate, that is, the unique other positive root
orthogonal to all roots orthogonal to $\a$, see \cite[Definition 3.1]{CGWBMW}.
If $\a = \eps_j-\eps_i$, then $\a^* = \eps_j+\eps_i$ and $(\a^*)^* = \a$.  We
write $r_\a^*$ for $r_{\a^*}$ and $r_n^*$ for the reflection with root
$\a_n^*$ and, similarly, $e_\a^*$ for $e_{\a^*}$ and $e_i^*$ for
$e_{\a_i}^*$.  We will also use the natural permutation action of $W$ on
$\{\pm1,\ldots,\pm n\}$, which maps $r_{\eps_j-\eps_i}$ to $ (i,j)(-i,-j)$
and $r_{\eps_j+\eps_i}$ to $(i,-j)(-i,j)$.

\begin{Def}
\label{df:nu}
\rm By \cite[Corollary 5.5]{CFW}, $\Br(\D_n)$ has a basis over
$\Z[\delta^{\pm1}]$ consisting of the elements $ue_Xzv$ with 
$X\in\AO_0$, $u,v^{-1}\in
D_X$, and $z\in W(C_{WX})$.  By \cite[Lemma 1.2]{CFW}, either
there is a root $\b\in\Phi^+$ such that both $\b$ and $\b^*$ belong to $X$, or
$X$ is in the same $W$-orbit as $Y(t)$, for some $t\in\{1,\ldots,\lfloor
n/2\rfloor\}$, or $Y'(n/2)$ (in which case $n$ is even), where $Y(t) =
\{\a_n,\a_{n-2},\ldots,\a_{n-2t+2}\}$ and $Y'(n/2) =
\{\a_n,\a_{n-2},\ldots,\a_{4},\a_1\}$, cf.\ \cite[Section 3]{CGWBMW}. It will
be more convenient for us to write the basis elements of the latter kind with
$Y(t)$, respectively $Y'(n/2)$, instead of the highest element $X$. Observe
that for this purpose $W(C_{WX})$ needs to be replaced by $\la r_n^*\ra
\times\la r_1,\ldots,r_{n-2t}\ra$.

The {\em Brauer diagram} corresponding to $a = ue_Xzv$
with $X = Y(t)$, $Y'(t)$, or $Y(t)\cup Y(t)^*$,
$u$, $v^{-1}\in D_X$, and $z = (r_n^*)^k z_0$ for some $k\in\{0,1\}$ and
$z_0\in \la r_1,\ldots,r_{n-2t}\ra$,
is $\nu(a) = \lambda f$ where $\lambda\in\Lambda$ and the decorated
$n$-connector $f$ are as follows.  We first describe the horizontal strands of
$f$.  If $\eps_j-\eps_i\in uX = a\emptyset$, then $f$ has an undecorated
horizontal strand pairing $i$ and $j$ at the top.  Similarly, if
$\eps_j-\eps_i\in v^{-1}X = a^{\op}\emptyset$ (where ${a}^{\op}$ is obtained
from $a$ by reversing an expression for $a$, see \cite[Remark 5.7]{CFW}) then
$f$ has an undecorated horizontal strand pairing $\bar i$ and $\bar j$ at the
bottom.  Dually, if $\eps_j+\eps_i\in uX$ and $\eps_j-\eps_i\not\in uX$, then
$f$ has a decorated horizontal strand pairing $i$ and $j$ at the top, and if
$\eps_j+\eps_i\in v^{-1}X$ but $\eps_j-\eps_i\not\in v^{-1}X$, then $f$ has a
decorated horizontal strand pairing $\bar i$ and $\bar j$ at the bottom.

Suppose $\b^*\in X $ for some $\b\in X$. Then, as $X$ is admissible, $\b^*\in
X$ for all $\b\in X$.  Put $\lambda = \theta\delta^{|X|/2-2}$. (This is equal
to $(\theta\delta^{-1})^{|X|/2}$.)  As for the vertical strands of $f$, we
connect $\bar j$ at the bottom with $j\not\in v^{-1}X$ to $i$ at the top with
an undecorated strand whenever $i = \pm uzv(j)$. This finishes the description
of $\nu(a)$ in case $\b,\b^*\in X$ for some $\b$.

Suppose now that $X = Y(t)$ for some $t\in\{1,\ldots,\lfloor
n/2\rfloor\}$. The treatment of $Y'(n/2)$ is the same as the treatment of
$Y(n/2)$ with the nodes $1$ and $2$ of $\D_n$ interchanged.  Therefore, we do
not discuss it further.  Now $z = (r_n^*)^k z_0$ for some $k\in\{0,1\}$ and
$z_0\in \la r_1,\ldots, r_{n-2t}\ra$.  Put $\lambda = (\xi\delta^{-1})^k$.
Finally, for the vertical strands of $f$, connect $\bar j$ at the bottom to
$i$ at the top with an undecorated strand whenever $i = uzv(j)$ and with a
decorated strand if $i = - uzv(j)$.  This completes the definition of
$\lambda$ and $f$ and hence of the Brauer diagram $\nu(a)$.

Observe that $\nu(a)$ is well defined. For instance, if $u$ and $u'$ are
elements of $W$ of minimal length with $uX = u'X$, then $uzv$ and $u'zv$ have
the same action on the roots orthogonal to $v^{-1}X$, and similarly for $v$
and $v'$ of minimal length with $v^{-1}X = {v'}^{-1}X$, see \cite[Lemma
4.8(i)]{CFW}. Therefore $uzv(j)$ for $j$ not occurring as an index in
$v^{-1}X$, does not depend on the choice of $D_X$.

So far we have defined $\nu(a)$ for $a$ belonging to a basis of
$\Br(\D_n)$. By \cite[Theorem 1.1]{CFW}, the latter is a free
$Z[\delta^{\pm1}]$-module, so $\nu$ can be extended by
$Z[\delta^{\pm1}]$-linearity to a map $\Br(\D_n)\to \BrD(\D_n)$,
which will also be denoted $\nu$.
\end{Def}

\begin{Remarks}\label{rmks:afternu}
\rm(i).
For example, the Brauer diagrams $\nu(e_i)$ and $\nu(r_i)$ are as for $E_i$
and $G_i$ in Figures \ref{picsGiEi} and \ref{ch5G0E0} after the twists around
the pole have been replaced by decorations.  Moreover, the version of $E_i$
with $i\ge2$ in which both strands are decorated is equal to $\nu(e_i^*)$.

\nl(ii).  The Weyl group $W$ of type $\D_n$ can be diagrammatically described
as follows by decorated diagrams having only vertical pairs: for $\b =
\eps_j \pm \eps_i$, the reflection $r_\b$ is depicted by the diagram all of
whose pairs are straight downwards and without decorations, except two
vertical pairs connecting $i$ and $j$ at the top with $\bar j$ and $\bar i$
at the bottom, respectively, and both decorated if $\b = \eps_j+\eps_i$, and
neither decorated if $\b = \eps_j-\eps_i$.  Also, it is readily checked that
the usual multiplication rules, with the one of the left picture in Figure
\ref{ch5brauerbasica} for the reduction of decorations if there are more than
one per strand, suffice for the description of $W$ in terms of diagrams.  This
describes the restriction of $\nu$ to $W$.

\nl(iii).
Restriction of $\nu$ to $\Br(\A_{n-1})$, where $\A_{n-1}$ stands
for the set of nodes $\{2,\ldots,n\}$ of $\D_n$, gives an isomorphism between
the Brauer algebra, $\Br(\A_{n-1})$ and the Brauer diagram algebra
$\BrD(\A_{n-1})$, the $\Z[\delta^{\pm1}]$-linear span of all undecorated
$n$-connectors (without scalars $\xi$ and $\theta$), see Figure
\ref{picsGiEi}.  This fact has been known since the time of \cite{Wen}.
\label{th:BrAnIso}

\nl(iv).
\label{rmk:BrauerDs}
The equality involving the rightmost closed loop of Figure \ref{ch5closedloop}
is a consequence of the fifth rule from the top in the right column of Figure
\ref{ch5brauerxi3big} (applied to the top three strands with a decoration on
each of the two diagonal strands). For, after removing the crossing by that
rule, a closed undecorated loop remains, so the whole loop can be replaced by
$(\xi\delta^{-1})\delta$, which is equal to $\xi$. Similarly the top rule on
the right can be used.

\nl(v).  The diagrams given in the two last rows of
Figure~\ref{ch5brauerxi3big}, where the common endpoint of the two horizontal
strands is in the middle, suffice for describing all computations involving
multiple horizontal strands.  An example is given in Figure~\ref{ch5asso2}.

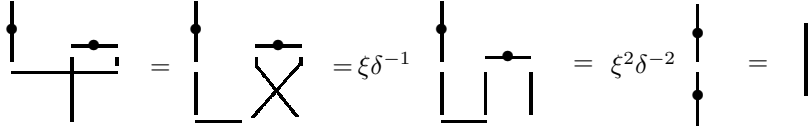
\begin{figure}[htbp]
\unitlength .8mm
\begin{picture}(150,20)(0,0)
\linethickness{0.3mm}
\put(13.69,13.43){\circle*{1.71}}

\put(0,10.63){\line(0,1){10}}

\put(10,0.63){\line(0,1){10}}

\put(0,8.75){\line(1,0){17.5}}

\put(10,13.13){\line(1,0){7.5}}

\put(10,10){\line(0,1){1.24}}

\put(17.5,10){\line(0,1){1.24}}

\put(0,16){\circle*{1.71}}

\put(24.75,9.38){\makebox(0,0)[cc]{$=$}}

\multiput(40.62,10.01)(0.08,-0.1){86}{\line(0,-1){0.1}}

\multiput(40,1.25)(0.08,0.09){101}{\line(0,1){0.09}}

\put(40.62,10.01){\line(0,1){1.24}}

\put(48.12,10.01){\line(0,1){1.24}}

\put(44.31,13.43){\circle*{1.71}}

\put(40.62,13.13){\line(1,0){7.5}}

\put(30.62,10.63){\line(0,1){10}}

\put(30.62,15.99){\circle*{1.71}}

\put(30.62,1.88){\line(0,1){6.87}}

\put(30.62,0.62){\line(1,0){7.5}}
\put(55,9.38){\makebox(0,0)[cc]{$=$}}

\put(62,10){\makebox(0,0)[cc]{$\xi\delta^{-1}$}}

\put(71.48,10.62){\line(0,1){10}}

\put(71.48,16){\circle*{1.71}}

\put(71.48,1.87){\line(0,1){6.87}}

\put(71.48,0.62){\line(1,0){7.5}}

\put(82.44,11.55){\circle*{1.71}}

\put(78.75,11.25){\line(1,0){7.5}}

\put(78.75,1.88){\line(0,1){6.87}}

\put(86.25,1.88){\line(0,1){6.87}}

\put(95,10){\makebox(0,0)[cc]{$=$}}

\put(105,10){\makebox(0,0)[cc]{$\xi^2\delta^{-2}$}}

\put(114,10.62){\line(0,1){9.38}}
\put(114,15.36){\circle*{1.71}}

\put(114,0){\line(0,1){8.75}}

\put(114,5){\circle*{1.71}}

\put(124,10){\makebox(0,0)[cc]{$=$}}

\put(132,5){\line(0,1){11.88}}
\end{picture}
\caption{Multiplication involving horizontal pairs} 
\label{ch5asso2}
\end{figure}

It is straightforward to check that when two horizontal strands have their
common endpoint on the left, the scalar $\xi\delta^{-1}$ only comes in when
precisely one of the two strands has a decoration. When two horizontal strands
have their common node on the right, the scalar $\xi$ nevers appears, as in
Figure \ref{ch5asso2}.

\nl(vi) The map $\nu$ satisfies the following identities for positive roots
$\a$, respectively $\a$ and $\b$, with $|(\a,\b)| = 1$.
$$\nu(r_\a^* e_\a) = \xi\delta^{-1}\nu(e_\a), \qquad \nu(r_\a r_\a^* e_\b) =
\xi\delta^{-1}\nu(e_\b^* e_\a e_\b).$$ Indeed, these equations hold for
$\a=\a_2$, respectively $\a=\a_2$ and $\b=\a_3$,
and follow from these case by conjugation with a suitable Weyl group element
(observe that the $\xi$ factors of $\nu(r_i e_\a)$ and $\nu(e_\a r_i)$
agree for every node $i$, so for each $w\in W$ with $w\a_1 = \a$,
we have $\nu(e_\a) = \nu(w)\nu(e_1)\nu(w^{-1})$).
\end{Remarks}

\begin{Lm}
\label{lm:AfterBMW}
For $i,j,k\in\{1,\ldots,n\}$, denote by $f_{ij}^k$ the decorated $n$-connector
all of whose pairs are of the form $\{x,\bar x\}$ except for $x\in \{i,j\}$,
where the pairs are $\{i,j\}$ and $\{\bar i, \bar j\}$, and of which only the
top horizontal pair $\{i,j\}$ and the vertical pair $\{k,\bar k\}$
are labeled $1$. Similarly, let $g_{ij}^k$ be as $f_{ij}^k$ but with
the bottom horizontal pair $\{\bar i,\bar j\}$
labeled $1$ instead of $\{i,j\}$.
Then, for $i<j$ and $k\not\in\{i,j\}$,
\begin{eqnarray*}
\nu(r_\a r_\a^* e_\b) &=& 
\begin{cases}
f_{ij}^k&\mbox{ if } \a=\eps_k-\eps_i\mbox{ and }\b=\eps_j-\eps_i\\
\xi\delta^{-1}f_{ij}^k&\mbox{ if } 
\a=\eps_k-\eps_j \mbox{ and } \b=\eps_j-\eps_i\\
g_{ij}^k&\mbox{ if } \a=\eps_k-\eps_i \mbox{ and } \b=\eps_j+\eps_i\\
\xi\delta^{-1}g_{ij}^k&\mbox{ if } 
\a=\eps_k-\eps_j \mbox{ and } \b=\eps_j+\eps_i\\
\end{cases}
\end{eqnarray*}
\end{Lm}

\begin{proof}
For $i=1$, $j=2$, $k=3$, the formulas follow from straightforward
computations. For instance, for the 
second line, with $\a = \eps_3-\eps_1$, $\b = \eps_3-\eps_2$,
we find
$r_\a r_\a^* e_\b = (r_3r_2r_3)(r_3r_1r_3) e_3 = 
r_3r_2r_1e_3 =  r_2r_1e_3 r_3^* $ and so
$\nu(r_\a r_\a^* e_\b) = \xi\delta^{-1}\nu(r_2r_1e_3) = 
\xi\delta^{-1} f_{23}^1$ by the definition of $\nu$.

For general $i<j$, the proof can be finished by induction on the sum of the
distances between $i$, $j$, and $k$.
\end{proof}

\begin{Prop}
\label{prop:AfterBMW}
The $\Z[\delta^{\pm1}]$-algebra $\BrD(\D_n)$ is well defined and the linear
map $\nu: \Br(\D_n)\to \BrD(\D_n)$ is an isomorphism.
\end{Prop}

\begin{proof}
By reversing the above definition of $\nu(a)$, we find that each Brauer
diagram is of the form $\delta^k\nu (a)$ for some $k\in\Z$ and some monomial
$a$ in $\Br(\D_n)$, so $\nu$ is surjective.  As a $\Z[\delta^{\pm1}]$-module,
the Brauer diagram algebra $\BrD(\D_n)$ is free with basis $T_n\cup \xi
T_n^=\cup \theta (T^0_n\cap T^=_n)$. Its dimension is $|T_n| + |T^=_n| +
|T^0_n \cap T^=_n|$.  By Lemma~\ref{ch5Fn} this number equals $d(n)$, the
dimension of the free $\Z[\delta^{\pm1}]$-algebra $\Br(\D_n)$,
cf. \cite[Theorem 1.1]{CFW}, and so $\nu$ is bijective.

There is a unique multiplication on $\BrD(\D_n)$ such that $\nu$ is an
isomorphism of $\Z[\delta^{\pm1}]$-algebras. We claim that this multiplication
satisfies the rules of Definition \ref{df:BrDalg}.  For Steps (i) and (ii), as
well as the $\delta$ part of (iv), the rules coincide with the laws known for
$\A_{n-1}$ and settle the multiplication up to decorations and the coefficient
in $\Lambda$. To substantiate the claim for these two steps, consider the
$\Z[\delta^{\pm1}]$-algebra homomorphism $\pi$ on $\Br(\D_n)$ determined by
$\pi(r_1) = r_2$, $\pi(e_1) = e_2$ and fixing the other generators.  The
homomorphism $\pi$ projects $\BrD(\D_n)$ onto its subalgebra generated by
$r_i$ and $e_i$ for $i>1$. By inspection of the defining relations, it is
clear that this subalgebra is a homomorphic image of $\Br(\A_{n-1})$ (with a
shift of $1$ in the indices of the usual generators). The decorated
$n$-connectors that are $\nu$-images of monomials from this subalgebra are
precisely those belonging to the usual Brauer diagram algebra of undecorated
$n$-connectors (appearing here as the $n$-connectors all of whose labels are
$0$, cf.~Remark \ref{rmks:afternu}(iii)). Therefore, the restriction of $\nu$
to the image of $\pi$ is surjective onto the Brauer diagram algebra
$\BrD(\A_{n-1})$, and so the image of $\pi$ is a subalgebra of $\Br(\D_n)$
isomorphic to $\Br(\A_{n-1})$. In terms of Brauer diagrams, $\nu\circ \pi
\circ \nu^{-1}$ sends a Brauer diagram $\lambda f$, where $\lambda\in \Lambda$
and $f$ is a decorated $n$-connector, to $\pi(\lambda) \pi(f)$, where $\pi(f)$
is obtained from $f$ by removing the decorations, and $\pi(\lambda)$ is
determined by $\pi(\delta) =\pi(\xi)=\delta$ and $\pi(\theta) =\delta^2$.  In
this light, all that is needed to verify is that, for monomials $a$, $b\in
\BrM(\D_n)$, the decoration of the decorated $n$-connector $f$ and the
coefficient $\lambda$ of the result $ \lambda f$ of the multiplication of
$\nu(a)$ and $\nu(b)$ according to Definition \ref{df:BrDalg} coincide with
those of $\nu(ab)$. 

The homomorphism $\pi$ also shows that the left rule of
Figure \ref{ch5brauerbasica} applies: a label $1$ on a vertical strand
from top node $h$ to bottom node $\bar k$ means that
the central part (the element
$uzv$ in the notation of Definition \ref{df:nu}) of the
corresponding monomial $ ue_Xzv$
maps $k$ to $-h$, and so two consecutive
labels $1$ on the same strand cancel.

For Step (iii) of Definition \ref{ch5brauerxi3big}, we need to verify the
rules of Figure \ref{ch5brauerxi3big}.  The four rules on the two bottom lines
are consequences of the others: taking the vertical pairs into account, we can
view the left hand sides of each of these as a $U$-shaped triple of line
segments followed or preceeded by a similar triple upside down whose
neighboring vertical pair is identified with the neighboring vertical pair
of the first triple.  By the rule for the $U$-shaped triple, read backwards,
and next for the second triple, read forward, the presence of a
$\xi\delta^{-1}$ factor can be determined. See Figure \ref{fig:nobottoms} for
an example.

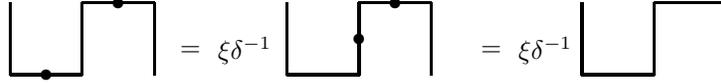
\begin{figure}[htbp]
\unitlength .8mm
\begin{picture}(122,20)(0,0)
\linethickness{0.3mm}

\put(0,4){\line(0,1){12}}

\put(0,4){\line(1,0){12}}

\put(12,4){\line(0,1){12}}

\put(6,4){\circle*{1.71}}

\put(12,16){\line(1,0){12}}

\put(18,16){\circle*{1.71}}

\put(24,4){\line(0,1){12}}

\put(30,8){\makebox(0,0)[cc]{$=$}}

\put(39,8){\makebox(0,0)[cc]{$\xi\delta^{-1}$}}

\put(46,4){\line(0,1){12}}

\put(46,4){\line(1,0){12}}

\put(58,4){\line(0,1){12}}

\put(64,16){\circle*{1.71}}

\put(58,16){\line(1,0){12}}

\put(58,10){\circle*{1.71}}

\put(70,4){\line(0,1){12}}

\put(80,8){\makebox(0,0)[cc]{$=$}}

\put(89,8){\makebox(0,0)[cc]{$\xi\delta^{-1}$}}

\put(95,4){\line(0,1){12}}

\put(95,4){\line(1,0){12}}

\put(107,4){\line(0,1){12}}

\put(107,16){\line(1,0){12}}

\put(119,4){\line(0,1){12}}

\end{picture}
\caption{Second rule at the left from below of Figure \ref{ch5brauerxi3big} 
obtained as a consequence of two earlier rules}
\label{fig:nobottoms}
\end{figure}

If $a$ and $b$ are outside the ideal generated by the $e_i$, they belong to
$W(\D_n)$ up to multiples from $\la\delta^{\pm1}\ra$; in this case $\nu(a)$
and $\nu(b)$ only have vertical pairs and the decoration rules of Definition
\ref{df:BrDalg} are easily seen to hold.  Also, when no bottom ends of a
crossing of two vertical pairs of $\nu(a)$ match the ends of a top horizontal
pair of $\nu(b)$, no changes in the coefficient $\lambda$ occurs. We will use
these observations to reduce the configurations to be considered.

Both the rules
of the figure and the relations defining $\Br(\D_n)$ are closed under
opposition, i.e., reflecting the picture in a horizontal mirror and reversing
an expression for a word in the generators of $\Br(\D_n)$. Therefore, it
suffices to consider the rules of Figure \ref{ch5brauerxi3big} in which the
horizontal strands belong to $\nu(b)$ and the two vertical strands belong to
$\nu(a)$.

We now show that the rules whose vertical pairs (belonging to $\nu(a)$)
cross, follow from those without crossings. The latter will be called {\em $U$
rules}, a name reminding us of their shapes. Take any left hand side of an
equality with the horizontal pair in $\nu(b)$ and crossing vertical pairs in
$\nu(a)$, insert straight line segments above the top corners and move the
decoration, if present, on the horizontal pair towards the vertical pair
ending at the left (this is a $U$ rule); next move all decorations up towards
the top vertical line segments by the rules for $W(\D_n)$. At the bottom, we
have a triple of undecorated pairs; now the rules for undecorated pairs are
those for the Brauer diagram algebra of type $\A_{n-1}$, and allow us to
replace the triple of line segments by a single horizontal pair, and we can
finish by the applying $U$ rules. See Figure \ref{ch5asso} for an example.

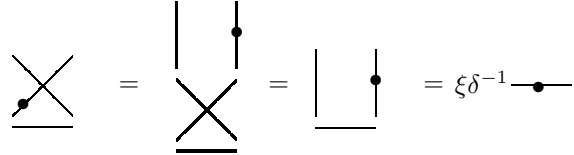
\begin{figure}[htbp]
\unitlength 0.8mm
\begin{picture}(109,25)(0,0)
\put(22.12,11){\makebox(0,0)[cc]{$=$}}
\multiput(2.74,6)(0.08,0.08){125}{\line(1,0){0.08}}
\multiput(2.74,16)(0.08,-0.08){125}{\line(1,0){0.08}}
\put(2.74,4.12){\line(1,0){10}}
\put(4.51,8){\circle*{1.7}}
\put(63.15,12){\circle*{1.7}}
\put(63.15,6){\line(0,1){11}}
\put(53.15,6){\line(0,1){11}}
\put(53,4){\line(1,0){10}}
\put(47,11){\makebox(0,0)[cc]{$=$}}
\put(85.59,11.15){\line(1,0){10}}
\put(90.18,10.92){\circle*{1.7}}
\put(80.71,11.27){\makebox(0,0)[cc]{$\xi\delta^{-1}$}}
\linethickness{0.3mm}
\put(72.71,11){\makebox(0,0)[cc]{$=$}}
\put(40,20){\circle*{1.7}}
\put(40,14){\line(0,1){11}}
\put(30,14){\line(0,1){11}}
\multiput(30,2)(0.08,0.08){125}{\line(1,0){0.08}}
\multiput(30,12)(0.08,-0.08){125}{\line(1,0){0.08}}
\put(30,0.12){\line(1,0){10}}
\end{picture}
\caption{Derivation of a multiplication rule with a crossing} \label{ch5asso}
\end{figure}

Next we verify the third and fourth rule from the top of Figure
\ref{ch5brauerxi3big}.  Suppose that the left hand side of a $U$ rule in one
of these two rows has left bottom corner at node $\bar i$ and right bottom
corner at node $\bar j$, so that $i<j$.  Then, by pre-multiplications with
suitable reflections $r_\c$ for $\c=\eps_k-\eps_i$ and $\c=\eps_j-\eps_k$, the
vertical pairs of the triple above $\bar i$ and $\bar j$ can be made straight,
that is, equal to $\{i,\bar i\}$ and $\{j,\bar j\}$, respectively, without
changing the $\xi$ factors. Set $\b=\eps_j -\eps_i$ or $\b=\eps_j +\eps_i$
according to whether the horizontal pair in the figure has label $0$ or $1$.

First consider the left hand side of the equation in rows 3 and 4 in the right
column of Figure \ref{ch5brauerxi3big}. Here $a =a' r_\b r_\b^*$ and $b = e_\b
b'$ for certain monomials $a'$, $b'$ in $\BrM(\D_n)$ and $\b = \eps_j\pm
\eps_i$.  Now $\nu(ab) = \nu(a') \nu(r_\b r_\b^* e_\b)\nu(b') = \nu(a')
\nu(r_\b^* e_\b)\nu(b') = \xi\delta^{-1} \nu(a') \nu(e_\b)\nu(b') =
\xi\delta^{-1} \nu(a') \nu( b) $, which
proves the rules on rows 3 and 4 in the right column.

To deal with the third and fourth rule from the top left of Figure
\ref{ch5brauerxi3big}, consider the left hand side of a $U$ rule with vertical
pairs $\{i,\bar i\}$ and $\{j,\bar j\}$, where $i<j$ as before, and suppose
that $\{j,\bar j\}$ is the only vertical pair of the triple with label $1$.
As only one decoration in $\nu(a)$ is visible and there are an even number,
there must be another decorated pair.  Without loss of generality, the second
decoration can be taken to be on a vertical pair with bottom end point, say
$\bar k$ distinct from $\bar i$ and $\bar j$.  For, if there would be a bottom
horizontal pair with a bottom end node $\bar k$, the monomial $a$ could be
rewritten so as to end with the monomial $a_2=r_\c r_\c^*$ with
$\c=\pm(\eps_k-\eps_j)$, so that $\nu(a_2)$ is the identity except that the
vertical pairs $\{j,\bar j\}$ and $\{k,\bar k\}$ are labeled $1$; as the pair
with label $1$ of $\nu(a_2)$ ending in $\bar k$ does not change and neither
does its label, the multiplication rule will be as required for $a$ once it
holds for $a_2$ instead of $a$.  Now let $\b = \eps_j\pm\eps_i$ as before and
set $\a = \eps_j-\eps_k$.  There are monomials $a'$, $b'\in\BrM(\D_n)$ such
that $a = a'r_\a r_\a^* e_\b$ and $b = e_\b b'$. By Lemma \ref{lm:AfterBMW},
$\nu(ab) = \nu (a') \nu(r_\a^* r_\a e_\b) \nu(b') = \xi\delta^{-1} \nu(a')
f_{ij}^k \nu(b')$ which establishes the remaining rules of Figure
\ref{ch5brauerxi3big}. Those $U$ rules which are not present (because they
produce no factor $\xi\delta^{-1}$), can be treated in the same way. This ends
the verification of Step (iii).

The parts of Step (iv) not involving $\theta$ are as before or familiar from
the case $\A_{n-1}$.  The simplest instance of the $\theta$ related rule
occurs for $a = e_1$ and $b = e_2$, when $\theta\delta^{-1} \nu(e_1) =
\nu(e_1e_2)$.  By conjugation with Weyl group elements, it generalizes to
$\theta\delta^{-1} \nu(e_\b) = \nu(e_\b e_\b^*)$ for every positive root $\b$,
and, after multiplication with suitable reflections, this accounts for the
last part of (iv).

As for Step (v), by definition of $\nu$, the $n$-connector of $\nu(a)$ has no
decorations in the presence of a scalar factor $\theta$. The rule is in
accordance with $\theta\nu(e_\b^*)  = \theta \nu(e_\b)$.

We conclude that the rules used in Definition \ref{df:BrDalg} are satisfied by
the multiplication turning $\nu$ into an isomorphism.  This establishes that
the latter is well defined.
\end{proof}

\begin{Remarks}\label{rmks:afterProp46}
\rm (i). The inverse of $\nu$ can be described effectively.  Given the Brauer
diagram $\lambda f$ of type $\D_n$ with decorated $n$-connector $f$ and
$\lambda\in\{1,\xi,\theta\}$, we describe how to find a monomial $a = ue_Xzv
\in\Br(\D_n)$ such that $\nu(a)\in\la\delta^{\pm1}\ra \lambda f$.  As
$\theta\nu(ue_Xzv) = \delta^{2-|X|}\nu(ue_{X'}zv)$ where $X' = \{\b,\b^*\mid
\b\in X\}$ if $X$ has no mutually orthogonal mates, and as $\xi\nu(ue_Xzv) =
\delta\nu(ue_Xr_n^*z v)$ whenever $\b\in X$ and $X$ is equal to some $Y(t)$ or
to $Y'(n/2)$, it suffices to consider the case where $\lambda = 1$.  Observe
that $a\emptyset = uX$ and $a^{\op} \emptyset = v^{-1}X$ can be read off from
the horizontal strands at the top and bottom of $f$, respectively.  As
$u,v^{-1}\in D_X$ and $D_X$ is fixed, this determines $u$ and $v$ uniquely.
Finally, $z$ is determined by the vertical strands of
$\nu(u^{-1})f\nu(v^{-1})$.

\nl(ii).
\label{rmk:SpecialBrauerDiagrams}
\rm Two special cases of Brauer diagrams of type $\D_n$ are known. First those
without horizontal pairs, see Remark \ref{rmks:afternu}(ii).  Second, in
\cite{Gre} diagrams for the Temperley-Lieb algebra of type $\D_n$ are
considered. As discussed in \cite[Remark 5.2]{CGWBMW}, the Temperley-Lieb
algebra is a subalgebra of $\Br(\D_n)$, and by the fact that $\nu$ is an
isomorphism, also of $\BrD(\D_n)$.  The diagrams appearing in \cite{Gre} are
precisely the Brauer diagrams without crossings (take into account that $\xi$
also contains a crossing and so does not appear), that is, these are
compositions of the $\nu(e_i)$.
\end{Remarks}

\np Recall the definition of $G_i$ and $E_i$ from Figures \ref{picsGiEi} and
\ref{ch5G0E0}. Write $G_i^{-1}$ for the same tangle as $G_i$ but with the
over crossing changed into an under crossing.  There is a natural map assigning
to a tangle in $\U_n$ that is a composition of the basic tangles $G_i$,
$G_i^{-1}$, and $E_i$, a Brauer diagram of type $\D_n$.  For such a $t\in
\U_n$, let $\psi(t)$ be the Brauer diagram of type $\D_n$ obtained by
identifying over and under crossings and by replacing twists around the pole of
a strand by addition of $1\mod 2$ (that is, applying a change of
decoration) to the strand. This means the element
$\psi(t)$ is the product in $\BrD(\D_n)$ of the Brauer diagrams arising by
substitution of $\nu(r_i)$ for each $G_i$ and for each $G_i^{-1}$ and of
$\nu(e_i)$ for each $E_i$ occurring in an expression of $t$ as a composition
of basic tangles. By construction, $\psi$ is a homomorphism of monoids.

The identification of over and under crossings means that the scalars
$\Xi^{\pm}$ and $\Theta$ are replaced by the scalars $\xi$ and $\theta$,
respectively.  Under $\psi$, the left and right hand sides of Figure
\ref{ch5xidef} become the rule at the right top of Figure
\ref{ch5brauerxi3big} but for the isolated horizontal strand at the bottom.

\begin{Prop} \label{homomKTDn2CDn}
The map $\psi$ induces an $R$-equivariant homomorphism of rings $ \KT(\D_n)\to
\BrD(\D_n)$, also denoted by $\psi$. It satisfies $\psi\circ\phi =
\nu\circ\mu$.
\end{Prop}

\begin{proof}
Denote by $\U_n^{(0)}$ the set of $(n,n)$-tangles in $\U_n$ that are a
composition of the basic tangles $G_i$, $G_i^{-1}$, and $E_i$.  As
$R/(\{l-1,m\})R \cong \Z[\delta^{\pm1}]$, the map $\psi$ can be linearly
extended to a map $R[\U_n^{(0)}]\to\BrD(\D_n)$ with $l-1$ and $m$ in its
kernel.

We claim that $\psi$ factors through $\KT(\D_n)$ because the relations of
Definition \ref{df:KTDn} hold for the tangles replaced by their images in
$\BrD(\D_n)$.  The double twist relation is equivalent with the rule in
$\BrD(\D_n)$ that two decorations on one strand cancel each other.  The
Kauffman skein relation (i) holds as by $m = 0$ the relation reduces to the
equality of over and under crossings.  The two partial diagrams satisfying the
commuting relation (ii) are both mapped to the same partial diagram containing
two decorated vertical strands.  The self-intersection relations (iii) hold in
$\BrD(\D_n)$ with $l = 1$.  The
idempotent relation (iv) carries over to the same relation in $\BrD(\D_n)$.
Both sides of the pole-related self-intersection relations (v) and (vi) are
mapped to the same multiple of a decorated $n$-connector by a scalar
$\xi\delta^{-1}$. (For (v), use the first and fifth rule in the left column of
Figure \ref{ch5brauerxi3big}; for (vi) use the second and the sixth from the
same column.)  The first closed pole loop relation (vii) is covered by
rewriting the images under $\psi$ of both sides by use of $\theta$.  Further
details are left to the reader.  It may be worthy of note that the
self-intersection relations (iii) only demand that $l^2 = 1$ and that the
choice $l=1$ corresponds to the complete removal of self-intersections as in
Reidemeister I. Therefore, the defining relations for $\KT(\D_n)$ are in the
kernel of $\psi$.  As $\psi$ is a homomorphism of monoids, even the ideal of
$\Z[\delta^{\pm1}][\U_n^{(0)}]$ generated by the defining relations for
$\KT(\D_n)$ is in the kernel of $\psi$, and so $\psi$ factors through
$\KT(\D_n)$, as claimed.

In view of Theorem \ref{th:surjhomo}, $\KT(\D_n)$ is linearly spanned by
monomials in $G_i$ and $E_i$ and so the map $\psi$ is defined on all of
$\KT(\D_n)$.  We conclude that $\psi$ is well defined as an $R$-equivariant
ring homomorphism $\KT(\D_n)\to\BrD(\D_n)$.

Now $\psi(G_i) = \nu(r_i)$ and $\psi(E_i) = \nu(e_i)$ so $\psi\phi(g_i) =
\nu\mu(g_i)$ and $\psi\phi(e_i) = \nu\mu(e_i)$ for $i=1,\ldots,n$.  In
particular, $\psi\circ\phi$ and $\nu\circ\mu$ are two $R$-equivariant ring
homomorphisms agreeing on the generators of $\BMW(\D_n)$, so they coincide.
As $\nu$ and $\mu$ are surjective, so is $\psi\circ\phi$. Consequently, $\psi$
is surjective.
\end{proof}

This proves Theorem \ref{th:main}(i) and (iv).  We are ready for the last main
result, which settles Theorem \ref{th:main}(iii).  Notice Theorem
\ref{th:main}(ii) is immediate from Proposition \ref{prop:AfterBMW}.

\begin{Thm}\label{th:MainIso}  The map $\phi: \BMW(\D_n)\to
\KT(\D_n)$ is an isomorphism of $R$-algebras.  Both algebras are free of
dimension
$d(n)$, the dimension of the Brauer algebra
of type $\D_n$.  The tensor products of these
algebras with $\Q(l,\delta)$  over $R$ are semisimple.
\end{Thm}

\begin{proof}
In view of surjectivity of $\psi$, the dimension of $\KT(\D_n)$ is at least
$d(n)$. By an argument similar to the one for \cite[Lemma 4.2]{CGWBMW} applied
to $\nu^{-1}\circ \psi$, we see that $\KT(\D_n)$ is a free $R$-module.  But
$\phi$ is surjective by Theorem \ref{th:surjhomo}, and $\BMW(\D_n)$ is free of
dimension $d(n)$ as well, by \cite[Theorem 1.1]{CGWBMW}.  Therefore, $\phi$ is
an isomorphism.  Also by \cite[Theorem 1.1]{CGWBMW}, $\BMW(\D_n)$, when
tensored with $\Q(l,\delta)$, is semi-simple. This implies the semisimplicity
statement.
\end{proof}

\begin{Remarks}\label{rmk:oldone}
\rm
(i).  Our work also proves the corresponding result for $\BMW(\A_{n-1})$
known from \cite{MorWas}:
\label{th:MainIsoAn}  
the BMW algebra $\BMW(\A_{n-1})$ is isomorphic to the tangle algebra
$\KT(\A_n)$ given via the map $\phi$ of Proposition \ref{homomKTDn2CDn}.  The
dimension is $n!!$, the same as the dimension of the Brauer algebra of type
$\A_{n-1}$.  The proof of these results obtainable from the arguments in the
proof of Theorem~\ref{th:MainIso} differs from the original one and is far
easier than the one for $\D_n$ as there is no $e_i^*$, no $\Xi^+$, and no
$\Theta$. 

\nl(ii).
\label{rmk:(2)}
In Definition \ref{ch5main} we encountered the $R$-algebra $\KT(\D_n)^{(2)}$
properly containing $\KT(\D_n)$.  There are corresponding versions of the
other algebras of Theorem \ref{th:main} and the maps such that the following
diagram commutes.
\begin{center}
\begin{tabular}{ccccccc}
$\BMW(\D_n)^{(2)}$&$ \overset{\displaystyle \mu^{(2)}}{\longrightarrow}$&$\Br(\D_n)^{(2)}$\\
$\phi^{(2)}\ \downarrow$&&$\downarrow\ \nu^{(2)}$\\
$\KT(\D_n)^{(2)}$&$\overset{\displaystyle \psi^{(2)}}{\longrightarrow}$&$\BrD(\D_n)^{(2)}$
\end{tabular}
\end{center}

Here $\BMW(\D_n)^{(2)}$ denotes the algebra over $R[\Xi^+,\Theta]$,
specified by the relations of Lemma \ref{ch5KT0}, that is generated by
$g_i$ and $e_i$ $(i=1,\ldots,n)$ subject to the relations known for
$\BMW(\D_n)$ from Definition \ref{BMW-def} and
\begin{eqnarray*}
\Xi^+ e_1 & = & \delta g_2e_1,\qquad\qquad
\Theta e_1  =  \delta e_1e_2.
\end{eqnarray*}
Similarly, $\Br(\D_n)^{(2)}$ denotes the algebra over $\Z[\Lambda]$ generated
by $r_i$ and $e_i$ $(i=1,\ldots,n)$ subject to the relations known for
$\Br(\D_n)$ in \cite[Table 1]{CFW} and
\begin{eqnarray*}
\xi e_1 & = & \delta r_2e_1,\qquad\qquad
\theta e_1  =  \delta e_2 =  \delta e_1e_2.
\end{eqnarray*}
The algebra $\BrD(\D_n)^{(2)}$ is linearly spanned over $\Z[\Lambda]$ by
$T_n$. It is not free over this ring, but free with basis $T_n\cup\xi T_n\cup
\theta T_n^0$ over $\Z[\delta^{\pm1}]$, and so its dimension is $(2^n+1)n!! $
if $n\ge1$.  
For $n=0$, the algebra
$\BrD(\D_n)^{(2)}$ coincides with $\Z[\Lambda]$ which is free of dimension 3
over $\Z[\delta^{\pm1}]$.

\label{lm:AfterBrauer}
The following relations hold in $\Br(\D_n)$ for all $\a$, $\b\in\Phi^+$.
\begin{eqnarray*}
r_{\a^*}e_\a e_\b &=& e_\a  r_{\b^*} e_\b, \qquad 
 e_{\a^*}e_\a e_\b = e_\a e_{\b^*} e_\b,
\end{eqnarray*}
They can be used to show that $\BrD(\D_n)$ embeds in $\BrD(\D_n)^{(2)}$, and
that $\BMW(\D_n)$ embeds in $\BMW(\D_n)^{(2)}$.  All maps in the diagram are
$R$-equivariant ring homomorphisms.  The vertical maps $\phi^{(2)}$ and
$\nu^{(2)}$ are algebra isomorphisms again.  The action $\sigma$ of
$\BrM(\D_n)$ on $\AO$, see \cite[Theorem 3.6(i)]{CFW} extends to an action of
$\BrM(\D_n)^{(2)}$ on $\AO$ with
\begin{eqnarray*}
\xi X & = & X,\ \qquad
\theta X  =  X\cup \{\b^*\mid\b\in X\}.
\end{eqnarray*}

\end{Remarks}

\end{document}